\documentclass[11pt,a4paper]{article} 

\usepackage[ruled,linesnumbered]{algorithm2e} 
\usepackage{amscd} 
\usepackage{amsfonts} 
\usepackage{amsmath} 
\usepackage{amssymb} 
\usepackage{amsthm} 
\usepackage{bm} 
\usepackage{booktabs} 
\usepackage{delarray} 
\usepackage{extarrows} 
\usepackage{geometry} 
\usepackage{graphicx} 
\usepackage{hyperref} 
\usepackage{longtable} 
\usepackage{tikz-cd} 
\usepackage[all]{xy} 
\usepackage[affil-it]{authblk} 

\usepackage{adjustbox}
\usepackage{booktabs}
\usepackage{caption}
\usepackage{color}
\usepackage{enumitem} 
\usepackage{float}
\usepackage{hyperref}
\usepackage{indentfirst} 
\usepackage{makecell} 
\usepackage{mathrsfs} 
\usepackage{mathtools} 
\usepackage{multirow} 
\usepackage{relsize} 
\usepackage{subfigure}

\setlist[enumerate]{itemsep=0pt,parsep=0pt}

\title{On the rank of the multivariable $(\varphi,\OK\x)$-modules associated to mod $p$ representations of $\GL_2(K)$}
\author{Yitong Wang\thanks{E-mail address: \texttt{yitongw.wang@utoronto.ca}}}
\date{}

\allowdisplaybreaks[4] 

\def\FF{\mathbb{F}}

\def\RR{\mathbb{R}}

\def\ZZ{\mathbb{Z}}

\def\NNN{\mathbb{Z}_{\geq0}}

\def\cC{\mathcal{C}}

\def\cJ{\mathcal{J}}

\def\cS{\mathcal{S}}

\def\fm{\mathfrak{m}}
\def\fp{\mathfrak{p}}

\def\ft{\mathfrak{t}}

\def\sh{\mathsf{h}}

\def\cont{\operatorname{cont}}

\def\deg{\operatorname{deg}} 
 
\def\det{\operatorname{det}}

\def\dim{\operatorname{dim}}

\def\Gal{\operatorname{Gal}}

\def\GL{\operatorname{GL}}
\def\gr{\operatorname{gr}}

\def\Hom{\operatorname{Hom}}
\def\id{\operatorname{id}}

\def\Ind{\operatorname{Ind}}
\def\Inj{\operatorname{Inj}}
\def\JH{\operatorname{JH}}

\def\M{\operatorname{M}}
\def\Mat{\operatorname{Mat}}
\def\max{\operatorname{max}}

\def\min{\operatorname{min}}
\def\mod{\operatorname{mod}}

\def\nss{\operatorname{nss}}

\def\rank{\operatorname{rank}} 
\def\reg{\operatorname{reg}}

\def\sh{\operatorname{sh}}

\def\soc{\operatorname{soc}}

\def\ss{\operatorname{ss}}

\def\Sym{\operatorname{Sym}}

\def\eqdef{\overset{\mathrm{def}}{=}}
\def\Fp{\mathbb{F}_p}

\def\Fq{\mathbb{F}_q}

\def\into{\hookrightarrow}
\def\inv{^{-1}}
\def\ism{\stackrel{\sim}{\rightarrow}}

\def\loc{\textit{loc.cit.}}

\def\OK{\mathcal{O}_K}
\def\onto{\twoheadrightarrow}

\def\Qp{\mathbb{Q}_p}
\def\Qpbar{\overline{\mathbb{Q}}_p}

\def\rhobar{\overline{\rho}}

\def\x{^{\times}}

\newcommand{\abs}[1]{|#1|}
\newcommand{\ang}[1]{\langle#1\rangle}
\newcommand{\babs}[1]{\left|#1\right|}
\newcommand{\bang}[1]{\left\langle#1\right\rangle}
\newcommand{\bbbra}[1]{\left[#1\right]}
\newcommand{\bbra}[1]{\left(#1\right)}
\newcommand{\bigabs}[1]{\big|#1\big|}
\newcommand{\bigang}[1]{\big\langle#1\big\rangle}
\newcommand{\bigbbbra}[1]{\big[#1\big]}
\newcommand{\bigbra}[1]{\big(#1\big)}
\newcommand{\bigset}[1]{\big\{#1\big\}}
\newcommand{\bra}[1]{(#1)}
\newcommand{\dbra}[1]{(\!(#1)\!)}
\newcommand{\ddbra}[1]{[\![#1]\!]}
\newcommand{\norm}[1]{\|#1\|}

\newcommand{\ovl}[1]{\overline{#1}}
\newcommand{\pmat}[1]{\begin{pmatrix}#1\end{pmatrix}}
\newcommand{\set}[1]{\{ #1 \}}
\newcommand{\smat}[1]{\left(\begin{smallmatrix}#1\end{smallmatrix}\right)}
\newcommand{\sset}[1]{\left\{ #1 \right\}}
\newcommand{\un}[1]{\underline{#1}}
\newcommand{\wh}[1]{\widehat{#1}}

\begin{document}

\newtheorem{definition}{Definition}[section] 
\newtheorem{remark}[definition]{Remark}
\newtheorem{example}[definition]{Example}
\newtheorem{proposition}[definition]{Proposition}
\newtheorem{lemma}[definition]{Lemma}
\newtheorem{corollary}[definition]{Corollary}
\newtheorem{theorem}[definition]{Theorem}
\newtheorem{conjecture}[definition]{Conjecture}

\maketitle

\begin{abstract}
    Let $p$ be a prime number, $K$ a finite unramified extension of $\Qp$ and $\FF$ a finite extension of $\Fp$. For $\pi$ an admissible smooth representation of $\GL_2(K)$ over $\FF$ satisfying certain multiplicity-one properties, we compute the rank of the associated \'etale $(\varphi,\OK\x)$-module $D_A(\pi)$ defined in \cite{BHHMS2}, extending the results of \cite{BHHMS2} and \cite{BHHMS3}. 
\end{abstract}

\tableofcontents

\section{Introduction}

Let $p$ be a prime number. The mod $p$ Langlands correspondence for $\GL_2(\Qp)$ is completely known by the work of Breuil, Colmez, Emerton, etc. In particular, Colmez (\cite{Col10}) constructed a functor from the category of admissible finite length mod $p$ representations of $\GL_2(\Qp)$ to the category of finite-dimensional continuous mod $p$ representations of $\Gal(\Qpbar/\Qp)$, using Fontaine's category of $(\varphi,\Gamma)$-modules (\cite{Fon90}) as an intermediate step. This gives a functorial way to realize the mod $p$ Langlands correspondence for $\GL_2(\Qp)$.

However, the situation becomes much more complicated when we consider $\GL_2(K)$ for $K$ a nontrivial finite extension of $\Qp$. For example, there are many more supersingular representations of $\GL_2(K)$ (\cite{BP12}) and we don't have a classification of these representations. Moreover, they are not of finite presentation (\cite{Sch15}, \cite{Wu21}), and it is impossible so far to write down explicitly one of these representations. 

Another important feature of the mod $p$ Langlands correspondence for $\GL_2(\Qp)$ is that it satisfies the local-global compatibility (\cite{Eme11}), in the sense that it can be realized in the $H^1$ of towers of modular curves. Motivated by this, we are particularly interested in the mod $p$ representations $\pi$ of $\GL_2(K)$ coming from the cohomology of towers of Shimura curves, see for example \cite[(1)]{BHHMS1}. We hope that these representations play a role in the mod $p$ Langlands correspondence for $\GL_2(K)$. 

There have been many results on the representation-theoretic properties of $\pi$ as above. For example, when $K$ is unramified over $\Qp$, we can explicitly describe the finite-dimensional invariant subspace $\pi^{K_1}$ of $\pi$ where $K_1\eqdef1+p\M_2(\OK)$ (\cite{Le19}), and we know that $\pi$ has Gelfand--Kirillov dimension $[K:\Qp]$ (\cite{BHHMS1}, \cite{HW22}, \cite{Wang1}). However, the complete understanding of $\pi$ still seems a long way off.

In \cite{BHHMS2}, Breuil-Herzig-Hu-Morra-Schraen constructed an exact functor $D_A$ from a nice subcategory of the category of admissible smooth mod $p$ representations of $\GL_2(K)$ (which contains $\pi$) to the category of multivariable $(\varphi,\OK\x)$-modules. This functor generalizes Colmez's functor. Then the key question is to determine the structure of $D_A(\pi)$ for $\pi$ as above. In this article, we answer this question further, extending the results of \cite{BHHMS2} and \cite{BHHMS3}. The results of this article can be used to deduce important properties of $\pi$.

\hspace{\fill}

To state the main result, we begin with the construction of the functor $D_A$. We let $K$ be a finite unramified extension of $\Qp$ of degree $f\geq1$ with ring of integers $\OK$ and residue field $\Fq$ (hence $q=p^f$). Let $\FF$ be a large enough finite extension of $\Fp$ and fix an embedding $\sigma_0:\Fq\into\FF$. We let $N_0\eqdef\smat{1&\OK\\0&1}\subseteq\GL_2(\OK)$. Then we have $\FF\ddbra{N_0}=\FF\ddbra{Y_0,\ldots,Y_{f-1}}$ with $Y_j\eqdef\sum\nolimits_{a\in\Fq\x}\sigma_0(a)^{-p^j}\smat{1&[a]\\0&1}\in\FF\ddbra{N_0}$ for $0\leq j\leq f-1$, where $[a]\in\OK\x$ is the Techm\"uller lift of $a\in\Fq\x$. We let $A$ be the completion of $\FF\ddbra{N_0}[1/(Y_0\cdots Y_{f-1})]$ with respect to the $(Y_0,\ldots,Y_{f-1})$-adic topology on $\FF\ddbra{N_0}$. There is an $\FF$-linear action of $\OK\x$ on $\FF\ddbra{N_0}$ given by multiplication on $N_0\cong\OK$, and an $\FF$-linear Frobenius $\varphi$ on $\FF\ddbra{N_0}$ given by multiplication by $p$ on $N_0\cong\OK$. They extend canonically by continuity to commuting continuous $\FF$-linear actions of $\varphi$ and $\OK\x$ on $A$. Then an \'etale $(\varphi,\OK\x)$-module over $A$ is by definition a finite free $A$-module endowed with a semi-linear Frobenius $\varphi$ and a commuting continuous semi-linear action of $\OK\x$ such that the image of $\varphi$ generates everything.

For $\pi$ an admissible smooth representation of $\GL_2(K)$ over $\FF$ with central character, we let $\pi^{\vee}$ be its $\FF$-linear dual, which is a finitely generated $\FF\ddbra{I_1}$-module and is endowed with the $\fm_{I_1}$-adic topology, where $I_1\eqdef\smat{1+p\OK&\OK\\p\OK&1+p\OK}\subseteq\GL_2(\OK)$ and $\fm_{I_1}$ is the maximal ideal of $\FF\ddbra{I_1}$. We define $D_A(\pi)$ to be the completion of $A\otimes_{\FF\ddbra{N_0}}\pi^{\vee}$ with respect to the tensor product topology. The $\OK\x$-action on $\pi^{\vee}$ given by $f\mapsto f\circ\smat{a&0\\0&1}$ (for $a\in\OK\x$) extends by continuity to $D_A(\pi)$, and the $\psi$-action on $\pi^{\vee}$ given by $f\mapsto f\circ\smat{p&0\\0&1}$ induces a continuous $A$-linear map
\begin{equation}\label{General Eq beta}
    \beta:D_A(\pi)\to A\otimes_{\varphi,A}D_A(\pi).
\end{equation}
Let $\cC$ be the abelian category of admissible smooth representations $\pi$ of $\GL_2(K)$ over $\FF$ with central characters such that $\gr{(D_A(\pi))}$ is a finitely generated $\gr(A)$-module. Then for $\pi$ in $\cC$, $D_A(\pi)$ is a finite free $A$-module (see \cite[Cor.~3.1.2.9]{BHHMS2} and \cite[Remark.~2.6.2]{BHHMS3}). If moreover $\beta$ is an isomorphism, then its inverse $\beta\inv=\id\otimes\varphi$ makes $D_A(\pi)$ an \'etale $(\varphi,\OK\x)$-module.

\hspace{\fill}

Let $\rhobar:\GL_2(K)\to\GL_2(\FF)$ be a continuous representation of the following form up to twist:

\begin{equation}\label{General Eq genericity}
    \rhobar|_{I_K}\cong\pmat{\omega_f^{\sum\nolimits_{j=0}^{f-1}(r_j+1)p^j}&*\\0&1}~\text{with}~2f+1\leq r_j\leq p-3-2f~\forall\,0\leq j\leq f-1,
\end{equation}
where $\omega_f:I_K\to\FF\x$ is the fundamental character of level $f$ (associated to $\sigma_0$). If $f=1$, we assume moreover that $r_0\geq4$. In particular, we have $p\geq4f+4$.
    
Let $\pi$ be a smooth representation of $\GL_2(K)$ over $\FF$ which satisfies
\begin{enumerate}
    \item 
    $\pi^{K_1}\cong D_0(\rhobar)$ as $K\x\GL_2(\OK)$-representations, where $D_0(\rhobar)$ is the representation of $\GL_2(\Fq)$ defined in \cite[\S13]{BP12} and is viewed as a representation of $\GL_2(\OK)$ by inflation, and $K\x$ acts on $D_0(\rhobar)$ by the character $\det(\rhobar)\omega\inv$, where $\omega$ is the mod $p$ cyclotomic character;
    \item 
    for any character $\chi:I\to\FF\x$ appearing in $\pi[\fm_{I_1}]=\pi^{I_1}$, we have $[\pi[\fm_{I_1}^3]:\chi]=1$, where $\pi[\fm_{I_1}^3]$ is the set of elements of $\pi$ annihilated by $\fm_{I_1}^3$, and $[\pi[\fm_{I_1}^3]:\chi]$ is the multiplicity of $\chi$ in the semisimplification of $\pi[\fm_{I_1}^3]$ as $I$-representations.
\end{enumerate}
In particular, (i) and (ii) are satisfied for those $\pi$ coming from the cohomology of towers of Shimura curves in a ``multiplicity-one'' situation (\cite{BHHMS1}, \cite{HW22}, \cite{Wang1}). Our main result is the following:

\begin{theorem}[\S\ref{General Sec basis}]\label{General Thm main1}
    Suppose that $\rhobar$ and $\pi$ are as above. Then $\pi$ is in $\cC$, $\beta$ in (\ref{General Eq beta}) is an isomorphism and $$\rank_AD_A(\pi)=2^f.$$
\end{theorem}

By \cite[Remark~3.3.2.6(ii)]{BHHMS2} we know that $\pi$ is in $\cC$. By \cite[Thm.~3.3.2.1]{BHHMS2} and localization we know that $\rank_AD_A(\pi)\leq2^f$. Theorem \ref{General Thm main1} is proved by \cite[Thm.~3.1.3]{BHHMS3} when $\rhobar$ is semisimple. We generalize the method of \cite{BHHMS3} to the non-semisimple case, which is seriously more delicate.

The proof of Theorem \ref{General Thm main1} is by an explicit construction of an $A$-basis of the dual \'etale $(\varphi,\OK\x)$-module $\Hom_A(D_A(\pi),A)$ in the following steps. 

\vspace{0.2cm}

\noindent\textbf{Step 1.} We construct $2^f$ projective systems $(x_{J,k})_{k\geq0}$ of elements of $\pi$ indexed by the subsets $J\subseteq\set{0,1,\ldots,f-1}$ (see Theorem \ref{General Thm seq}). To do this, we first define suitable elements $x_{J,k}\in\pi^{K_1}\cong D_0(\rhobar)$ for $k\leq f$. Since $D_0(\rhobar)$ is explicit, we can then study the precise relations among these elements, which enable us to define $x_{J,k}\in\pi$ for all $k$ inductively. This is the content of \S\S\ref{General Sec shift}-\ref{General Sec xJi} and Appendix \ref{General Sec app vanish}.

\vspace{0.2cm}

\noindent\textbf{Step 2.} We prove that $x_{J,k}\in\pi[\fm_{I_1}^{kf+1}]$ for all $J\subseteq\set{0,1,\ldots,f-1}$ and $k\geq0$ (see Proposition \ref{General Prop degree}, whose proof makes use of condition (ii) and is given in Appendix \ref{General Sec app degree}), from which we deduce that each projective system $x_J$ can be regarded as an element of $\Hom_{\FF}^{\cont}(D_A(\pi),\FF)$.

\vspace{0.2cm}

\noindent\textbf{Step 3.} There is a canonical $A$-linear injection (see \cite[(87)]{BHHMS3})
\begin{equation*}
    \mu_*:\Hom_A(D_A(\pi),A)\into\Hom_{\FF}^{\cont}(D_A(\pi),\FF).
\end{equation*}
We prove that each $x_J\in\Hom_{\FF}^{\cont}(D_A(\pi),\FF)$ satisfies a crucial finiteness condition (see Theorem \ref{General Thm finiteness}), which guarantees that it lies in the image of $\mu_*$. Once we prove that $x_J\in\Hom_A(D_A(\pi),A)$ for all $J$, it is not difficult to conclude that they form an $A$-basis of $\Hom_A(D_A(\pi),A)$.


\hspace{\fill}

Finally, we prove the following generalization of Theorem \ref{General Thm main1}. This result is crucially needed to prove that $\pi$ is of finite length (in the non-semisimple case) in \cite{BHHMS4}.

\begin{theorem}[Theorem \ref{General Thm rank sub}]\label{General Thm main2}
    Suppose that $\rhobar$ and $\pi$ are as above. Then for $\pi_1$ a subrepresentation of $\pi$, we have
    \begin{equation*}
        \rank_AD_A(\pi_1)=\babs{\JH\bra{\pi_1^{K_1}}\cap W(\rhobar^{\ss})},
    \end{equation*}
    where $\JH\bra{\pi_1^{K_1}}$ is the set of Jordan--H\"older factors of $\pi_1^{K_1}$ as a $\GL_2(\OK)$-representation, $\rhobar^{\ss}$ is the semisimplification of $\rhobar$, and $W(\rhobar^{\ss})$ is the set of Serre weights of $\rhobar^{\ss}$ defined in \cite[\S3]{BDJ10}.
\end{theorem}

\subsection*{Organization of the article}

In \S\S\ref{General Sec sw}-\ref{General Sec ps}, we review the notion of the extension graph and recall some results of \cite[\S2]{BP12} that are needed in the proof of Theorem \ref{General Thm main1} and Theorem \ref{General Thm main2}. In \S\S\ref{General Sec shift}-\ref{General Sec finiteness} and Appendix \ref{General Sec app vanish}, we explicitly construct some projective systems of elements of $\pi$ and study their basic properties. In particular, we prove the crucial finiteness condition in \S\ref{General Sec finiteness}. In \S\ref{General Sec basis} and Appendix \ref{General Sec app degree}, we use these projective systems to construct an explicit basis of $D_A(\pi)$ and finish the proof of Theorem \ref{General Thm main1}. In \S\ref{General Sec subrep}, we finish the proof of Theorem \ref{General Thm main2}. In Appendix \ref{General Sec app okx}, we compute the actions of $\varphi$ and $\OK\x$ on $D_A(\pi)$. Finally, in Appendix \ref{General Sec app lemmas}, we list some equalities among the various constants defined throughout this article.

\subsection*{Acknowledgements}

We thank Christophe Breuil for suggesting this problem, and thank Christophe Breuil and Ariane M\'ezard for helpful discussions and a careful reading of the earlier drafts of this paper. 

This work was supported by the Ecole Doctorale de Math\'ematiques Hadamard (EDMH).

\subsection*{Notation}

Let $p$ be a prime number. We fix an algebraic closure $\Qpbar$ of $\Qp$. Let $K\subseteq\Qpbar$ be the unramified extension of $\Qp$ of degree $f\geq1$ with ring of integers $\OK$ and residue field $\Fq$ (hence $q=p^f$). We denote by $G_K\eqdef\Gal(\Qpbar/K)$ the absolute Galois group of $K$ and $I_K\subseteq G_K$ the inertia subgroup. Let $\FF$ be a large enough finite extension of $\FF_p$. Fix an embedding $\sigma_0:\Fq\into\FF$ and let $\sigma_j\eqdef\sigma_0\circ\varphi^j$ for $j\in\ZZ$, where $\varphi:x\mapsto x^p$ is the arithmetic Frobenius on $\Fq$. We identify $\cJ\eqdef\Hom(\Fq,\FF)$ with $\set{0,1,\ldots,f-1}$, which is also identified with $\ZZ/f\ZZ$ so that the addition and subtraction in $\cJ$ are modulo $f$. For $a\in\OK$, we denote by $\ovl{a}\in\Fq$ its reduction modulo $p$. For $a\in\Fq$, we also view it as an element of $\FF$ via $\sigma_0$.

For $F$ a perfect ring of characteristic $p$, we denote by $W(F)$ the ring of Witt vectors of $F$. For $x\in F$, we denote by $[x]\in W(F)$ its Techm\"uller lift.

Let $I\eqdef\smat{\OK\x&\OK\\p\OK&\OK\x}$ be the Iwahori subgroup of $\GL_2(\OK)$, $I_1\eqdef\smat{1+p\OK&\OK\\p\OK&1+p\OK}$ be the pro-$p$ Iwahori subgroup, $K_1\eqdef1+p\M_2(\OK)$ be the first congruence subgroup, $N_0\eqdef\smat{1&\OK\\0&1}$ and $H\eqdef\smat{[\Fq\x]&0\\0&[\Fq\x]}$.

For $P$ a statement, we let $\delta_P\eqdef1$ if $P$ is true and $\delta_P\eqdef0$ otherwise.

Throughout this article, we let $\rhobar:G_K\to\GL_2(\FF)$ be as in \eqref{General Eq genericity} and $\pi$ be an admissible smooth representation of $\GL_2(K)$ over $\FF$ satisfying the conditions (i),(ii) before Theorem \ref{General Thm main1}.

\section{Combinatorics of Serre weights}\label{General Sec sw}

In this section, we review the notion of the extension graph following \cite{BHHMS1}.

We write $\un{i}$ for an element $(i_0,\ldots,i_{f-1})\in\ZZ^f$. For $a\in\ZZ$, we write $\un{a}\eqdef(a,\ldots,a)\in\ZZ^f$. For each $j\in\cJ$, we define $e_j\in\ZZ^f$ to be $1$ in the $j$-th coordinate, and $0$ otherwise. For $J\subseteq\cJ$, we define $\un{e}^J\in\ZZ^f$ by $e^J_j\eqdef\delta_{j\in J}$. We say that $\un{i}\leq\un{i}'$ if $i_j\leq i'_j$ for all $j$. We write
\begin{align*}
    X_1(\un{T})&\eqdef\sset{(\un{\lambda}_1,\un{\lambda}_2)\in\ZZ^{2f}:\un{0}\leq\un{\lambda}_1-\un{\lambda}_2\leq\un{p}-\un{1}};\\
    X_{\reg}(\un{T})&\eqdef\sset{(\un{\lambda}_1,\un{\lambda}_2)\in\ZZ^{2f}:\un{0}\leq\un{\lambda}_1-\un{\lambda}_2\leq\un{p}-\un{2}};\\
    X^0(\un{T})&\eqdef\sset{(\un{\lambda}_1,\un{\lambda}_2)\in\ZZ^{2f}:\un{\lambda}_1=\un{\lambda}_2}.
\end{align*}
We define the left shift $\delta:\ZZ^f\to\ZZ^f$ by $\delta(\un{i})_j\eqdef i_{j+1}$ and define $\pi:\ZZ^{2f}\to\ZZ^{2f}$ by $\pi(\un{\lambda}_1,\un{\lambda}_2)\eqdef\bigbra{\delta(\un{\lambda}_1),\delta(\un{\lambda}_2)}$. 

A \textbf{Serre weight} of $\GL_2(\Fq)$ is an isomorphism class of an absolutely irreducible representation of $\GL_2(\Fq)$ over $\FF$. For $\lambda=(\un{\lambda}_1,\un{\lambda}_2)\in X_1(\un{T})$, we define 
\begin{equation*}
    F(\lambda)\eqdef\bigotimes\limits_{j=0}^{f-1}\bbra{\!\bbra{\Sym^{\lambda_{1,j}-\lambda_{2,j}}\Fq^2\otimes_{\Fq}\det^{\lambda_{2,j}}}\otimes_{\Fq,\sigma_j}\FF}.
\end{equation*}
We also denote it by $(\un{\lambda}_1-\un{\lambda}_2)\otimes\det^{\un{\lambda}_2}$. This induces a bijection $$F:X_1(\un{T})/(p-\pi)X^0(\un{T})\ism\set{\text{Serre weights of }\GL_2(\Fq)}.$$ We say that a Serre weight $\sigma$ is \textbf{regular} if $\sigma\cong F(\lambda)$ with $\lambda\in X_{\reg}(\un{T})$. 

For $\lambda=(\un{\lambda}_1,\un{\lambda}_2)\in\ZZ^{2f}$, we define the character $\chi_{\lambda}:I\to\FF\x$ by $\smat{a&b\\pc&d}\mapsto(\ovl{a})^{\un{\lambda}_1}(\ovl{d})^{\un{\lambda}_2}$, where $a,d\in\OK\x$ and $b,c\in\OK$. Here, for $x\in\FF$ and $\un{i}\in\ZZ^f$ we write $x^{\un{i}}\eqdef x^{\sum\nolimits_{j=0}^{f-1}i_jp^j}$. In particular, if $\lambda\in X_1(\un{T})$, then $\chi_{\lambda}$ is the $I$-character acting on $F(\lambda)^{I_1}$. We still denote $\chi_{\lambda}$ for its restriction to $H$. For each $j\in\cJ$ we define $\alpha_j\eqdef(e_j,-e_j)\in\ZZ^{2f}$, and for each $\un{i}\in\ZZ^f$ we define $\alpha^{\un{i}}\eqdef\sum\nolimits_{j=0}^{f-1}i_j\alpha_j\in\ZZ^{2f}$. We also denote $\alpha_j$ and $\alpha^{\un{i}}$ the corresponding characters $\chi_{\alpha_j}$ and $\chi_{\alpha^{\un{i}}}$ when there is no possible confusion. Concretely, we have $\alpha^{\un{i}}\bbra{\!\smat{a&b\\pc&d}\!}=\bigbra{\ovl{a}\ovl{d}\inv}^{\sum\nolimits_{j=0}^{f-1}i_jp^j}$.

For $\mu=(\un{\mu}_1,\un{\mu}_2)\in\ZZ^{2f}$, we define the extension graph associated to $\mu$ by
\begin{equation}\label{General Eq extension graph}
    \Lambda_W^{\mu}\eqdef\sset{\un{b}\in\ZZ^f:0\leq\un{\mu}_1-\un{\mu}_2+\un{b}\leq\un{p}-\un{2}}.
\end{equation}
As in \cite[p.16]{BHHMS1}, there is a map
\begin{equation*}
    \ft_{\mu}:\Lambda_W^{\mu}\to X_{\reg}(\un{T})/(p-\pi)X^0(\un{T}),
\end{equation*}
such that the map $\un{b}\mapsto F(\ft_{\mu}(\un{b}))$ gives a bijection between $\Lambda_W^{\mu}$ and the set of regular Serre weights of $\GL_2(\Fq)$ with central character $\chi_{\mu}|_Z$, where $Z\cong\Fq\x$ is the center of $\GL_2(\Fq)$.

We let $\mu_{\un{r}}\eqdef(\un{r},\un{0})\in\ZZ^{2f}$ with $\un{r}=(r_0,\ldots,r_{f-1})$ and $r_j$ as in \eqref{General Eq genericity}. For $\un{b}\in\ZZ^f$ such that $-\un{r}\leq\un{b}\leq\un{p}-\un{2}-\un{r}$, we denote $\sigma_{\un{b}}\eqdef F\bigbra{\ft_{\mu_{\un{r}}}(\un{b})}$. For $\rhobar$ as in \eqref{General Eq genericity}, we let $J_{\rhobar}\subseteq\cJ$ be as in \cite[(17)]{Bre14}. Then by \cite[Prop.~A.3]{Bre14} and \cite[(14)]{BHHMS1} we have
\begin{equation}\label{General Eq SW of rhobar}
    W(\rhobar)=\sset{\sigma_{\un{b}}:~\begin{array}{ll}b_j=0&\text{if}~j\notin J_{\rhobar}\\b_j\in\set{0,1}&\text{if}~j\in J_{\rhobar}\end{array}}.
\end{equation}
In particular, $\rhobar$ is semisimple if and only if $J_{\rhobar}=\cJ$. For each $J\subseteq\cJ$, we define $\sigma_J\eqdef\sigma_{\un{a}^J}$ with
\begin{equation}\label{General Eq aJ}
    a^J_j\eqdef
    \begin{cases}
        0&\text{if}~j\notin J\\
        1&\text{if}~j\in J,~j+1\notin J~\text{or}~j\in J,~j+1\in J,~j\in J_{\rhobar}\\
        -1&\text{if}~j\in J,~j+1\in J,~j\notin J_{\rhobar}.
    \end{cases}
\end{equation}
In particular, for $J\subseteq J_{\rhobar}$ we have $\sigma_J=\sigma_{\un{e}^J}$. Then as a special case of \cite[(14)]{BHHMS1}, we have $\sigma_J=(\un{s}^J)\otimes\det^{\un{t}^J}$ with
\begin{align}
    \label{General Eq sJ}s^J_j&\eqdef
    \begin{cases}
        r_j&\text{if}~j\notin J,~j+1\notin J\\
        r_j+1&\text{if}~j\in J,~j+1\notin J\\
        p-2-r_j&\text{if}~j\notin J,~j+1\in J\\
        p-1-r_j&\text{if}~j\in J,~j+1\in J,~j\notin J_{\rhobar}\\
        p-3-r_j&\text{if}~j\in J,~j+1\in J,~j\in J_{\rhobar};
    \end{cases}\\
    \label{General Eq tJ}t^J_j&\eqdef
    \begin{cases}
        0&\text{if}~j\notin J,~j+1\notin J\\
        -1&\text{if}~j\in J,~j+1\notin J\\
        r_j+1&\text{if}~j\notin J,~j+1\in J~\text{or}~j\in J,~j+1\in J,~j\in J_{\rhobar}\\
        r_j&\text{if}~j\in J,~j+1\in J,~j\notin J_{\rhobar}.
    \end{cases}
\end{align}
We let $\chi_J\eqdef\chi_{\lambda_J}$ with $\lambda_J\eqdef(\un{s}^J+\un{t}^J,\un{t}^J)$. Then $\chi_J$ is the $I$-character acting on $\sigma_J^{I_1}$. For each $I$-character $\chi$, we denote by $\chi^s$ its conjugation by the matrix $\smat{0&1\\p&0}$. 

For $J\subseteq\cJ$ and $k\in\ZZ$, we define $J+k\eqdef\set{j+k:j\in J}$. Then we define the \textbf{semisimple part} of $J$, the \textbf{non-semisimple part} of $J$ and the \textbf{shifting index} of $J$ to be respectively
\begin{equation}\label{General Eq Jsh}
    J^{\ss}\eqdef J\cap J_{\rhobar},\quad J^{\nss}\eqdef J\setminus J_{\rhobar}=J\setminus J^{\ss},\quad J^{\sh}\eqdef J\cap(J-1)\cap J_{\rhobar}\subseteq J^{\ss}.
\end{equation}
In particular, if $\rhobar$ is semisimple, then $J^{\ss}=J$ and $J^{\nss}=\emptyset$ for all $J\subseteq\cJ$. By (\ref{General Eq genericity}), we have from (\ref{General Eq sJ}) and (\ref{General Eq Jsh})
\begin{equation}\label{General Eq bound s}
    2(f-\delta_{j\in J^{\sh}})+1\leq2(f-\delta_{j\in J^{\sh}})+1+\delta_{f=1}\leq s^J_j\leq p-2-2(f+\delta_{j\in J^{\sh}})~\forall\,j\in\cJ.
\end{equation}

\begin{lemma}\label{General Lem change origin}
    Let $J\subseteq\cJ$ and $\un{b}\in\ZZ^f$ such that $-\bigbra{2(\un{f}-\un{e}^{J^{\sh}})+\un{1}}\leq\un{b}\leq2(\un{f}+\un{e}^{J^{\sh}})$. Then we have $F\bigbra{\ft_{\lambda_J}(\un{b})}=\sigma_{\un{a}}$ with $a_j=(-1)^{\delta_{j+1\in J}}(b_j+\delta_{j\in J})+2\delta_{j\in J^{\sh}}$ for all $j\in\cJ$. In particular,
    \begin{enumerate}
    \item 
    we have $\sigma_{J^{\ss}}=F\bigbra{\ft_{\lambda_J}(-\un{b})}$ with $b_j=\delta_{j\in J^{\nss}}$ for all $j\in\cJ$;
    \item 
    we have $\sigma_{(J-1)^{\ss}}=F\bigbra{\ft_{\lambda_J}(-\un{b})}$ with $b_j=\delta_{j\in J\Delta(J-1)^{\ss}}$ for all $j\in\cJ$;
    \item 
    for each $J'\subseteq\cJ$, we have $\sigma_{J'}=F\bigbra{\ft_{\lambda_J}(-\un{b})}$ with $b_j=\delta_{j\in J}+\delta_{j\in J'}(-1)^{\delta_{j+1\notin J\Delta J'}}$ if $j\notin J_{\rhobar}$, and $b_j=\bigbra{\delta_{j\in J}-\delta_{j\notin J'}}(-1)^{\delta_{j+1\in J}}$ if $j\in J_{\rhobar}$.
    \end{enumerate}
    Here we recall that $J\Delta J'\eqdef (J\setminus J')\sqcup(J'\setminus J)$.
\end{lemma}

\begin{proof}
    The assumption on $\un{b}$ implies that $F(\ft_{\lambda_J}(\un{b}))$ is well-defined. By (\ref{General Eq aJ}) and a case-by-case examination we have
    \begin{equation*}
        a^J_j=\delta_{j\in J}(-1)^{\delta_{j+1\in J}}+2\delta_{j\in J^{\sh}}~\forall\,j\in\cJ.
    \end{equation*}
    Then by \cite[Lemma~2.4.4]{BHHMS1} applied to $\mu=\mu_{\un{r}}$ and $\omega=\un{b}$, we deduce that
    \begin{equation*}
    \begin{aligned}
        a_j\!=\!a^J_j\!+\!(-1)^{a^J_{j+1}}b_j\!=\!\bigbra{\delta_{j\in J}(-1)^{\delta_{j+1\in J}}\!+\!2\delta_{j\in J^{\sh}}}\!+\!(-1)^{\delta_{j+1\in J}}b_j\!=\!(-1)^{\delta_{j+1\in J}}(b_j\!+\!\delta_{j\in J})\!+\!2\delta_{j\in J^{\sh}}.
    \end{aligned}
    \end{equation*}
    The assertions (i),(ii),(iii) then follow easily whose proofs are left as an exercise.
\end{proof}

\section{The principal series}\label{General Sec ps}

In this section, we recall some results of \cite[\S2]{BP12}. 

For $j\in\cJ$, we define
\begin{equation*}
    Y_j\eqdef\sum\limits_{a\in\Fq\x}a^{-p^j}\pmat{1&[a]\\0&1}\in\FF\ddbra{N_0}.
\end{equation*}
Then we have $\FF\ddbra{N_0}=\FF\ddbra{Y_0,\ldots,Y_{f-1}}$. For $\un{i}=(i_0,\ldots,i_{f-1})\in\ZZ^f$, we set $\norm{\un{i}}\eqdef\sum\nolimits_{j=0}^{f-1}i_j$ and write $\un{Y}^{\un{i}}$ for $\prod\nolimits_{j=0}^{f-1}Y_j^{i_j}$. We recall the following results of \cite[Lemma~3.2.2.1]{BHHMS2}.

\begin{lemma}\label{General Lem Yj}
    For $j\in\cJ$ and $\mu_1,\mu_2\in\Fq\x$, we have
    \begin{enumerate}
    \item 
    $Y_j^p\smat{p&0\\0&1}=\smat{p&0\\0&1}Y_{j+1}$;
    \item 
    $\smat{[\mu_1]&0\\0&[\mu_2]}Y_j=(\mu_1\mu_2\inv)^{p^j}Y_j\smat{[\mu_1]&0\\0&[\mu_2]}$. In particular, if $V$ is a representation of $I$ and $v\in V^{H=\chi}$, then for $\un{i}\geq\un{0}$ we have $\un{Y}^{\un{i}}v\in V^{H=\chi\alpha^{\un{i}}}$.    
    \end{enumerate}
\end{lemma}

Let $\lambda=(\un{\lambda}_1,\un{\lambda}_2)\in X_1(\un{T})$ such that $\un{1}\leq\un{\lambda}_1-\un{\lambda}_2\leq\un{p}-\un{2}$. Let $f_0,\ldots,f_{q-1},\phi$ be the elements of $\Ind_I^{\GL_2(\OK)}(\chi_{\lambda}^s)$ defined as in \cite[\S2]{BP12}. For $0\leq\un{i}\leq\un{p}-\un{1}$ we let $i\eqdef\sum\nolimits_{j=0}^{f-1}i_jp^j$. Then by definition and \cite[Lemma~3.2.2.5(ii)]{BHHMS2} we have
\begin{equation}\label{General Eq fi}
    (-1)^{f-1}\bbbra{\scalebox{1}{$\prod\limits_{j=0}^{f-1}$}i_j!}\un{Y}^{\un{p}-\un{1}-\un{i}}\smat{0&1\\1&0}\phi=
    \begin{cases}
        f_i&\text{if}~0\leq i\leq q-2\\
        f_{q-1}-f_0&\text{if}~i=q-1.
    \end{cases}
\end{equation}
The following lemma is a restatement of some results of \cite[\S2]{BP12}.

\begin{lemma}\label{General Lem BP12}
    \begin{enumerate}
    \item
    The $\GL_2(\OK)$-representation $\Ind_I^{\GL_2(\OK)}(\chi_{\lambda}^s)$ is multiplicity-free with constituents $\sset{F(\ft_{\lambda}(-\un{b})):\un{0}\leq\un{b}\leq\un{1}}$. Moreover, the constituent
    $F\bigbra{\ft_{\lambda}(-\un{b})}$ corresponds to the subset $\set{j:b_{j+1}=1}$ in the parametrization of \cite[\S2]{BP12}.
    \item
    The elements $\sset{\un{Y}^{\un{k}}\smat{0&1\\1&0}\phi:\un{0}\leq\un{k}\leq\un{p}-\un{1},\phi}$ form a basis of $\Ind_I^{\GL_2(\OK)}(\chi_{\lambda}^s)$. Moreover, $\phi$ has $H$-eigencharacter $\chi_{\lambda}^s$ and $\un{Y}^{\un{k}}\smat{0&1\\1&0}\phi$ has $H$-eigencharacter $\chi_{\lambda}\alpha^{-\un{k}}=\chi_{\lambda}^s\alpha^{\un{r}-\un{k}}$.
    \item
    Let $\tau$ be the constituent of $\Ind_I^{\GL_2(\OK)}(\chi_{\lambda}^s)$ corresponding to $J\subseteq\cJ$ as in (i) and denote by $Q(\chi_{\lambda}^s,J)$ the unique quotient of $\Ind_I^{\GL_2(\OK)}(\chi_{\lambda}^s)$ with socle $\tau$ (see \cite[Thm.~2.4(iv)]{BP12}).
    \begin{enumerate}
    \item 
    If $J=\emptyset$, then the following $H$-eigenvectors
    \begin{equation*}
        \sset{\un{Y}^{\un{k}}\smat{0&1\\1&0}\phi:\un{p}-\un{1}-(\un{\lambda}_1-\un{\lambda}_2)<\un{k}\leq\un{p}-\un{1},\un{Y}^{\un{p}-\un{1}-(\un{\lambda}_1-\un{\lambda}_2)}\smat{0&1\\1&0}\phi+x\phi}
    \end{equation*}
    with $x=(-1)^{\norm{\un{\lambda}_1}+(f-1)}\bigbra{\!\prod\nolimits_{j=0}^{f-1}(\lambda_{1,j}-\lambda_{2,j})!}\inv\in\FF\x$ form a basis of $\tau$ inside $Q(\chi_{\lambda}^s,\emptyset)=\Ind_I^{\GL_2(\OK)}(\chi_{\lambda}^s)$.
    \item 
    If $J\neq\emptyset$, then the following $H$-eigenvectors
    \begin{equation*}
        \sset{\un{Y}^{\un{k}}\smat{0&1\\1&0}\phi:\begin{aligned}
            &0\leq k_j\leq p-2-(\lambda_{1,j}-\lambda_{2,j})+\delta_{j-1\in J}&\text{if}~j\in J\\
            &p-1-(\lambda_{1,j}-\lambda_{2,j})+\delta_{j-1\in J}\leq k_j\leq p-1&\text{if}~j\notin J
        \end{aligned}}
    \end{equation*}
    map to a basis of $\tau$ inside $Q(\chi_{\lambda}^s,J)$.
    \end{enumerate}
    \end{enumerate}
\end{lemma}

\begin{proof}
    The first statement of (i) is \cite[Lemma~6.2.1(i)]{BHHMS1}, and the second statement of (i) follows from the proof of \cite[Lemma~6.2.1(i)]{BHHMS1}. (ii) and (iii) are restatements of \cite[Lemma~2.5]{BP12} and \cite[Lemma~2.7]{BP12} using (\ref{General Eq fi}).
\end{proof}

\section{On certain \texorpdfstring{$H$}.-eigenvectors in \texorpdfstring{$D_0(\rhobar)$}.}\label{General Sec shift}

In this section, we construct some elements in $D_0(\rhobar)$, which is the (finite-dimensional) representation of $\GL_2(\OK)$ defined in \cite[\S13]{BP12} and is identified with $\pi^{K_1}$ from now on (see condition (i) above Theorem \ref{General Thm main1}). The main result is Proposition \ref{General Prop shift}. They will be the first step in constructing elements of $D_A(\pi)$.

\begin{lemma}\label{General Lem Diamond diagram}
    \begin{enumerate}
    \item
    The $\GL_2(\OK)$-representation $D_0(\rhobar)$ is multiplicity-free with constituents
    \begin{equation*}
        \JH\bra{D_0(\rhobar)}=\sset{\sigma_{\un{b}}:\begin{array}{ll}b_j\in\set{-1,0,1}&\text{if}~j\notin J_{\rhobar}\\b_j\in\set{-1,0,1,2}&\text{if}~j\in J_{\rhobar}\end{array}}.
    \end{equation*}
    Moreover, there is a decomposition of $\GL_2(\OK)$-representations $D_0(\rhobar)=\oplus_{J\subseteq J_{\rhobar}}D_{0,\sigma_J}(\rhobar)$ such that for each $J\subseteq J_{\rhobar}$, $D_{0,\sigma_J}(\rhobar)$ has socle $\sigma_{J}=\sigma_{\un{e}^J}$ and has constituents
    \begin{equation}\label{General Eq JH D_0,sigma}
        \JH\bra{D_{0,\sigma_J}(\rhobar)}=\sset{\sigma_{\un{b}}:\begin{array}{ll}b_j\in\set{-1,0,1}&\text{if}~j\notin J_{\rhobar}\\b_j\in\set{-1,0}&\text{if}~j\in J_{\rhobar}\setminus J\\b_j\in\set{1,2}&\text{if}~j\in J\end{array}}.
    \end{equation}
    \item 
    The $I$-representation $D_0(\rhobar)^{I_1}$ is a direct sum of distinct $I$-characters. For each $J\subseteq\cJ$, $\chi_J$ occurs as a direct summand.
    \item 
    For each $J\subseteq\cJ$, the character $\chi_J$ appears in the component $D_{0,\sigma_{J^{\ss}}}(\rhobar)$, and the character $\chi_J^s$ appears in the component $D_{0,\sigma_{(J-1)^{\ss}}}(\rhobar)$.    
    \end{enumerate}
\end{lemma}

\begin{proof}
    (i). For $J\subseteq J_{\rhobar}$ and $\tau$ an arbitrary Serre weight, we let $\ell(\sigma_J,\tau)\in\NNN\cup\set{\infty}$ be as in \cite[\S12]{BP12} and $\ell(\rhobar,\tau)\eqdef\min_{J\subseteq J_{\rhobar}}\ell(\sigma_J,\tau)\in\NNN\cup\set{\infty}$. Then by \cite[Cor.~4.11]{BP12} and \cite[Lemma~2.4.4]{BHHMS1}, $\ell(\sigma_J,\tau)<\infty$ if and only if  $\tau=\sigma_{\un{b}}$ with $-\un{1}\leq\un{b}-\un{e}^J\leq\un{1}$, in which case we have $\ell(\sigma_J,\tau)=\abs{\set{j:b_j\neq\delta_{j\in J}}}$. In particular, $\ell(\rhobar,\tau)<\infty$ if and only if $\tau=\sigma_{\un{b}}$ with $-\un{1}\leq\un{b}-\un{e}^J\leq\un{1}+\un{e}^{J_{\rhobar}}$, in which case we have $\ell(\rhobar,\tau)=\ell(\sigma_{J(\tau)},\tau)$ with $J(\tau)\eqdef\set{j\in J_{\rhobar}:b_j\geq1}$, and $\tau$ is a constituent of $D_{0,\sigma_{J(\tau)}}(\rhobar)$ by Proposition \cite[Prop.~13.4]{BP12}. Hence for each $J\subseteq J_{\rhobar}$, $D_{0,\sigma_J}(\rhobar)$ has constituents $\tau$ as above such that $J(\tau)=J$, which agrees with (\ref{General Eq JH D_0,sigma}). The other assertions then follow from \cite[Prop.~13.4]{BP12} and \cite[Cor.~13.5]{BP12}.

    (ii). By \cite[Lemma~14.1]{BP12}, the $I$-representation $D_0(\rhobar)^{I_1}$ is a direct sum of distinct $I$-characters. By the proof of \cite[Cor.~13.6]{BP12}, it suffices to find $I$-characters $\chi$ such that $\sigma_{\un{0}}\in\JH\bigbra{\Ind_I^{\GL_2(\OK)}\chi^s}$. Then we conclude using \cite[Prop.~4.2]{Bre14} and \eqref{General Eq sJ}. 

    (iii). The first assertion is clear since $\sigma_J$ lies in
    the component $D_{0,\sigma_{J^{\ss}}}(\rhobar)$ by (\ref{General Eq JH D_0,sigma}). To prove the second assertion, we follow the notation of \cite[\S15]{BP12}. In particular, we let $\cS,\cS^-,\cS^+\subseteq\cJ$ be the subsets associated to $\rhobar^{\ss}$ and $\sigma_J$. By definition we have $\cS=J$, $\cS^-=\cS^+=\emptyset$, hence by \cite[Lemma~15.2]{BP12} applied to $\rhobar^{\ss}$ and $\sigma_J$, we deduce that $\ell\bigbra{\rhobar^{\ss},\sigma_J^{[s]}}=\ell\bigbra{\sigma_{\un{e}^{J-1}},\sigma_J^{[s]}}$. Then by \cite[Lemma~15.3]{BP12} applied to $\rhobar$ and $\sigma_J^{[s]}$ we deduce that $\ell\bigbra{\rhobar,\sigma_J^{[s]}}=\ell\bigbra{\sigma_{(J-1)^{\ss}},\sigma_J^{[s]}}$ (note that the Serre weight $\sigma^{\max}$ in the statement of \cite[Lemma~15.3]{BP12} is our $\sigma_{J_{\rhobar}}$), which completes the proof using \cite[Prop.~13.4]{BP12}.
\end{proof}

For each $J\subseteq\cJ$ we fix a choice of $0\neq v_J\in\pi^{I_1}=D_0(\rhobar)^{I_1}$ with $I$-character $\chi_J$, which is unique up to scalar by Lemma \ref{General Lem Diamond diagram}(ii). The following proposition shows the existence of certain shifts of the elements $v_J$. We will apply $\smat{p&0\\0&1}$ to these elements in order to go beyond $\pi^{K_1}=D_0(\rhobar)$.

\begin{proposition}\label{General Prop shift}
    Let $J\subseteq\cJ$ and $\un{i}\in\ZZ^f$ such that $\un{0}\leq\un{i}\leq\un{f}-\un{e}^{J^{\sh}}$ (see (\ref{General Eq Jsh}) for $J^{\sh}$). Then there exists a unique $H$-eigenvector $y\in D_0(\rhobar)$ satisfying
    \begin{enumerate}
        \item 
        $Y_j^{i_j+1}y=0\ \forall\,j\in\cJ$;
        \item 
        $\un{Y}^{\un{i}}y=v_J$.
    \end{enumerate}
    Moreover, $y$ has $H$-eigencharacter $\chi_J\alpha^{-\un{i}}$. The $\GL_2(\OK)$-subrepresentation of $D_0(\rhobar)$ generated by $y$ lies in $D_{0,J^{\ss}}(\rhobar)$ and has constituents $\sigma_{\un{b}}$ with
    \begin{equation}\label{General Eq range of b}
    \begin{cases}
        b_j=\delta_{j\in J}(=\delta_{j\in J^{\ss}})&\text{if}~j\notin J^{\nss}\\
        b_j\in\set{0,(-1)^{\delta_{j+1\in J}}}&\text{if}~j\in J^{\nss},~i_j=0\\
        b_j\in\set{-1,0,1}&\text{if}~j\in J^{\nss},~i_j>0.
    \end{cases}
    \end{equation}
    We denote this element $y$ by $\un{Y}^{-\un{i}}v_J$.
\end{proposition}

\begin{proof}
    For each $y\in D_0(\rhobar)$ satisfying (i) and (ii), by Lemma \ref{General Lem Yj}(ii) the $I$-representation generated by $y$ is an $I/K_1$-representation with socle $\chi_J$ and cosocle $\chi_J\alpha^{-\un{i}}$, and has constituents $\chi_J\alpha^{-\un{i}'}$ with $\un{0}\leq\un{i}'\leq\un{i}$, each occurring with multiplicity $1$. By \cite[Lemma~6.1.3]{BHHMS1}, such a representation is unique up to isomorphism, and we denote it by $W'$. To prove the existence and uniqueness of such $y$, it suffices to show that there is a unique (up to scalar) $I$-equivariant injection $W'\into D_0(\rhobar)$. Since $W'$ is indecomposable with $I$-socle $\chi_J$, which appears in $D_{0,\sigma_{J^{\ss}}}(\rhobar)$ by Lemma \ref{General Lem Diamond diagram}(iii), any such injection factors through $D_{0,\sigma_{J^{\ss}}}(\rhobar)$.

    \hspace{\fill}

    \noindent\textbf{Claim 1.} The $\GL_2(\OK)$-representation $V'\eqdef\Ind_{I}^{\GL_2(\OK)}(W')$ is multiplicity-free and $\sigma_{J^{\ss}}\in\JH(V')$.
    
    \proof By \cite[Lemma~2.2]{BP12}, $\Ind_{I}^{\GL_2(\OK)}\bigbra{\chi_J\alpha^{-\un{i}'}}$ and $\Ind_{I}^{\GL_2(\OK)}\bigbra{\chi_J\alpha^{-\un{i}'}}^s$ have the same constituents. Since twisting $\chi_J$ by $\alpha^{-\un{i}'}$ corresponds to shifting by $-2\un{i}'$ in the extension graph, it follows from Lemma \ref{General Lem BP12}(i), \cite[Remark~2.4.5(ii)]{BHHMS1} and (\ref{General Eq bound s}) that $\Ind_{I}^{\GL_2(\OK)}\bigbra{\chi_J\alpha^{-\un{i}'}}$ is multiplicity-free and has constituents $F\bigbra{\ft_{\lambda_J}(-\un{b})}$ with $2\un{i}'\leq\un{b}\leq2\un{i}'+\un{1}$. Hence the $\GL_2(\OK)$-representation $V'$ is multiplicity-free and has constituents $F\bigbra{\ft_{\lambda_J}(-\un{b})}$ with $\un{0}\leq\un{b}\leq2\un{i}+\un{1}$. By Lemma \ref{General Lem change origin}(i) and taking $b_j=\delta_{j\in J^{\nss}}$  we deduce that $\sigma_{J^{\ss}}\in\JH\bigbra{\Ind_{I}^{\GL_2(\OK)}(\chi_J)}\subseteq\JH(V')$.\qed

    \hspace{\fill}
    
    It follows from Claim 1 that there is a unique (up to scalar) $\GL_2(\OK)$-equivariant map $f:V'\to\Inj_{\GL_2(\Fq)}\sigma_{J^{\ss}}$. We denote by $V''$ the image of $f$. 

    \hspace{\fill}

    \noindent\textbf{Claim 2.} The $\GL_2(\OK)$-representation $V''$ has constituents $\sigma_{\un{b}}$ for $\un{b}$ as in (\ref{General Eq range of b}). 
    
    \proof We let $\tau$,$\tau'$ be constituents of $V'$ such that $\tau=F\bigbra{\ft_{\lambda_J}(-\un{b})}$ and $\tau'=F\bigbra{\ft_{\lambda_J}(-\un{b}+e_{j_0})}$ with $\un{0}<\un{b}\leq2\un{i}+\un{1}$, $j_0\in\cJ$ and $b_{j_0}\neq0$. We write $\un{b}=2\un{c}+\un{\varepsilon}$ with $\un{0}\leq\un{c}\leq\un{i}$ and $\un{0}\leq\un{\varepsilon}\leq\un{1}$. If $\varepsilon_{j_0}=1$, then both $\tau$ and $\tau'$ are constituents of $\Ind_I^{\GL_2(\OK)}(\chi_J\alpha^{-\un{c}})$. We deduce from \cite[Thm.~2.4]{BP12} that $V'$ has a length $2$ subquotient with socle $\tau$ and cosocle $\tau'$. If $\varepsilon_{j_0}=0$, then we deduce from \cite[Lemma~3.8]{HW22} (with $j=j_0$, $\chi=\chi_J\alpha^{-\un{c}}$, $J(\tau)=\set{j:\varepsilon_{j+1}=0}$, $J(\tau')=J(\tau)\setminus\set{j_0-1}$) that $V'$ has a length $2$ subquotient with socle $\tau'$ and cosocle $\tau$. Moreover, these are all possible non-split length $2$ subquotients of $V'$ by \cite[Lemma~2.4.6]{BHHMS1}. 
    
    Then we use the notation of \cite[\S4.1.1]{LLHLM20}. We make $\JH(V')$ into a directed graph by letting $\sigma\in\JH(V')$ point to $\sigma'\in\JH(V')$ if $V'$ has a length $2$ subquotient with socle $\sigma'$ and cosocle $\sigma$. By construction, $V''$ is a quotient of $V'$ with socle $\sigma_{J^{\ss}}$. It follows from the dual version of \cite[Prop.~4.1.1]{LLHLM20} that the constituents of $V''$ are those $\sigma\in\JH(V')$ which admit a path towards $\sigma_{J^{\ss}}=F\bigbra{\ft_{\lambda_J}(-\un{e}^{J^{\nss}})}$. From the structure of $\JH(V')$ we deduce that $V''$ has constituents $F\bigbra{\ft_{\lambda_{J}}(-\un{b})}$ with
    \begin{equation*}
    \begin{cases}
        b_j=0&\text{if}~j\notin J^{\nss}\\
        b_j\in\set{0,1}&\text{if}~j\in J^{\nss},~i_j=0\\
        b_j\in\set{0,1,2}&\text{if}~j\in J^{\nss},~i_j>0.
    \end{cases}
    \end{equation*}
    Then we conclude (\ref{General Eq range of b}) by Lemma \ref{General Lem change origin} with a case-by-case examination. \qed

    \hspace{\fill}
    
    Since $V'$ is multiplicity-free, it follows from Claim 2 and (\ref{General Eq JH D_0,sigma}) that $f$ factors through $D_{0,\sigma_{J^{\ss}}}(\rhobar)$. Then by Frobenius reciprocity, we have
    \begin{equation*}
        \dim_{\FF}\Hom_I\bigbra{W',D_{0,\sigma_{J^{\ss}}}(\rhobar)|_I}=\dim_{\FF}\Hom_{\GL_2(\OK)}\bigbra{V',D_{0,\sigma_{J^{\ss}}}(\rhobar)}=1.
    \end{equation*}
    To complete the proof, it remains to show that any nonzero $I$-equivariant map $W'\to D_{0,\sigma_{J^{\ss}}}(\rhobar)$ is injective. Since $W'$ has $I$-socle $\chi_J$, it suffices to show that the image of $\chi_J$ is nonzero. By Frobenius reciprocity, it suffices to show that the image of the subrepresentation $\Ind_I^{\GL_2(\OK)}(\chi_J)$ of $V'$ under $f$ is nonzero. This follows from the fact that both $\Ind_I^{\GL_2(\OK)}(\chi_J)$ and $V''$ contain  $\sigma_{J^{\ss}}$ as a constituent, and $V'$ is multiplicity-free.
\end{proof}

\begin{remark}
    When $J_{\rhobar}\neq\emptyset$, up to scalars there are more $I_1$-invariants than these $v_J$ for $J\subseteq\cJ$. However, Proposition \ref{General Prop shift} does not hold for these extra $I_1$-invariants.
\end{remark}

\section{The relations between \texorpdfstring{$H$}.-eigenvectors}\label{General Sec relation}

In this section, we study various $\GL_2(\OK)$-subrepresentations of $\pi$ generated by the elements $\un{Y}^{-\un{i}}v_J$ defined in Proposition \ref{General Prop shift}. The main results are Proposition \ref{General Prop relation 1} and Proposition \ref{General Prop relation 2}. Then we study the relations between the vectors $v_J$ for $J\subseteq\cJ$. The main results are Proposition \ref{General Prop vector} and Proposition \ref{General Prop vector complement}.

Recall that we have defined $Q(\chi_{\lambda}^s,J)$ for $\lambda=(\un{\lambda}_1,\un{\lambda}_2)\in X_1(\un{T})$ such that $\un{1}\leq\un{\lambda}_1-\un{\lambda}_2\leq\un{p}-\un{2}$ and $J\subseteq\cJ$ in Lemma \ref{General Lem BP12}(iii). The following lemma is a generalization of \cite[Lemma~3.2.3.3]{BHHMS2} (where $\rhobar$ was assumed to be semisimple).

\begin{lemma}\label{General Lem p001}
    Let $J\subseteq\cJ$ and $\un{i}\in\ZZ^f$ such that $\un{0}\leq\un{i}\leq\un{f}-\un{e}^{J^{\sh}}$. 
\begin{enumerate}
    \item The $\GL_2(\OK)$-subrepresentation $\bang{\GL_2(\OK)\smat{p&0\\0&1}\un{Y}^{-\un{i}}v_J}$ of $\pi$ is multiplicity-free with socle $\sigma_{(J-1)^{\ss}}=\sigma_{\un{e}^{(J-1)^{\ss}}}$ and cosocle $\sigma_{\un{c}}$ with 
    \begin{equation}\label{General Eq p001 c}
        c_j=(-1)^{\delta_{j+1\notin J}}\bigbra{2i_j+1+\delta_{j\in(J-1)^{\ss}}-\delta_{j\in J\Delta(J-1)^{\ss}}}~\forall\,j\in\cJ.
    \end{equation}
    \item 
    We have
    \begin{multline}\label{General Eq p001}
        \bang{\GL_2(\OK)\smat{p&0\\0&1}\un{Y}^{-\un{i}}v_J}\Big/\sum\limits_{\un{0}\leq\un{i}'<\un{i}}\bigang{\!\GL_2(\OK)\smat{p&0\\0&1}\un{Y}^{-\un{i}'}v_J}\\
        \cong Q\bbra{\chi_J^s\alpha^{\un{i}},\sset{j:j+1\in J\Delta(J-1)^{\ss},\ i_{j+1}=0}}.
    \end{multline}
    \item 
    Let $\un{m}\in\ZZ^f$ with each $m_j$ between $\delta_{j\in(J-1)^{\ss}}$ and $c_j$ (as in (\ref{General Eq p001 c})). Then there is a unique subrepresentation $I\bigbra{\sigma_{(J-1)^{\ss}},\sigma_{\un{m}}}$ of $\bang{\GL_2(\OK)\smat{p&0\\0&1}\un{Y}^{-\un{i}}v_J}$ with cosocle $\sigma_{\un{m}}$. In particular, $\bang{\GL_2(\OK)\smat{p&0\\0&1}\un{Y}^{-\un{i}}v_J}=I\bigbra{\sigma_{(J-1)^{\ss}},\sigma_{\un{c}}}$. Moreover, $I\bigbra{\sigma_{(J-1)^{\ss}},\sigma_{\un{m}}}$ has constituents $\sigma_{\un{b}}$ with each $b_j$ between $\delta_{j\in(J-1)^{\ss}}$ and $m_j$, and we have
    \begin{equation}\label{General Eq p001 dim1}
        \dim_{\FF}\Hom_{\GL_2(\OK)}\bigbra{I\bigbra{\sigma_{(J-1)^{\ss}},\sigma_{\un{m}}},\pi}=1.
    \end{equation}
\end{enumerate}
\end{lemma}

\begin{proof}
    (i). We follow closely the proof of \cite[Lemma~3.2.3.3]{BHHMS2}. The vectors $\un{Y}^{-\un{i}}v_J$ and $\un{Y}^{-\un{i'}}v_J$ are defined in Proposition \ref{General Prop shift}. We let $W'$ (resp.\,$W$) be the $I$-subrepresentation of $\pi$ generated by $\un{Y}^{-\un{i}}v_J$ (resp.\,$\smat{0&1\\p&0}\un{Y}^{-\un{i}}v_J$) and $V\eqdef\Ind_I^{\GL_2(\OK)}(W)$. In particular, $W'$ is the same representation as in the proof of Proposition \ref{General Prop shift}. By the proof of \cite[Lemma~3.2.3.3]{BHHMS2}, we have:
    \begin{enumerate}
        \item[(a)]
        $V$ is multiplicity-free as a $\GL_2(\OK)$-representation with constituents $F\bigbra{\ft_{\lambda_J}(-\un{b})}$ for $\un{0}\leq\un{b}\leq2\un{i}+1$ (they are well-defined by (\ref{General Eq bound s}));
        \item[(b)] 
        For each $\un{0}\leq\un{b}\leq2\un{i}+1$, the unique subrepresentation of $V$ with cosocle $F\bigbra{\ft_{\lambda_J}(-\un{b})}$ has constituents $F\bigbra{\ft_{\lambda_J}(-\un{a})}$ for $\un{0}\leq\un{a}\leq\un{b}$;
        \item[(c)]
        $V$ has a filtration with subquotients $\Ind_I^{\GL_2(\OK)}\!\bigbra{\chi_J^s\alpha^{\un{i}'}}$ for $\un{0}\leq\un{i}'\leq\un{i}$. Each $\Ind_I^{\GL_2(\OK)}\!\bigbra{\chi_J^s\alpha^{\un{i}'}}$ has constituents $F\bigbra{\ft_{\lambda_J}(-\un{b})}$ with $2\un{i}'\leq\un{b}\leq2\un{i}'+1$, and the constituent $F\bigbra{\ft_{\lambda_J}(-\un{b})}$ of $\Ind_I^{\GL_2(\OK)}\!\bigbra{\chi_J^s\alpha^{\un{i}'}}$ corresponds to the subset $\sset{j:b_{j+1}\ \text{is odd}}\subseteq\cJ$ (see Lemma \ref{General Lem BP12}(i)).
    \end{enumerate}
    The $I$-equivariant inclusion $W'\into D_0(\rhobar)$ in the proof of Proposition \ref{General Prop shift} induces an $I$-equivariant inclusion $W\into\pi$ by applying $\smat{0&1\\p&0}$. By Frobenius reciprocity, this induces a $\GL_2(\OK)$-equivariant map $V\to\pi$ with image $\ovl{V}\eqdef\bang{\GL_2(\OK)\smat{0&1\\p&0}\un{Y}^{-\un{i}}v_J}=\bang{\GL_2(\OK)\smat{p&0\\0&1}\un{Y}^{-\un{i}}v_J}\subseteq\pi$. In particular, it follows from (b) that the cosocle of $\ovl{V}$ is $F\bigbra{\ft_{\lambda_J}(-(2\un{i}+1))}=\sigma_{\un{c}}$ with
    \begin{align*}
        c_j&=(-1)^{\delta_{j+1\in J}}\bigbra{-(2i_j+1)+\delta_{j\in J}}+2\delta_{j\in J^{\sh}}\\
        &=(-1)^{\delta_{j+1\notin J}}\bigbra{2i_j+1-\delta_{j\in J}+2(-1)^{\delta_{j+1\notin J}}\delta_{j\in J^{\sh}}}\\
        &=(-1)^{\delta_{j+1\notin J}}\bigbra{2i_j+1+\delta_{j\in(J-1)^{\ss}}-\delta_{j\in J\Delta(J-1)^{\ss}}},
    \end{align*}
    where the first equality follows from Lemma \ref{General Lem change origin} and the last equality is elementary (for example, one can separate the cases $j\in(J-1)^{\ss}$ and $j\notin(J-1)^{\ss}$).
    
    \hspace{\fill}

    \noindent\textbf{Claim.} We have $W(\rhobar)\cap\JH(V)\subseteq\JH\!\bigbra{\!\Ind_I^{\GL_2(\OK)}(\chi_J^s)}$. 
    
    \proof It suffices to show that for each $\sigma\in W(\rhobar)$, we have $\sigma=F\bigbra{\ft_{\lambda_J}(-\un{b})}$ for some $\un{b}\leq\un{1}$. We check it for $\sigma_{\emptyset}$, the other cases being similar. By Lemma \ref{General Lem change origin} (with $\un{a}=\un{0}$), we get $b_j=(-1)^{\delta_{j+1\in J}}\bigbra{2\delta_{j\in J^{\sh}}}+\delta_{j\in J}$. If $j\notin J^{\sh}$, then $b_j=\delta_{j\in J}\leq1$. If $j\in J^{\sh}$, then $b_j=-2+1=-1$.\qed

    \hspace{\fill}
    
    Recall that $\soc_{\GL_2(\OK)}\pi=\bigoplus_{\sigma\in W(\rhobar)}\sigma$. Assume that $\sigma$ is in the socle of $\ovl{V}$, then we have $\sigma\in W(\rhobar)\cap\JH(V)\subseteq\JH\!\bigbra{\!\Ind_I^{\GL_2(\OK)}(\chi_J^s)}$ by the claim above. Moreover, the image of the subrepresentation $\Ind_I^{\GL_2(\OK)}(\chi_J^s)$ of $V$ in $\pi$ lies in $D_0(\rhobar)$, hence lies in the component $D_{0,\sigma_{(J-1)^{\ss}}}(\rhobar)$ by Lemma \ref{General Lem Diamond diagram}(iii) and Frobenius reciprocity, which implies that $\sigma$ must be $\sigma_{(J-1)^{\ss}}$, the only Serre weight of $\rhobar$ appearing in $D_{0,\sigma_{(J-1)^{\ss}}}(\rhobar)$. Since $V$ is multiplicity-free by (a), we deduce that $\ovl{V}$ is the unique quotient of $V$ with socle $\sigma_{(J-1)^{\ss}}$.

    \hspace{\fill}
    
    (ii). By Lemma \ref{General Lem change origin}(ii), we have $\sigma_{(J-1)^{\ss}}=F\bigbra{\ft_{\lambda_J}(-\un{e}^{J\Delta(J-1)^{\ss}})}$, hence $\ovl{V}$ has constituents $F\bigbra{\ft_{\lambda_J}(-\un{b})}$ with $\delta_{j\in J\Delta(J-1)^{\ss}}\leq b_j\leq2i_j+1$ for all $j$ (or equivalently, $\sigma_{\un{b}}$ with each $b_j$ between $\delta_{j\in(J-1)^{\ss}}$ and $c_j$ by Lemma \ref{General Lem change origin}). By (c), the LHS of (\ref{General Eq p001}) is the quotient of $\Ind_I^{\GL_2(\OK)}\!\bigbra{\chi_J^s\alpha^{\un{i}}}$ whose constituents are $F\bigbra{\ft_{\lambda_J}(-\un{b})}$ with $\max(\delta_{j\in J\Delta(J-1)^{\ss}},2i_j)\leq b_j\leq2i_j+1$, hence it has irreducible socle $F\bigbra{\ft_{\lambda_J}(-\un{a})}$ with $a_j=\max(\delta_{j\in J\Delta(J-1)^{\ss}},2i_j)$ by (b). Since $a_j$ is odd if and only if $i_j=0$ and $j\in J\Delta(J-1)^{\ss}$, it follows from (c) that the constituent $F\bigbra{\ft_{\lambda_J}(-\un{a})}$ of $\Ind_I^{\GL_2(\OK)}\!\bigbra{\chi_J^s\alpha^{\un{i}}}$ corresponds to the subset $\sset{j:j+1\in J\Delta(J-1)^{\ss},\ i_{j+1}=0}$. 

    \hspace{\fill}

    (iii). Since $\sigma_{\un{m}}$ is a constituent of the multiplicity-free representation $\ovl{V}$ by the previous paragraph, there is a unique subrepresentation of $\ovl{V}$ with cosocle $\sigma_{\un{m}}$, which moreover has constituents as in the statement by (b). We denote it by $I\bigbra{\sigma_{(J-1)^{\ss}},\sigma_{\un{m}}}$. By the last paragraph of the proof of (i), any constituent of $I\bigbra{\sigma_{(J-1)^{\ss}},\sigma_{\un{m}}}$ which is also an element of $W(\rhobar)$ must appear in $D_{0,\sigma_{(J-1)^{\ss}}}(\rhobar)$, hence has to be $\sigma_{(J-1)^{\ss}}$. Together with $\soc_{\GL_2(\OK)}\pi=\oplus_{\sigma\in W(\rhobar)}\sigma$, we deduce that 
    \begin{equation*}
        1\leq\dim_{\FF}\Hom_{\GL_2(\OK)}\bigbra{I\bigbra{\sigma_{(J-1)^{\ss}},\sigma_{\un{m}}},\pi}\leq\dim_{\FF}\Hom_{\GL_2(\OK)}\bigbra{\sigma_{(J-1)^{\ss}},\pi}=1,
    \end{equation*}
    which completes the proof.
\end{proof}

\begin{remark}\label{General Rk Q123}
    For $\lambda=(\un{\lambda}_1,\un{\lambda}_2)\in X_1(\un{T})$, $\un{i}\in\NNN^f$ such that $2\un{i}+\un{1}\leq\un{\lambda}_1-\un{\lambda}_2\leq\un{p}-\un{2}$ and $J'\subseteq\cJ$, we let $W'$ be the $I$-representation as in the proof of Proposition \ref{General Prop shift} with $\chi_J$ replaced by $\chi_{\lambda}$, and we denote by $Q\bigbra{\chi_{\lambda}^s,\chi_{\lambda}^s\alpha^{\un{i}},J'}$ the unique quotient of the $\GL_2(\OK)$-representation $\Ind_I^{\GL_2(\OK)}\bbra{\!\smat{0&1\\p&0}W'}$ whose socle is the constituent of $\Ind_I^{\GL_2(\OK)}(\chi_{\lambda}^s)$ corresponding to $J'$. Then the proof of Lemma \ref{General Lem p001} shows that $\bang{\GL_2(\OK)\smat{p&0\\0&1}\un{Y}^{-\un{i}}v_J}\cong Q\bigbra{\chi_J^s,\chi_J^s\alpha^{\un{i}},(J\Delta(J-1)^{\ss})-1}$.
\end{remark}

\begin{corollary}\label{General Cor structure of D0}
    For each $J\subseteq J_{\rhobar}$, we have (see Lemma \ref{General Lem p001}(iii) for the notation)
    \begin{equation*}
        D_{0,\sigma_J}(\rhobar)=\sum\limits_{(J')^{\ss}=J}I\bigbra{\sigma_{J},\sigma_{\un{e}^{J}+\un{\varepsilon}^{J'}}}=\sum\limits_{(J')^{\ss}=J}I\bigbra{\sigma_{\un{e}^{(J')^{\ss}}},\sigma_{\un{e}^{(J')^{\ss}}+\un{\varepsilon}^{J'}}},
    \end{equation*}
    where $\un{\varepsilon}^{J'}\in\set{\pm1}^f$ with $\varepsilon^{J'}_j\eqdef(-1)^{\delta_{j\notin J'}}$.
\end{corollary}

\begin{proof}
    For each $J'\subseteq\cJ$ such that $(J')^{\ss}=J$, by applying Lemma \ref{General Lem p001}(i),(iii) with $(J,\un{i})$ there being $(J'+1,\un{f}-\un{e}^{(J'+1)^{\sh}})$, we see that $I\bigbra{\sigma_{J},\sigma_{\un{e}^{J}+\un{\varepsilon}^{J'}}}$ is well-defined. Then the result follows from Lemma \ref{General Lem Diamond diagram}(i), \cite[Prop.~13.4]{BP12} and (\ref{General Eq p001 dim1}).
\end{proof}

The following proposition is a generalization of \cite[Lemma~3.2.3.1]{BHHMS2} (where $\rhobar$ was assumed to be semisimple), which gives a first example of the relations between the vectors $v_J\in D_0(\rhobar)$ and is a special case of Proposition \ref{General Prop vector} below.

\begin{proposition}\label{General Prop vector simple}
    For $J\subseteq\cJ$, there exists a unique element $\mu_{J,(J-1)^{\ss}}\in\FF\x$ such that 
    \begin{equation}\label{General Eq vector simple}
        \bbbra{\prod\limits_{j+1\in J\Delta(J-1)^{\ss}}Y_j^{s^{(J-1)^{\ss}}_j}\prod\limits_{j+1\notin J\Delta(J-1)^{\ss}}Y_j^{p-1}}\smat{p&0\\0&1}v_J=\mu_{J,(J-1)^{\ss}}v_{(J-1)^{\ss}}.
    \end{equation}
\end{proposition}

\begin{proof}
    By Lemma \ref{General Lem p001}(ii) and its proof, we have $\bang{\GL_2(\OK)\smat{p&0\\0&1}v_J}\cong Q\bigbra{\chi_J^s,(J\Delta(J-1)^{\ss})-1}$ such that $\smat{0&1\\p&0}v_J$ corresponds to the image of $\phi\in\Ind_I^{\GL_2(\OK)}(\chi_J^s)$ (see above (\ref{General Eq fi}) for $\phi$) in $Q\bigbra{\chi_J^s,(J\Delta(J-1)^{\ss})-1}$, and the socle is $\sigma_{(J-1)^{\ss}}$ which corresponds to the subset $(J\Delta(J-1)^{\ss})-1$ for $\Ind_I^{\GL_2(\OK)}(\chi_J^s)$ (see Lemma \ref{General Lem BP12}(i)). By Lemma \ref{General Lem p-2-s and s} applied to $J$ and $J'=(J-1)^{\ss}$, for $j\in J\Delta(J-1)^{\ss}-1$ we have $(p-2-s^J_j)+\delta_{j-1\in(J\Delta(J-1)^{\ss})-1}=s^{(J-1)^{\ss}}_j$. Then by Lemma \ref{General Lem BP12}(iii) applied to $\lambda=\lambda_J$ (and recall that $\chi_J=\chi_{\lambda_J}$ with $\lambda_J=(\un{s}^J+\un{t}^J,\un{t}^J)$), the LHS of (\ref{General Eq vector simple}) is nonzero in $\sigma_{(J-1)^{\ss}}$ and is the unique (up to scalar) $H$-eigenvector in $\sigma_{(J-1)^{\ss}}$ killed by all $Y_j$. It follows that the LHS of (\ref{General Eq vector simple}) is a nonzero $I_1$-invariant of $\sigma_{(J-1)^{\ss}}$, hence is a scalar multiple of $v_{(J-1)^{\ss}}$.
\end{proof}

For $J,J'\subseteq\cJ$, we define $\un{t}^J(J')\in\ZZ^f$ by
\begin{equation}\label{General Eq tJJ'}
    t^J(J')_j\eqdef p-1-s^J_j+\delta_{j-1\in J'},
\end{equation}
where $s^J_j$ is defined in (\ref{General Eq sJ}). In particular, by (\ref{General Eq bound s}) we have 
\begin{equation}\label{General Eq bound tJ}
    1\leq t^J(J')_j\leq p-1-2(f-\delta_{j\in J^{\sh}})~\forall\,j\in\cJ.
\end{equation}
The following proposition is a generalization of \cite[Lemma~3.2.3.4]{BHHMS2} (where $\rhobar$ was assumed to be semisimple).

\begin{proposition}\label{General Prop relation 1}
    Let $J\subseteq\cJ$, $j_0\in\cJ$ and $\un{i}\in\ZZ^f$ such that $\un{0}\leq\un{i}\leq\un{f}-\un{e}^{J^{\sh}}$ and $i_{j_0+1}=0$. Suppose that $j_0+1\in J\Delta(J-1)^{\ss}$. Then for each $J'\subseteq\cJ$ such that $j_0\notin J'$, we have
    \begin{equation}\label{General Eq relation 1 statement}
    \begin{aligned}
        \bbbra{Y_{j'}^{\delta_{J'=\emptyset}}\prod\limits_{j\notin J'}Y_j^{2i_j+t^J(J')_j}}\smat{p&0\\0&1}\bbra{\un{Y}^{-\un{i}}v_J}&=0~\forall\,j'\in\cJ&.
    \end{aligned}
    \end{equation}
\end{proposition}

\begin{proof}
    The proof is analogous to the one of \cite[Lemma~3.2.3.4]{BHHMS2}. We assume that $J'\neq\emptyset$. The case $J=\emptyset$ is similar and is left as an exercise. By Lemma \ref{General Lem Yj}(ii), it suffices to show that the $H$-eigencharacter $\chi\eqdef\chi_J\alpha^{-\un{i}}\bigbbbra{\prod\nolimits_{j\notin J'}\alpha_j^{2i_j+p-1-s^J_j+\delta_{j-1\in J'}}}$ does not occur in $\bang{\GL_2(\OK)\smat{p&0\\0&1}\un{Y}^{-\un{i}}v_J}$. By Lemma \ref{General Lem p001}(ii), it suffices to show that the $H$-character $\chi$ does not occur in $V_{\un{i}'}\eqdef Q\bigbra{\chi_J^s\alpha^{\un{i}'},J_{\un{i}'}}$ for $\un{0}\leq\un{i}'\leq\un{i}$, where $J_{\un{i}'}\eqdef\set{j:j+1\in J\Delta(J-1)^{\ss}, i'_{j+1}=0}$. Note that $j_0\in J_{\un{i}'}$ for all $\un{0}\leq\un{i}'\leq\un{i}$ by assumption.

    We have $\chi_J\alpha^{-\un{i}'}=\chi_{\lambda_J-\alpha^{\un{i}'}}$ (see \S\ref{General Sec sw} for the notation). Then by Lemma \ref{General Lem BP12}(i),(ii),(iii)(b) applied to $\lambda=\lambda_J-\alpha^{\un{i}'}$, the $H$-eigencharacters that occur in $V_{\un{i}'}$ are $\chi_J\alpha^{-\un{i}'}\alpha^{-\un{k}}$ (coming from the element $\un{Y}^{\un{k}}\smat{p&0\\0&1}\bigbra{\un{Y}^{-\un{i}'}v_J}$), where
    \begin{equation}\label{General Eq relation 1 bound}
    \begin{cases}
        0\leq k_j\leq p-2-(s^J_j-2i'_j)+\delta_{j-1\in J_0}&~\text{if}~j\in J_0\\
        p-1-(s^J_j-2i_j')+\delta_{j-1\in J_0}\leq k_j\leq p-1&~\text{if}~j\notin J_0
    \end{cases}
    \end{equation}
    for $J_0\supseteq J_{\un{i}'}$. In particular, we have $j_0\in J_0$.
    
    Assume $\chi=\chi_J\alpha^{-\un{i}'}\alpha^{\un{k}}$ for some $\un{i}',\un{k}$ as above, then from the definition of $\chi$ we have 
    \begin{equation*}
        \sum\limits_{j\notin J'}(2i_j+p-1-s^J_{j}+\delta_{j-1\in J'})p^j-\sum\limits_{j=0}^{f-1}(i_j-i'_j)p^j\equiv\sum\limits_{j=0}^{f-1}k_jp^j~\mod~(q-1),
    \end{equation*}
    or equivalently,
    \begin{equation}\label{General Eq relation 1 mod q-1}
        \sum\limits_{j\notin J'}(i_j+i'_j+p-1-s^J_{j}+\delta_{j-1\in J'})p^j-\sum\limits_{j\in J'}(i_j-i'_j)p^j\equiv\sum\limits_{j=0}^{f-1}k_jp^j~\mod~(q-1).
    \end{equation}

    Then we define integers $\eta_j\in\ZZ$ for all $j\in\cJ$. For $j_1\notin J'$ (such $j_1$ exists since $J'\neq\emptyset$), we let $w\in\set{0,\ldots,f-1}$ (depending on $j_1$) such that $j_1+1,\ldots,j_1+w\in J'$ and $j_1+w+1\notin J'$ (so $w=0$ if $j_1+1\notin J'$). We define $\eta_j$ for $j=j_1+1,\ldots,j_1+w+1$ as follows: 
    \begin{enumerate}
    \item 
    If $i_{j_1+w'}=i'_{j_1+w'}$ for all $1\leq w'\leq w$ (which is automatic if $w=0$), then we define $\eta_j\eqdef0$ for all $j=j_1+1,\ldots,j_1+w+1$;
    \item
    Otherwise, we let $w_0\in\set{1,\ldots,w}$ be minimal such that $i_{j_1+w_0}\neq i'_{j_1+w_0}$ and we define
    \begin{equation}\label{General Eq etaJ}
        \eta_j\eqdef
    \begin{cases}
        0&\text{if}~j=j_1+1,\ldots,j_1+w_0-1~(\text{and $w_0\neq1$})\\
        p&\text{if}~j=j_1+w_0\\
        p-1&\text{if}~j=j_1+w_0+1,\ldots,j_1+w~(\text{and $w_0\neq w$})\\
        -1&\text{if}~j=j_1+w+1.
    \end{cases}
    \end{equation}
    \end{enumerate}
    In particular, we have $\sum\nolimits_{j=j_1+1}^{j_1+w+1}\eta_jp^j\equiv0~\mod~(q-1)$. When we vary $j_1\notin J'$, we get the definition of $\eta_j$ for all $j\in\cJ$. By adding $\sum\nolimits_{j=0}^{f-1}\eta_jp^j$ to (\ref{General Eq relation 1 mod q-1}) for all $j_1\notin J'$, we get
    \begin{equation}\label{General Eq relation 1 mod q-1 eta}
        \sum\limits_{j\notin J'}\bigbra{i_j+i'_j+p-1-s^J_{j}+\delta_{j-1\in J'}+\eta_j}p^j+\sum\limits_{j\in J'}\bigbra{\eta_j-(i_j-i'_j)}p^j\equiv\sum\limits_{j=0}^{f-1}k_jp^j~\mod~(q-1).
    \end{equation}

    \hspace{\fill}

    \noindent\textbf{Claim 1.} Each coefficient of the LHS of (\ref{General Eq relation 1 mod q-1 eta}) is between $0$ and $p-1$, not all equal to $0$ and not all equal to $p-1$. 
    
    \proof First we prove that each coefficient of the LHS of (\ref{General Eq relation 1 mod q-1 eta}) is between $0$ and $p-1$. By (\ref{General Eq bound s}) we have
    \begin{equation}\label{General Eq bound s weak}
         1\leq p-1-s^J_j\leq p-2-2(f-\delta_{j\in J^{\sh}}).
    \end{equation}    
    We remark that the first inequality of (\ref{General Eq bound s weak}) is weaker than (\ref{General Eq bound s}), and is needed to prove Remark \ref{General Rk relation 1 general}. If $j\notin J'$, then using $0\leq i_j,i'_j\leq f-\delta_{j\in J^{\sh}}$, $\delta_{j-1\in J'}\in\set{0,1}$ and $\eta_j\in\set{-1,0}$ since $j\notin J'$, we deduce from (\ref{General Eq bound s weak}) that $0\leq(i_j+i'_j+p-1-s^J_{j}+\delta_{j-1\in J'}+\eta_j)\leq p-1$. If $j\in J'$, by the definition of $\eta_j$ and a case-by-case examination, we deduce that $0\leq\eta_j-(i_j-i'_j)\leq p-1$.
    
    Next we prove that the coefficients of the LHS of (\ref{General Eq relation 1 mod q-1 eta}) are not all equal to $0$. Otherwise, by the previous paragraph we must in particular have $\eta_j=-1$ for all $j\notin J'$. By the definition of $\eta_j$ for $j\notin J'$ (that is, for $j=j_1+w+1$ in (\ref{General Eq etaJ})), there exists $j'\in J'$ such that $\eta_{j'}=p$, which implies $\eta_{j'}-(i_{j'}-i'_{j'})>0$ since $p\geq4f+4$ by (\ref{General Eq genericity}), a contradiction. 
    
    Finally we prove that the coefficients of the LHS of (\ref{General Eq relation 1 mod q-1 eta}) are not all equal to $p-1$. Otherwise, by the first paragraph we must have $\eta_j=0$ for all $j\notin J'$. By the definition of $\eta_j$ for $j\notin J'$, we must have $\eta_j=0$ for all $j\in\cJ$, hence $\eta_j-(i_j-i'_j)$ cannot be $p-1$. This implies $J'=\emptyset$, which is a contradiction.\qed

    \hspace{\fill}
    
    It follows from Claim 1 that the equation (\ref{General Eq relation 1 mod q-1 eta}) has solution   
    \begin{equation}\label{General Eq relation 1 kj}
        k_j=
    \begin{cases}
        i_j+i'_j+p-1-s^J_{j}+\delta_{j-1\in J'}+\eta_j&\text{if}~j\notin J'\\
        \eta_j-(i_j-i'_j)&\text{if}~j\in J'.
    \end{cases}
    \end{equation}

    \hspace{\fill}

    \noindent\textbf{Claim 2.} We have $j_0-1\notin J'$ and $j_0-1\in J_0$.
    
    \proof Since $j_0\notin J'$ and $j_0\in J_0$, by (\ref{General Eq relation 1 bound}) and (\ref{General Eq relation 1 kj}) we have 
    \begin{equation}\label{General Eq relation 1 kj0}
        k_{j_0}=i_{j_0}+i'_{j_0}+p-1-s^J_{j_0}+\delta_{j_0-1\in J'}+\eta_{j_0}\leq p-2-s^J_{j_0}+2i'_{j_0}+\delta_{j_0-1\in J_0}.
    \end{equation}
    By the definition of $\eta_j$, if $j_0-1\notin J'$, then $\eta_{j_0}=0$ since $j_0\notin J'$, and thus $\eta_{j_0}=-1$ implies $j_0-1\in J'$. In particular, we have $\delta_{j_0-1\in J'}+\eta_{j_0}\geq0$. Then we deduce from (\ref{General Eq relation 1 kj0}) that $i_{j_0}+1\leq i'_{j_0}+\delta_{j_0-1\in J_0}$, which implies $i_{j_0}=i'_{j_0}$ and $j_0-1\in J_0$ since $i'_{j_0}\leq i_{j_0}$. 
    
    Then by (\ref{General Eq relation 1 bound}) we have 
    \begin{equation}\label{General Eq relation 1 kj0-1}
        k_{j_0-1}\leq p-2-(s^J_{j_0-1}-2i'_{j_0-1})+\delta_{j_0-2\in J_0}\leq p-1-s^J_{j_0-1}+2i'_{j_0-1}.
    \end{equation}
    Suppose that $j_0-1\in J'$, then by (\ref{General Eq relation 1 kj0}) and using $i_{j_0}=i'_{j_0}$ and $j_0-1\in J_0$, we must have $\eta_{j_0}=-1$. Then by (\ref{General Eq etaJ}) we have $\eta_{j_0-1}\geq p-1$, which implies $k_{j_0-1}\geq p-1-(i_{j_0-1}-i'_{j_0-1})$ by (\ref{General Eq relation 1 kj}). Combining with (\ref{General Eq relation 1 kj0-1}) we deduce that $s^J_{j_0-1}\leq i_{j_0-1}+i'_{j_0-1}\leq 2(f-\delta_{j_0-1\in J^{\sh}})$ since $\un{i}'\leq\un{i}\leq\un{f}-\un{e}^{J^{\sh}}$, which contradicts (\ref{General Eq bound s}). Thus we have $j_0-1\notin J'$.\qed
    
    \hspace{\fill}
    
    Since Claim 2 proves that $j_0-1\notin J'$ and $j_0-1\in J_0$ assuming $j_0\notin J'$ and $j_0\in J_0$, we can continue this process and finally deduce that $J'=\emptyset$, which is a contradiction.
\end{proof}

\begin{remark}\label{General Rk relation 1 general}
    Let $\lambda=(\un{\lambda}_1,\un{\lambda}_2)\in X_1(\un{T})$, $\un{i}\in\NNN^f$ such that $2\un{i}+\un{1}\leq\un{\lambda}_1-\un{\lambda}_2\leq\un{p}-\un{2}$, and $J,J'\subseteq\cJ$. Assume that there exists $j_0\in\cJ$ such that $j_0\in J$, $j_0\notin J'$ and $i_{j_0+1}=0$. We consider the $H$-character 
    \begin{equation*}
        \chi\eqdef\chi_{\lambda}\alpha^{-\un{i}}\scalebox{1}{$\prod\limits_{j\notin J'}$}\alpha_j^{2i_j+p-1-(\lambda_{1,j}-\lambda_{2,j})+\delta_{j-1\in J'}}.
    \end{equation*}
    Then the same proof as in Proposition \ref{General Prop relation 1} shows that (see Remark \ref{General Rk Q123} for the notation) the $H$-character $\chi\alpha_{j'}^{\delta_{J'=\emptyset}}$ does not occur in $Q\bigbra{\chi_{\lambda}^s,\chi_{\lambda}^s\alpha^{\un{i}},J}$ for all $j'\in\cJ$.
\end{remark}

For $J,J'\subseteq\cJ$ and $\un{i}\in\ZZ^f$ such that $\un{0}\leq\un{i}\leq\un{f}-\un{e}^{J^{\sh}}$, we define $\un{m}=\un{m}(\un{i},J,J')\in\ZZ^f$ by 
\begin{equation}\label{General Eq mj}
    m_j\eqdef(-1)^{\delta_{j+1\notin J}}(2i_j+\delta_{j\in(J-1)^{\ss}}-\delta_{j\in J\Delta(J-1)^{\ss}}+\delta_{j-1\in J'}).
\end{equation}
In particular, if $2i_j-\delta_{j\in J\Delta(J-1)^{\ss}}+\delta_{j-1\in J'}\geq0$ for all $j$, then by Lemma \ref{General Lem p001}(iii), $\sigma_{\un{m}}$ is a constituent of $\bang{\GL_2(\OK)\smat{p&0\\0&1}\un{Y}^{-\un{i}}v_J}$. The following proposition is a generalization of Proposition \ref{General Prop relation 1}.

\begin{proposition}\label{General Prop Isigma0m}
    Let $J,J'\subseteq\cJ$, $\un{i}\in\ZZ^f$ such that $\un{0}\leq\un{i}\leq\un{f}-\un{e}^{J^{\sh}}$ and $\un{m}=\un{m}(\un{i},J,J')$. We denote $B\eqdef\bigbbbra{\prod\nolimits_{j\notin J'}Y_j^{2i_j+t^J(J')_j}}\smat{p&0\\0&1}\bbra{\un{Y}^{-\un{i}}v_J}\in\pi$. Then we have (see Lemma \ref{General Lem p001}(iii) for the notation)
    \begin{equation*}
    \begin{cases}
        Y_{j'}^{\delta_{J'=\emptyset}}B=0~\forall\,j'\in\cJ&\text{if $2i_j-\delta_{j\in J\Delta(J-1)^{\ss}}+\delta_{j-1\in J'}<0$ for some $j$}\\
        Y_{j'}{\delta_{J'=\emptyset}}B\in I\bigbra{\sigma_{(J-1)^{\ss}},\sigma_{\un{m}}}~\forall\,j'\in\cJ&\text{if $2i_j-\delta_{j\in J\Delta(J-1)^{\ss}}+\delta_{j-1\in J'}\geq0$ for all $j$}.
    \end{cases}
    \end{equation*}
\end{proposition}

\begin{proof}
    Suppose that $2i_j-\delta_{j\in J\Delta(J-1)^{\ss}}+\delta_{j-1\in J'}<0$ for some $j$, then we must have $i_j=0$, $j\in J\Delta(J-1)^{\ss}$ and $j-1\notin J'$. Hence $Y_{j'}^{\delta_{J'=\emptyset}}B=0$ by Proposition \ref{General Prop relation 1} applied to $(\un{i},J,J')$ as above and $j_0=j-1$.

    Suppose that $2i_j-\delta_{j\in J\Delta(J-1)^{\ss}}+\delta_{j-1\in J'}\geq0$ for all $j$. By Lemma \ref{General Lem p001}(iii) and Remark \ref{General Rk Q123}, we have 
    \begin{equation*}
        \bang{\GL_2(\OK)\smat{p&0\\0&1}\un{Y}^{-\un{i}}v_J}=I\bigbra{\sigma_{(J-1)^{\ss}},\sigma_{\un{b}}}\cong Q\bigbra{\chi_J^s,\chi_J^s\alpha^{\un{i}},(J\Delta(J-1)^{\ss})-1}
    \end{equation*}
    with $b_j=(-1)^{\delta_{j+1\notin J}}(2i_j+\delta_{j\in(J-1)^{\ss}}+1-\delta_{j\in J\Delta(J-1)^{\ss}})$ for $j\in\cJ$. Since $b_j=m_j$ if and only if $j-1\in J'$, to prove $Y_{j'}^{\delta_{J'=\emptyset}}B\in I\bigbra{\sigma_{(J-1)^{\ss}},\sigma_{\un{m}}}$, it suffices to show that for each $j_0\in\cJ$ such that $j_0-1\notin J'$, the image of $Y_{j'}^{\delta_{J'=\emptyset}}B$ in the unique quotient $Q$ of $\bang{\GL_2(\OK)\smat{p&0\\0&1}\un{Y}^{-\un{i}}v_J}$ with socle $\sigma_{\un{e}^{(J-1)^{\ss}\setminus\set{j_0}}+b_{j_0}e_{j_0}}$ is zero. 
    
    By Lemma \ref{General Lem BP12}(i), we have $Q\cong Q\bigbra{\chi_J^s\alpha_{j_0}^{i_{j_0}},\chi_J^s\alpha^{\un{i}},J''}$ with $J''\eqdef\bigbra{(J\Delta(J-1)^{\ss})-1}\cup\set{j_0-1}$.    Since $(\un{i}-i_{j_0}e_{j_0})_{j_0}=0$, $j_0-1\notin J'$ and $j_0-1\in J''$, it follows from Remark \ref{General Rk relation 1 general} (with $\lambda=\lambda_J\alpha_{j_0}^{-i_{j_0}}$, $\un{i}$ replaced with $\un{i}-i_{j_0}e_{j_0}$ and $j_0$ replaced with $j_0-1$) that the $H$-eigencharacter of $Y_{j'}^{\delta_{J'=\emptyset}}B$ does not occur in $Q$, hence $Y_{j'}^{\delta_{J'=\emptyset}}B$ maps to zero in $Q$.
\end{proof}

The following proposition studies the overlaps between different $\GL_2(\OK)$-subrepresentations $\bang{\GL_2(\OK)\smat{p&0\\0&1}\un{Y}^{-\un{i}}v_J}$ of $\pi$. This phenomenon is new in the non-semisimple case.

\begin{proposition}\label{General Prop relation 2}
    Let $J\subseteq\cJ$ and $\un{i}\in\ZZ^f$ such that $\un{0}\leq\un{i}\leq\un{f}-\un{e}^{J^{\sh}}$. Let $j_0\in\cJ$ such that $j_0+1\in(J-1)^{\nss}$ and $i_{j_0+1}=0$. Let $J'\subseteq\cJ$ such that $j_0\in J'$ if $j_0+1\in J$ and $j_0\notin J'$ if $j_0+1\notin J$. We let $J''\eqdef J'\Delta\set{j_0+1}$ and let $\un{i}'\in\ZZ^f$ be such that $i'_j=i_j$ if $j\neq j_0+2$ and $i'_{j_0+2}=i_{j_0+2}-\delta_{j_0+1\notin J'}+\delta_{j_0+2\in(J-1)^{\ss}}$. Then we have
    \begin{multline}\label{General Eq relation 2 statement}
        Y_{j_0+1}^{\delta_{j_0+1\notin J}}\bbbra{\prod\limits_{j\notin J'}Y_j^{2i_j+t^J(J')_j}}\smat{p&0\\0&1}\bbra{\un{Y}^{-\un{i}}v_J}\\
        =\frac{\mu_{J,(J-1)^{\ss}}}{\mu_{J\setminus\set{j_0+2},(J-1)^{\ss}}}Y_{j_0+1}^{\delta_{j_0+1\notin J}}\bbbra{\prod\limits_{j\notin J''}Y_j^{2i'_j+t^{J\setminus\set{j_0+2}}(J'')_j}}\smat{p&0\\0&1}\bbra{\un{Y}^{-\un{i}'}v_{J\setminus\set{j_0+2}}},
    \end{multline}
    where $\mu_{J,(J-1)^{\ss}}$ and $\mu_{J\setminus\set{j_0+2},(J-1)^{\ss}}$ are defined in Proposition \ref{General Prop vector simple}, and we let $\un{Y}^{-\un{i}'}v_{J\setminus\set{j_0+2}}\eqdef0$ if $i'_j<0$ for some $j\in\cJ$.
\end{proposition}

Note that the assumption $j_0+1\in(J-1)^{\nss}$ implies $\bigbra{(J\setminus\set{j_0+2})-1}^{\ss}=(J-1)^{\ss}$, hence $\mu_{J\setminus\set{j_0+2},(J-1)^{\ss}}$ is defined in Proposition \ref{General Prop vector simple}. We claim that $\un{i}'\leq\un{f}-\un{e}^{(J\setminus\set{j_0+2})^{\sh}}$, which implies that $\un{Y}^{-\un{i}'}v_{J\setminus\set{j_0+2}}$ is well-defined by Proposition \ref{General Prop shift}. Indeed, if $j\neq j_0+2$ or $j=j_0+2\notin (J-1)^{\ss}$, then we have
\begin{equation*}
    i'_j\leq i_j\leq f-\delta_{j\in J^{\sh}}\leq f-\delta_{j\in(J\setminus\set{j_0+2})^{\sh}}.
\end{equation*}
If $j_0+2\in(J-1)^{\ss}$, then the assumption $j_0+1\in(J-1)^{\nss}$ implies $j_0+2\in J$ and thus $j_0+2\in J^{\sh}$, hence we have
\begin{equation*}
    i'_{j_0+2}\leq i_{j_0+2}+1\leq f-\delta_{j_0+2\in J^{\sh}}+1=f=f-\delta_{j_0+2\in(J\setminus\set{j_0+2})^{\sh}}.
\end{equation*}
We denote by $B_1$ (resp.\,$B_2$) the element on the LHS (resp.\,RHS) of (\ref{General Eq relation 2 statement}). In order to prove Proposition \ref{General Prop relation 2}, we need the following lemma, and we refer to \S\ref{General Sec app lemmas} for its proof (see Lemma \ref{General Lemma relation 2 appendix}).

\begin{lemma}\label{General Lemma relation 2}
    Keep the assumptions of Proposition \ref{General Prop relation 2}.
\begin{enumerate}
    \item 
    Let $\un{m}\eqdef\un{m}(\un{i},J,J')$ and $\un{m}'\eqdef\un{m}(\un{i}',J\setminus\set{j_0+2},J'')$ (see (\ref{General Eq mj})). Then we have $\un{m}=\un{m}'$ and $m_{j_0+1}=m'_{j_0+1}=0$.
    \item 
    We have (see (\ref{General Eq tJJ'}) for $t^J(J')$)
    \begin{equation}\label{General Eq relation 2 tJ}
    \begin{aligned}
        2i_j+t^J(J')_j&=2i'_j+t^{J\setminus\set{j_0+2}}(J'')_j~\text{if}~j\neq j_0+1;\\
        2i_{j_0+1}+t^J(J')_{j_0+1}&=r_{j_0+1}+1;\\
        2i'_{j_0+1}+t^{J\setminus\set{j_0+2}}(J'')_{j_0+1}&=p-1-r_{j_0+1}.\\
    \end{aligned}        
    \end{equation}
    \item 
    If $2i_j-\delta_{j\in J\Delta(J-1)^{\ss}}+\delta_{j-1\in J'}\geq0$ for all $j\in\cJ$, then $B_1,B_2$ are nonzero and have the same $H$-eigencharacter. 
\end{enumerate}
\end{lemma}

\begin{proof}[Proof of Proposition \ref{General Prop relation 2}]     
    As in the proof of Proposition \ref{General Prop relation 1}, the $H$-eigencharacters that occur in $\bang{\GL_2(\OK)\smat{p&0\\0&1}\un{Y}^{-\un{i}}v_J}$ are those in $Q\bigbra{\chi_J^s\alpha^{\un{i}''},J_{\un{i}''}}$ for $\un{0}\leq\un{i}''\leq\un{i}$ (where $J_{\un{i}''}\eqdef\set{j:j+1\in J\Delta(J-1)^{\ss},~i''_{j+1}=0}$), which are $\chi_J\alpha^{-\un{i}''}\alpha^{-\un{k}}$, where
    \begin{equation}\label{General Eq relation 2 bound}
    \begin{cases}
        0\leq k_j\leq p-2-(s^J_j-2i''_j)+\delta_{j-1\in J_0}&~\text{if}~j\in J_0\\
        p-1-(s^J_j-2i''_j)+\delta_{j-1\in J_0}\leq k_j\leq p-1&~\text{if}~j\notin J_0
    \end{cases}
    \end{equation}
    for $J_0\supseteq J_{\un{i}''}$. By Lemma \ref{General Lem BP12}(iii), unless $J_0=\emptyset$ and $k_j=p-1-(s^J_j-2i''_j)$ for all $j$, the $H$-eigencharacter in (\ref{General Eq relation 2 bound}) comes from the element $\un{Y}^{\un{k}}\smat{p&0\\0&1}\bigbra{\un{Y}^{-\un{i}''}v_J}\in \bigang{\!\GL_2(\OK)\smat{p&0\\0&1}\un{Y}^{-\un{i}''}v_J}$.
    
    Suppose that $2i_j-\delta_{j\in J\Delta(J-1)^{\ss}}+\delta_{j-1\in J'}<0$ for some $j$. By Lemma \ref{General Lemma relation 2}(i) and using $\bigbra{(J\setminus\set{j_0+2})-1}^{\ss}=(J-1)^{\ss}$, we have
    \begin{equation*}
        2i'_j-\delta_{j\in(J\setminus\set{j_0+2})\Delta((J\setminus\set{j_0+2})-1)^{\ss}}+\delta_{j-1\in J''}<0
    \end{equation*}
    for the same $j$. Then by Proposition \ref{General Prop Isigma0m}(i) applied to $(i,J,J')$ and $(i',J\setminus\set{j_0+2},J'')$ we deduce that $B_1=B_2=0$ (if $\un{i}'\ngeq\un{0}$ then $B_2=0$ by definition), which proves (\ref{General Eq relation 2 statement}). So in the rest of the proof we assume that $2i_j-\delta_{j\in J\Delta(J-1)^{\ss}}+\delta_{j-1\in J'}\geq0$ for all $j$, which implies that $$2i'_j-\delta_{j\in(J\setminus\set{j_0+2})\Delta((J\setminus\set{j_0+2})-1)^{\ss}}+\delta_{j-1\in J''}\geq0$$ for all $j$. In particular, this implies $\un{i}'\geq0$. Then by Proposition \ref{General Prop Isigma0m}(ii) applied to $(i,J,J')$ and $(i',J\setminus\set{j_0+2},J'')$ we deduce that $B_1,B_2\in I\bigbra{\sigma_{(J-1)^{\ss}},\sigma_{\un{m}}}=I\bigbra{\sigma_{(J-1)^{\ss}},\sigma_{\un{m}'}}$ (see Lemma \ref{General Lemma relation 2}(i)), which is a subrepresentation of $\bang{\GL_2(\OK)\smat{p&0\\0&1}\un{Y}^{-\un{i}}v_J}$ and of $\bigang{\!\GL_2(\OK)\smat{p&0\\0&1}\un{Y}^{-\un{i}'}v_{J\setminus\set{j_0+2}}}$.
    
    \hspace{\fill}    

    (i). We suppose that $j_0+1\in J$, hence $j_0\in J'$. In this case, we claim that it suffices to prove (\ref{General Eq relation 2 statement}) for $J'=\cJ$, that is (using (\ref{General Eq relation 2 tJ}))
    \begin{equation}\label{General Eq relation 2 J=all}
        \smat{p&0\\0&1}\bbra{\un{Y}^{-\un{i}}v_J}=\frac{\mu_{J,(J-1)^{\ss}}}{\mu_{J\setminus\set{j_0+2},(J-1)^{\ss}}}Y_{j_0+1}^{p-1-r_{j_0+1}}\smat{p&0\\0&1}\bbra{\un{Y}^{-\un{i}'}v_{J\setminus\set{j_0+2}}},
    \end{equation}
    where $i'_j=i_j$ if $j\neq j_0+2$ and $i'_{j_0+2}=i_{j_0+2}+\delta_{j_0+2\in(J-1)^{\ss}}$. Indeed, once (\ref{General Eq relation 2 J=all}) is proved, we multiply both sides of (\ref{General Eq relation 2 J=all}) by $\prod\nolimits_{j\notin J'}Y_j^{2i_j+t^J(J')_j}$. If $j_0+1\in J'$, then using (\ref{General Eq relation 2 tJ}) we obtain (\ref{General Eq relation 2 statement}) for $J'$. If $j_0+1\notin J'$, then using (\ref{General Eq relation 2 tJ}) together with Lemma \ref{General Lem Yj}(i) applied to $j=j_0+1$ we obtain (\ref{General Eq relation 2 statement}) for $J'$.
    
    Then we prove $(\ref{General Eq relation 2 J=all})$. Since $B_1,B_2\in I\bigbra{\sigma_{(J-1)^{\ss}},\sigma_{\un{m}}}\subseteq\bang{\GL_2(\OK)\smat{p&0\\0&1}\un{Y}^{-\un{i}}v_J}$ have common $H$-eigencharacter $\chi_J\alpha^{-\un{i}}$ (see Lemma \ref{General Lemma relation 2}(iii)), it suffices to show that the $H$-eigencharacter $\chi_J\alpha^{-\un{i}}$ only appears once in $\bang{\GL_2(\OK)\smat{p&0\\0&1}\un{Y}^{-\un{i}}v_J}$, which implies $B_1=B_2$ by Lemma \ref{General Lemma relation 2}(iii). Since $j_0+1\in J$, the assumptions $j_0+1\in(J-1)^{\nss}$ and $i_{j_0+1}=0$ imply that $j_0\in J_{\un{i}''}$ for all $\un{i}''$ as in (\ref{General Eq relation 2 bound}). In particular, we have $j_0\in J_0$. As in the proof of Proposition \ref{General Prop relation 1}, the equation (\ref{General Eq relation 1 mod q-1}) then becomes
    \begin{equation}\label{General Eq relation 2 mod q-1}
        -\sum\limits_{j=0}^{f-1}(i_j-i''_j)p^j\equiv\sum\limits_{j=0}^{f-1}k_jp^j~\mod~(q-1),
    \end{equation}
    which is congruent to $\sum\nolimits_{j=0}^{f-1}\bigbra{p-1-(i_j-i''_j)}p^j$ modulo $(q-1)$. If $i_j\neq i''_j$ for some $j$, then we must have $k_{j_0}=p-1-(i_{j_0}-i''_{j_0})$. Since $j_0\in J_0$, by (\ref{General Eq relation 2 bound}) we have
    \begin{equation*}
        k_{j_0}=p-1-(i_{j_0}-i''_{j_0})\leq p-2-(s^J_{j_0}-2i''_{j_0})+\delta_{j_0-1\in J_0}\leq p-1-(s^J_{j_0}-2i''_{j_0}).
    \end{equation*}
    Hence $s^J_{j_0}\leq i_{j_0}+i''_{j_0}\leq2(f-\delta_{j_0\in J^{\sh}})$, which contradicts (\ref{General Eq bound s}).
    Therefore, we must have $i_j=i''_j$ for all $j$ and the LHS of (\ref{General Eq relation 2 mod q-1}) equals $0$. Since $j_0\in J_0$, by (\ref{General Eq relation 2 bound}) and (\ref{General Eq bound s}) we have $k_{j_0}<p-1$. It follows from (\ref{General Eq relation 2 mod q-1}) that $k_j=0$ for all $j$. 

    \hspace{\fill}

    (ii) We suppose that $j_0+1\notin J$ (which implies $f\geq2$), hence $j_0\notin J'$. We prove (\ref{General Eq relation 2 statement}) by the following steps. 

    \hspace{\fill}
    
    \noindent\textbf{Step 1.} We prove (\ref{General Eq relation 2 statement}) for $J'=\cJ\setminus\set{j_0}$.
    
    Using (\ref{General Eq relation 2 tJ}), it is enough to prove that
    \begin{equation}\label{General Eq relation 2 minus j0}
        Y_{j_0}^{2i_{j_0}+p-s^J_{j_0}}\smat{p&0\\0&1}\bbra{\un{Y}^{-\un{i}}v_J}=\frac{\mu_{J,(J-1)^{\ss}}}{\mu_{J\setminus\set{j_0+2},(J-1)^{\ss}}}\bbbra{Y_{j_0}^{2i_{j_0}+p-s^J_{j_0}}Y_{j_0+1}^{p-1-r_{j_0+1}}}\smat{p&0\\0&1}\bbra{\un{Y}^{-\un{i}'}v_{J\setminus\set{j_0+2}}},
    \end{equation}
    where $i'_j=i_j$ if $j\neq j_0+2$ and $i'_{j_0+2}=i_{j_0+2}+\delta_{j_0+2\in(J-1)^{\ss}}$. Since $i_{j_0+1}=i'_{j_0+1}=0$, by Lemma \ref{General Lem Yj}(i) applied to $j=j_0$ and Proposition \ref{General Prop shift}, if we apply $Y_{j_0}^{s^J_{j_0}-2i_{j_0}}$ to either side of (\ref{General Eq relation 2 minus j0}) we get zero. Moreover, $B_1,B_2\in I\bigbra{\sigma_{(J-1)^{\ss}},\sigma_{\un{m}}}\subseteq\bang{\GL_2(\OK)\smat{p&0\\0&1}\un{Y}^{-\un{i}}v_J}$ have common $H$-eigencharacter $\chi\eqdef\chi_J\alpha^{-\un{i}}\alpha_{j_0}^{2i_{j_0}+p-s^J_{j_0}}$ (see Lemma \ref{General Lemma relation 2}(iii)). Hence it suffices to show that up to scalar there exists a unique $H$-eigenvector $C\in\bang{\GL_2(\OK)\smat{p&0\\0&1}\un{Y}^{-\un{i}}v_J}$ satisfying $Y_{j_0}^{s^J_{j_0}-2i_{j_0}}C=0$ with $H$-eigencharacter $\chi$, which implies $B_1=B_2$ by Lemma \ref{General Lemma relation 2}(iii).
    
    As in the proof of Proposition \ref{General Prop relation 1} (in the case $J'=\cJ\setminus\set{j_0}$ with the same definition of $\eta_j$), for each $\un{i}''$ such that $\un{0}\leq\un{i}''\leq\un{i}$, the equation $\chi=\chi_J\alpha^{-\un{i}''}\alpha^{\un{k}}$ has at most one solution for $\un{k}$ as in (\ref{General Eq relation 2 bound}), which is given by (see (\ref{General Eq relation 1 kj}) and since $j_0-1\in J'$)
    \begin{equation}\label{General Eq relation 2 kj}
    \begin{cases}
        k_{j_0}=i_{j_0}+i''_{j_0}+p-s^J_{j_0}+\eta_{j_0}\\
        k_j=\eta_j-(i_j-i''_j)&\text{if}~j\neq j_0.
    \end{cases}
    \end{equation}
    It follows from (\ref{General Eq relation 2 bound}) that $C$ is a linear combination of the elements $C'\eqdef\un{Y}^{\un{k}}\smat{p&0\\0&1}\bigbra{\un{Y}^{-\un{i}''}v_J}\in\bigang{\!\GL_2(\OK)\smat{p&0\\0&1}\un{Y}^{-\un{i}''}v_J}$ with distinct $\un{i}''$ such that $\un{0}\leq\un{i}''\leq\un{i}$ and $\un{k}$ as in (\ref{General Eq relation 2 kj}), each of which has nonzero image in the quotient $Q$ of $\bigang{\!\GL_2(\OK)\smat{p&0\\0&1}\un{Y}^{-\un{i}''}v_J}$ isomorphic to $Q\bigbra{\chi_J^s\alpha^{\un{i}''},J_{\un{i}''}}$ (see Lemma \ref{General Lem p001}(ii)).

    We claim that for $\un{i}''\neq\un{i}$, the element $Y_{j_0}^{s^J_{j_0}-2i_{j_0}}C'\in\bigang{\!\GL_2(\OK)\smat{p&0\\0&1}\un{Y}^{-\un{i}''}v_J}$ also has nonzero image in $Q\cong Q\bigbra{\chi_J^s\alpha^{\un{i}''},J_{\un{i}''}}$. Then we deduce from Lemma \ref{General Lem p001}(ii) that the coefficients of $C'$ with $\un{i}''\neq\un{i}$  in the linear combination for $C$ must be zero, which concludes the proof of (\ref{General Eq relation 2 minus j0}).

    Now we prove the claim. We let $J_0$ be the subset corresponding to the $H$-eigencharacter of $C'=\un{Y}^{\un{k}}\smat{p&0\\0&1}\bigbra{\un{Y}^{-\un{i}''}v_J}$ in (\ref{General Eq relation 2 bound}). Suppose that $j_0\in J_0$. Then by Claim 2 in the proof of Proposition \ref{General Prop relation 1}, we deduce from $j_0\in J_0$ and $j_0\notin J'$ that $j_0-1\notin J'$, which is a contradiction since $J'=\cJ\setminus\set{j_0}$. Hence we must have $j_0\notin J_0$. By the definition of $\eta_j$ in the case $J'=\cJ\setminus\set{j_0}$ (see (\ref{General Eq etaJ})), we have either $\eta_{j_0}=-1$ or $\eta_{j_0}=0$. Moreover, if $\eta_{j_0}=0$, then the definition of $\eta_j$ implies that $i''_j=i_j$ for all $j\neq j_0$, hence $i''_{j_0}<i_{j_0}$ since $\un{i}''\neq\un{i}$. In particular, in either case we deduce from (\ref{General Eq relation 2 kj}) that $k_{j_0}<2i_{j_0}+p-s^J_{j_0}$, hence $(s^J_{j_0}-2i_{j_0})+k_{j_0}\leq p-1$. Then using $j_0\notin J_0$, the $H$-eigencharacter of $Y_{j_0}^{s^J_{j_0}-2i_{j_0}}C'$ still appears in (\ref{General Eq relation 2 bound}) (with the corresponding $\un{i}''$ and $J_0$ unchanged), hence has nonzero image in $Q\cong Q\bigbra{\chi_J^s\alpha^{\un{i}''},J_{\un{i}''}}$.

    \hspace{\fill}
    
    \noindent\textbf{Step 2.} We prove (\ref{General Eq relation 2 statement}) for all $J'$ such that $j_0\notin J'$ and $j_0-1\in J'$. 
    
    We multiply both sides of (\ref{General Eq relation 2 minus j0}) by $Y_{j_0+1}\bigbbbra{\prod\nolimits_{j\notin J'\cup\set{j_0}}Y_j^{2i_j+t^J(J')_j}}$. Since $j_0-1\in J'$, we deduce that $t^J(J')_{j_0}$ is the same as in Step 1. If $j_0+1\in J'$, then using (\ref{General Eq relation 2 tJ}) we obtain (\ref{General Eq relation 2 statement}) for $J'$. If $j_0+1\notin J'$, then using (\ref{General Eq relation 2 tJ}) together with Lemma \ref{General Lem Yj}(i) applied to $j=j_0+1$ we obtain (\ref{General Eq relation 2 statement}) for $J'$.
    
    \hspace{\fill}
    
    \noindent\textbf{Step 3.} We prove (\ref{General Eq relation 2 statement}) for all $J'$ such that $j_0\notin J'$ and $j_0-1\notin J'$. 
    
    We multiply both sides of (\ref{General Eq relation 2 minus j0}) by $Y_{j_0+1}\bigbbbra{\prod\nolimits_{j\notin J'\cup\set{j_0}}Y_j^{2i_j+t^J(J')_j}}$. Similarly to Step 2 but using $j_0-1\notin J'$,  we get $Y_{j_0}B_1=Y_{j_0}B_2$. Moreover, $B_1,B_2\in I\bigbra{\sigma_{(J-1)^{\ss}},\sigma_{\un{m}}}\subseteq\bang{\GL_2(\OK)\smat{p&0\\0&1}\un{Y}^{-\un{i}}v_J}$ have $H$-eigencharacter $\chi\eqdef\chi_J\alpha_{j_0+1}\alpha^{-\un{i}}\bigbbbra{\prod\nolimits_{j\notin J'}\alpha_j^{2i_j+p-1-s^J_j+\delta_{j-1\in J'}}}$. Hence it suffices to show that there is no nonzero $H$-eigenvector $C\in I(\sigma_{(J-1)^{\ss}},\sigma_{\un{m}})$ with $H$-eigencharacter $\chi$ satisfying $Y_{j_0}C=0$, which implies $B_1-B_2=0$.

    The rest of the proof is similar to the one of Step 1, and is left as an exercise. Here the analogous assertion $j_0\notin J_0$ is guaranteed by the fact that $C\in I(\sigma_{(J-1)^{\ss}},\sigma_{\un{m}})$ and using Lemma \ref{General Lem change origin}.
\end{proof}

The following Proposition is a generalization of Proposition \ref{General Prop vector simple} and gives more relations between the vectors $v_J\in D_0(\rhobar)$.

\begin{proposition}\label{General Prop vector}
    Let $J,J'\subseteq\cJ$ such that $(J')^{\nss}\neq\cJ$ (i.e.\,$(J',J_{\rhobar})\neq(\cJ,\emptyset)$) and satisfying
    \begin{equation}\label{General Eq good pair}
    \begin{cases}
        (J-1)^{\ss}=(J')^{\ss}\\
        (J')^{\nss}\subseteq(J-1)^{\nss}\Delta(J'-1)^{\nss}.
    \end{cases}
    \end{equation}
    Then there exists a unique element $\mu_{J,J'}\in\FF\x$, such that
    \begin{equation}\label{General Eq Prop vector statement}
        \bbbra{\prod\limits_{j+1\in J\Delta J'}Y_j^{s^{J'}_j}\prod\limits_{j+1\notin J\Delta J'}Y_j^{p-1}}\smat{p&0\\0&1}\bbra{\un{Y}^{-\un{e}^{(J\cap J')^{\nss}}}v_J}=\mu_{J,J'}v_{J'},
    \end{equation}
    where $\un{Y}^{-\un{e}^{(J\cap J')^{\nss}}}v_J$ is defined in Proposition \ref{General Prop shift}. 
\end{proposition}

\begin{proof}
    If $f=1$, then the assumption implies $J'=(J-1)^{\ss}$ and $(J\cap J')^{\nss}=\emptyset$, and the proposition is already proved in Proposition \ref{General Prop vector simple}. Hence in the rest of the proof we assume that $f\geq2$. We denote by $B$ the LHS of (\ref{General Eq Prop vector statement}).

    \hspace{\fill}

    \noindent\textbf{Claim 1.} The element $B$ is nonzero and has $H$-eigencharacter $\chi_{J'}$.
    
    \proof By Lemma \ref{General Lem p001}(ii), the representation $\bigang{\!\GL_2(\OK)\smat{p&0\\0&1}\un{Y}^{-\un{e}^{(J\cap J')^{\nss}}}v_J}$ has a quotient $Q$ isomorphic to $Q\bigbra{\chi_J^s\alpha^{\un{e}^{(J\cap J')^{\nss}}},J'''}$ with $J'''\eqdef\bigbra{\!\bbra{J\Delta(J-1)^{\ss}}\setminus(J\cap J')^{\nss}}-1$. By the proof of Lemma \ref{General Lem p001}, $Q$ has constituents $F\bigbra{\ft_{\lambda_J}(-\un{b})}$ with
    \begin{equation}\label{General Eq relation bj1}
        \max\bbra{\delta_{j\in J\Delta(J-1)^{\ss}},2\delta_{j\in(J\cap J')^{\nss}}}\leq b_j\leq2\delta_{j\in(J\cap J')^{\nss}}+1.
    \end{equation}
    
    We claim that $\sigma_{J'}$ is a constituent of $Q$ and corresponds to the subset $J''\eqdef(J\Delta J')-1$ for $\Ind_I^{\GL_2(\OK)}\bigbra{\chi_J^s\alpha^{\un{e}^{(J\cap J')^{\nss}}}}$ (see Lemma \ref{General Lem BP12}(i)). Indeed, by Lemma \ref{General Lem change origin}(iii), we have $\sigma_{J'}=F\bigbra{\ft_{\lambda_J}(-\un{b})}$ with 
    \begin{equation}\label{General Eq relation bj2}
        b_j=
    \begin{cases}
        \delta_{j\in J}+\delta_{j\in J'}(-1)^{\delta_{j+1\notin J\Delta J'}}&\text{if}~j\notin J_{\rhobar}\\
        \bbra{\delta_{j\in J}-\delta_{j\notin J'}}(-1)^{\delta_{j+1\in J}}&\text{if}~j\in J_{\rhobar}.
    \end{cases}
    \end{equation}
    We need to check that $b_j$ satisfies (\ref{General Eq relation bj1}). We assume that $j\notin J_{\rhobar}$, the case $j\in J_{\rhobar}$ being similar. By (\ref{General Eq good pair}), we have $j\in J'$ implies $j+1\in J\Delta J'$. Hence we have $\delta_{j\in J'}(-1)^{\delta_{j+1\notin J\Delta J'}}=\delta_{j\in J'}$, and from (\ref{General Eq relation bj1}) it suffices to show that
    \begin{equation*}
        \max\bbra{\delta_{j\in J},2\delta_{j\in J\cap J'}}\leq\delta_{j\in J}+\delta_{j\in J'}\leq2\delta_{j\in J\cap J'}+1,
    \end{equation*}
    which is easy. Then we prove the second assertion. By Lemma \ref{General Lem BP12}(i) applied to $\lambda=\lambda_J-\alpha^{\un{e}^{(J\cap J')^{\nss}}}$, it suffices to show that $b_j=2\delta_{j\in(J\cap J')^{\nss}}+1$ if and only if $j\in J\Delta J'$. Once again we assume that $j\notin J_{\rhobar}$, and the case $j\in J_{\rhobar}$ is similar. Then it suffices to show that $\delta_{j\in J}+\delta_{j\in J'}=2\delta_{j\in J\cap J'}+1$ if and only if $j\in J\Delta J'$, which is easy.
    
    By Lemma \ref{General Lem p-2-s and s}, for $j\in J''$ we have $\bigbra{p-2-(s^J_j-2\delta_{j\in(J\cap J')^{\nss}})}+\delta_{j-1\in J''}=s^{J'}_j$. Then by Lemma \ref{General Lem BP12}(iii) applied to $\lambda=\lambda_J-\alpha^{\un{e}^{(J\cap J')^{\nss}}}$, we deduce that the image of $B$ in $Q\cong Q\bigbra{\chi_J^s\alpha^{\un{e}^{(J\cap J')^{\nss}}},J'''}$ is a nonzero $I_1$-invariant of $\sigma_{J'}$. In particular, $B$ is nonzero and has $H$-eigencharacter $\chi_{J'}$.\qed

    \hspace{\fill}

    \noindent\textbf{Claim 2.} The element $B$ is $K_1$-invariant.
    
    \proof First we claim that $\un{m}\bigbra{\un{e}^{(J\cap J')^{\nss}},J,J''}=\un{a}^{J'}$ (see (\ref{General Eq mj}) for $\un{m}$ and (\ref{General Eq aJ}) for $\un{a}^{J'}$). Indeed, using Lemma \ref{General Lem m} it suffices to show that $\delta_{j\in J'}(-1)^{\delta_{j+1\notin J}}=a^{J'}_j$ for $j\in\cJ$. If $j\in J_{\rhobar}$, then the assumption $(J-1)^{\ss}=(J')^{\ss}$ implies that $j\in J-1$ if and only if $j\in J'$, hence we have $\delta_{j\in J'}(-1)^{\delta_{j+1\notin J}}=\delta_{j\in J'}(-1)^{\delta_{j\notin J'}}=\delta_{j\in J'}$, which equals $a^{J'}_j$ by (\ref{General Eq aJ}). If $j\notin J_{\rhobar}$, then $j\in J'$ implies $j+1\in J\Delta J'$ by the second formula of (\ref{General Eq good pair}), hence we have $\delta_{j\in J'}(-1)^{\delta_{j+1\notin J}}=\delta_{j\in J'}(-1)^{\delta_{j+1\in J'}}$, which equals $a^{J'}_j$ by (\ref{General Eq aJ}).
    
    By Proposition \ref{General Prop Isigma0m}(ii) applied to $\un{i}=\un{e}^{(J\cap J')^{\nss}}$ and with $J'$ there being $J''$, we deduce that 
    \begin{equation}\label{General Eq relation Isigma}
        Y_{j'}\bbbra{\scalebox{1}{$\prod$}_{j+1\notin J\Delta J'}Y_j^{2\delta_{j\in(J\cap J')^{\nss}}+t^J(J'')_j}}\smat{p&0\\0&1}\bbra{\un{Y}^{-\un{e}^{(J\cap J')^{\nss}}}v_J}\in I(\sigma_{(J-1)^{\ss}},\sigma_{J'})\subseteq D_{0,\sigma_{(J-1)^{\ss}}}(\rhobar)
    \end{equation}
    for all $j'\in\cJ$, where the last inclusion follows from the fact that $\sigma_{J'}$ is a constituent of $D_{0,\sigma_{(J-1)^{\ss}}}(\rhobar)$ (which follows from Lemma \ref{General Lem Diamond diagram}(iii) and (\ref{General Eq good pair})). Since $s^{J'}_j\geq1$ and $2\delta_{j\in(J\cap J')^{\nss}}+t^J(J'')_j\leq p-2$ for all $j$ by (\ref{General Eq bound s}), (\ref{General Eq bound tJ}) and $f\geq2$, multiplying (\ref{General Eq relation Isigma}) by a suitable power of $\un{Y}$ we deduce that $B\in D_{0,\sigma_{(J-1)^{\ss}}}(\rhobar)$, hence is $K_1$-invariant.\qed

    \hspace{\fill}

    \noindent\textbf{Claim 3.} We have $Y_{j_0}B=0$ for all $j_0\in\cJ$. 
    
    \proof (i). Suppose that $j_0+1\notin J\Delta J'$ and $j_0+1\notin (J\cap J')^{\nss}$. By Proposition \ref{General Prop shift}, we have $Y_{j_0+1}\bigbra{\un{Y}^{-\un{e}^{(J\cap J')^{\nss}}}v_J}=0$. Hence it  follows from Lemma \ref{General Lem Yj}(i) applied to $j=j_0$ that $Y_{j_0}B=0$.
    
    (ii). Suppose that $j_0+1\in J\Delta J'$, which equals $J\Delta\bigbra{(J')^{\ss}\Delta(J')^{\nss}}=\bigbra{J\Delta(J')^{\ss}}\Delta(J')^{\nss}$. Hence for each $j\in J\Delta J'$, we have either $j\in J\Delta(J')^{\ss}$, $j\notin(J')^{\nss}$ or $j\notin J\Delta(J')^{\ss}$, $j\in(J')^{\nss}$, and in the latter case we have $j+1\in J\Delta J'$ by (\ref{General Eq good pair}). In particular, since $(J')^{\nss}\neq\cJ$, there exists $0\leq w\leq f-1$ such that $j\notin J\Delta(J')^{\ss}$, $j\in(J')^{\nss}$ for $j=j_0+1,\ldots,j_0+w$ and $j_0+w+1\in J\Delta(J')^{\ss}$, $j_0+w+1\notin(J')^{\nss}$. 
    
    By (\ref{General Eq good pair}) we have $j_0+w+1\in J\Delta(J')^{\ss}=J\Delta(J-1)^{\ss}$. Then by proposition \ref{General Prop relation 1} applied to $\un{i}=\un{e}^{(J\cap J')^{\nss}}$, $j_0$ replaced by $j_0+w$ and $J'$ replaced by $J''-1$ with $J''\eqdef(J\Delta J')\setminus\set{j_0+1,\ldots,j_0+w+1}$, and possibly multiplying (\ref{General Eq relation 1 statement}) by $Y_{j_0+w+1}$, we have
    \begin{equation*}
        Y_{j_0+w+1}\bbbra{\scalebox{1}{$\prod$}_{j+1\notin J''}Y_j^{2\delta_{j\in(J\cap J')^{\nss}}+p-1-s^J_j+\delta_{j\in J''}}}\smat{p&0\\0&1}\bbra{\un{Y}^{-\un{e}^{(J\cap J')^{\nss}}}v_J}=0.
    \end{equation*}
    Since $2\delta_{j\in(J\cap J')^{\nss}}+p-1-s^J_j+\delta_{j\in J''}\leq p-1$ for all $j$ by (\ref{General Eq bound s}), to prove that $Y_{j_0}B=0$, it suffices (from the formula of $B$) to show that 
    \begin{equation*}
        s^{J'}_j+\delta_{j=j_0}=2\delta_{j\in(J\cap J')^{\nss}}+p-1-s^J_j+\delta_{j\in(J\Delta J')\setminus\set{j_0+1,\ldots,j_0+w+1}}+\delta_{j=j_0+w+1}
    \end{equation*}
    for $j+1\in(J\Delta J')\setminus J''$, that is $j=j_0,\ldots,j_0+w$. This follows from Lemma \ref{General Lem p-2-s and s} with $J,J'$ as above noting that $j=j_0+w+1$ and $j\in\set{j_0,\ldots,j_0+w}$ imply that $j=j_0$ and $w=f-1$. 
    
    (iii). Suppose that $j_0+1\in (J\cap J')^{\nss}$, then by Lemma \ref{General Lem Yj}(i) applied to $j=j_0$ and using $j_0+1\notin J\Delta J'$, we have 
    \begin{equation}\label{General Eq relation Yj0B}
        Y_{j_0}B=\bbbra{\scalebox{1}{$\prod$}_{j+1\in J\Delta J'}Y_j^{s^{J'}_j}\scalebox{1}{$\prod$}_{j+1\notin(J\Delta J')\cup\set{j_0+1}}Y_j^{p-1}}\smat{p&0\\0&1}\bbra{\un{Y}^{-\un{e}^{(J\cap J')^{\nss}\setminus\set{j_0+1}}}v_J}.
    \end{equation}
    As in (ii) (with the difference that $j_0+1\notin J\Delta J'$), there exists $1\leq w\leq f-1$ such that $j\notin J\Delta(J')^{\ss}$, $j\in(J')^{\nss}$ for $j=j_0+2,\ldots,j_0+w$ and $j_0+w+1\in J\Delta(J')^{\ss}$, $j_0+w+1\notin(J')^{\nss}$. The rest of the proof then follows exactly as in (ii) except that we take $\un{i}=\un{e}^{(J\cap J')^{\nss}\setminus\set{j_0+1}}$ and $J''=\bigbra{(J\Delta J')\setminus\set{j_0+2,\ldots,j_0+w+1}}\cup\set{j_0+1}$.\qed

    \hspace{\fill}

    From Claim 2 and Claim 3 we deduce that $B$ is $I_1$-invariant. Since $B\neq0$ has $H$-eigencharacter $\chi_{J'}$ by Claim 1 and since $D_0(\rhobar)^{I_1}$ is multiplicity-free by Lemma \ref{General Lem Diamond diagram}(ii), we conclude that $B$ is a scalar multiple of $v_{J'}$, which completes the proof.
\end{proof}

\begin{remark}\label{General Rk Delta and partial}
    For $J\subseteq\cJ$, we define the \textbf{right boundary} of $J$ by
    $\partial J\eqdef\sset{j\in J:j+1\notin J}.$
    Then the second formula in (\ref{General Eq good pair}) is equivalent to 
    \begin{equation*}
        (\partial J')^{\nss}\subseteq(J-1)^{\nss}\subseteq\bigbra{(J'\setminus\partial J')^c}^{\nss}.
    \end{equation*}
\end{remark}

If $J_{\rhobar}=\emptyset$, then we define $x_{\emptyset,\un{r}}\eqdef\mu_{\emptyset,\emptyset}\inv\un{Y}^{\un{p}-\un{1}-\un{r}}\smat{p&0\\0&1}v_{\emptyset}$ so that $\un{Y}^{\un{r}}x_{\emptyset,\un{r}}=v_{\emptyset}$ by (\ref{General Eq vector simple}) applied to $J=\emptyset$. This agrees with the definition of $x_{\emptyset,\un{r}}$ given in Theorem \ref{General Thm seq} below (see (\ref{General Eq explicit sequence})). Then we have the following complement of Proposition \ref{General Prop vector}, which together with Proposition \ref{General Prop vector} gives all possible relations between the vectors $v_J\in D_0(\rhobar)$.

\begin{proposition}\label{General Prop vector complement}
    Assume that $J_{\rhobar}=\emptyset$. Then for $\emptyset\neq J\subseteq\cJ$, there exists a unique element $\mu_{J,\cJ}\in\FF\x$ such that
    \begin{equation*}
        \bbbra{\prod\limits_{j+1\notin J}Y_j^{p-1-r_j}}\smat{p&0\\0&1}v_J=\mu_{J,\cJ}v_{\cJ}+\mu_{J,\emptyset}x_{\emptyset,\un{r}},
    \end{equation*}
    where $\mu_{J,\emptyset}$ is defined in Proposition \ref{General Prop vector simple}.
\end{proposition}

\begin{proof}
    By Lemma \ref{General Lem p001}(ii) and its proof, the isomorphism $\Ind_I^{\GL_2(\OK)}(\chi_{\emptyset}^s)\cong\bang{\GL_2(\OK)\smat{p&0\\0&1}v_{\emptyset}}$ identifies the element $\phi\in\Ind_I^{\GL_2(\OK)}(\chi_{\emptyset}^s)$ (see \S\ref{General Sec ps}) with $\smat{0&1\\p&0}v_{\emptyset}$, which is a scalar multiple of $v_{\cJ}$ since $\chi_{\cJ}=\chi_{\emptyset}^s$ when $J_{\rhobar}=\emptyset$. Hence by Lemma \ref{General Lem BP12}(iii)(a) applied to $\lambda=\lambda_{\emptyset}$, any nonzero element in the $I$-cosocle of $\sigma_{\emptyset}$ is a linear combination of $v_{\cJ}$ and $x_{\emptyset,\un{r}}$ with nonzero coefficients.

    By Lemma \ref{General Lem p001}(i),(ii) and its proof, the representation $\bang{\GL_2(\OK)\smat{p&0\\0&1}v_J}\cong Q\bigbra{\chi_J^s,J-1}$ has socle $\sigma_{\emptyset}$, and identifies $\smat{0&1\\p&0}v_J$ with the image of $\phi\in\Ind_I^{\GL_2(\OK)}(\chi_{J}^s)$ in $Q\bigbra{\chi_J^s,J-1}$. Since $J\neq\emptyset$, we deduce from Lemma \ref{General Lem BP12}(iii)(b) applied to $\lambda=\lambda_J$ that the element $B\eqdef\bigbbbra{\prod\nolimits_{j+1\notin J}Y_j^{p-1-r_j}}\smat{p&0\\0&1}v_J$ is nonzero and lies in the $I$-cosocle of $\sigma_{\emptyset}$, hence $B=\mu_{J,\cJ}v_{\cJ}+\mu'_{J,\emptyset}x_{\emptyset,\un{r}}$ for some $\mu_{J,\cJ},\mu'_{J,\emptyset}\in\FF\x$ by the previous paragraph. Finally, by applying $\un{Y}^{\un{r}}$ to $B$ and using $\un{Y}^{\un{r}}v_{\cJ}=0$ since $v_{\cJ}$ is $I_1$-invariant, we deduce from Proposition \ref{General Prop vector simple} (with $J_{\rhobar}=J=\emptyset$) that $\mu'_{J,\emptyset}=\mu_{J,\emptyset}$.
\end{proof}

\begin{lemma}
    Let $J_1,J_2,J_3,J_4\subseteq\cJ$ such that the pairs $(J_1,J_3),(J_1,J_4),(J_2,J_3),(J_2,J_4)$ satisfy the assumptions of either Proposition \ref{General Prop vector} or Proposition \ref{General Prop vector complement} (here we say that $(J,J')$ satisfies the assumption of Proposition \ref{General Prop vector complement} if $J_{\rhobar}=\emptyset$, $J\neq\emptyset$ and $J'=\cJ$). Then we have
    \begin{equation}\label{General Eq ratio}
        \frac{\mu_{J_1,J_3}}{\mu_{J_1,J_4}}=\frac{\mu_{J_2,J_3}}{\mu_{J_2,J_4}},
    \end{equation}
    where each term of (\ref{General Eq ratio}) is defined in either Proposition \ref{General Prop vector} or Proposition \ref{General Prop vector complement}.
\end{lemma}

\begin{proof}
    First we suppose that $J_{\rhobar}=\emptyset$ and $J_4=\cJ$. If $J_3=\cJ$, then (\ref{General Eq ratio}) is clear. If $J_3\neq\cJ$, then by the proof of Proposition \ref{General Prop vector complement} and using that the $I$-cosocle of $\sigma_{\emptyset}$ has dimension $1$ over $\FF$, the ratio $\mu_{J,\cJ}/\mu_{J,\emptyset}$ does not depend on $J$, hence we can replace $J_4=\cJ$ by $J_4=\emptyset$.

    From now on, we assume that all the pairs $(J_1,J_3),(J_1,J_4),(J_2,J_3),(J_2,J_4)$ satisfy the assumptions of Proposition \ref{General Prop vector}. In particular, we have $(J_1-1)^{\ss}=(J_2-1)^{\ss}=J_3^{\ss}=J_4^{\ss}$. Using that $\mu_{J_i,J_3}/\mu_{J_i,J_4}=\bra{\mu_{J_1,J_3}/\mu_{J_i,J_4^{\ss}}}/\bra{\mu_{J_i,J_4}/\mu_{J_i,J_4^{\ss}}}$
    for $i=1,2$ with each term defined in Proposition \ref{General Prop vector}, we may assume that $J_4=J_4^{\ss}\subseteq J_{\rhobar}$.

    Then using Remark \ref{General Rk Delta and partial}, the assumption (\ref{General Eq good pair}) for the pairs $(J_1,J_3),(J_1,J_4),(J_2,J_3),(J_2,J_4)$ is equivalent to 
    \begin{equation}\label{General Eq ratio region J1J2}
        (\partial J_3)^{\nss}\sqcup J_4\subseteq J_i-1\subseteq\bigbra{(J_3\setminus\partial J_3)^c}^{\nss}\sqcup J_4
    \end{equation}
    for $i=1,2$. By choosing a sequence $J^0,J^1,\ldots,J^r\subseteq\cJ$ for some $r\geq0$ such that $J^0=J_1$, $J^r=J_2$, $|J^i\Delta J^{i-1}|=1$ for $1\leq i\leq r$ and
    \begin{equation*}
        (\partial J_3)^{\nss}\sqcup J_4\subseteq J^i-1\subseteq\bigbra{(J_3\setminus\partial J_3)^c}^{\nss}\sqcup J_4
    \end{equation*}
    for $0\leq i\leq r$, it suffices to prove the proposition with $J_1,J_2$ as in (\ref{General Eq ratio region J1J2}) such that $(J_1-1)=(J_2-1)\sqcup\set{j_0+1}$ for some $j_0\in\cJ$. In particular, we have $j_0+1\in(J_1-1)^{\nss}$ (since $(J_1-1)^{\ss}=(J_2-1)^{\ss}=J_4$) and $j_0+1\notin J_3$. 

    \hspace{\fill}

    (i). Suppose that $j_0\in J_1-1$. By Proposition \ref{General Prop vector} applied to $(J_1,J_3)$, we have
    \begin{equation}\label{General Eq ratio J1}
        \bbbra{\scalebox{1}{$\prod$}_{j+1\in J_1\Delta J_3}Y_j^{s^{J_3}_j}\scalebox{1}{$\prod$}_{j+1\notin J_1\Delta J_3}Y_j^{p-1}}\smat{p&0\\0&1}\bbra{\un{Y}^{-\un{e}^{(J_1\cap J_3)^{\nss}}}v_{J_1}}=\mu_{J_1,J_3}v_{J_3}.
    \end{equation}
    Since $j_0+1\in J_1$ and $j_0+1\notin(J_1\cap J_3)^{\nss}$ (since $j_0+1\notin J_3$), by Proposition \ref{General Prop relation 2} applied to $(\un{i},J,J')=\bigbra{\un{e}^{(J_1\cap J_3)^{\nss}},J_1,\cJ}$ with $j_0$ as above together with (\ref{General Eq relation 2 tJ}) (and note that $(J_1-1)^{\ss}=J_4$), we have
    \begin{equation}\label{General Eq ratio J1-J2}
        \smat{p&0\\0&1}\bbra{\un{Y}^{-\un{e}^{(J_1\cap J_3)^{\nss}}}v_{J_1}}=\frac{\mu_{J_1,J_4}}{\mu_{J_2,J_4}}Y_{j_0+1}^{p-1-r_{j_0+1}}\smat{p&0\\0&1}\bbra{\un{Y}^{-\un{i}'}v_{J_2}}
    \end{equation}
    with
    \begin{multline}\label{General Eq ratio i'}
        \un{i}'\eqdef\un{e}^{(J_1\cap J_3)^{\nss}}+\delta_{j_0+2\in J_4}e_{j_0+2}\\
        =\bigbra{\un{e}^{(J_2\cap J_3)^{\nss}}+\delta_{j_0+2\in J_3^{\nss}}e_{j_0+2}}+\delta_{j_0+2\in J_3^{\ss}}e_{j_0+2}=\un{e}^{(J_2\cap J_3)^{\nss}}+\delta_{j_0+2\in J_3}e_{j_0+2},
    \end{multline}
    where the second equality uses $J_1\setminus J_2=\set{j_0+2}$  and $J_3^{\ss}=J_4$. We assume that $j_0+2\in J_3$, the case $j_0+2\notin J_3$ being similar. Since $j_0+1\notin J_3$, we have $s^{J_3}_{j_0+1}=p-2-r_{j_0+1}$ by (\ref{General Eq sJ}). Combining (\ref{General Eq ratio J1}) and (\ref{General Eq ratio J1-J2}), we have
    \begin{align*}
        \mu_{J_1,J_3}v_{J_3}&=\frac{\mu_{J_1,J_4}}{\mu_{J_2,J_4}}\bbbra{\scalebox{1}{$\prod$}_{j+1\in J_1\Delta J_3}Y_j^{s^{J_3}_j}\scalebox{1}{$\prod$}_{j+1\notin J_1\Delta J_3}Y_j^{p-1}}Y_{j_0+1}^{p-1-r_{j_0+1}}\smat{p&0\\0&1}\bbra{\un{Y}^{-\un{i}'}v_{J_2}}\\
        &=\frac{\mu_{J_1,J_4}}{\mu_{J_2,J_4}}\bbbra{\scalebox{1}{$\prod$}_{j+1\in J_2\Delta J_3}Y_j^{s^{J_3}_j}\scalebox{1}{$\prod$}_{j+1\notin J_2\Delta J_3}Y_j^{p-1}}Y_{j_0+1}^{r_{j_0+1}}Y_{j_0+1}^{p-1-r_{j_0+1}}\smat{p&0\\0&1}\bbra{\un{Y}^{-\un{i}'}v_{J_2}}\\
        &=\frac{\mu_{J_1,J_4}}{\mu_{J_2,J_4}}\bbbra{\scalebox{1}{$\prod$}_{j+1\in J_2\Delta J_3}Y_j^{s^{J_3}_j}\scalebox{1}{$\prod$}_{j+1\notin J_2\Delta J_3}Y_j^{p-1}}\smat{p&0\\0&1}\bbra{\un{Y}^{-\un{e}^{(J_2\cap J_3)^{\nss}}}v_{J_2}}\\
        &=\frac{\mu_{J_1,J_4}}{\mu_{J_2,J_4}}\mu_{J_2,J_3}v_{J_3},
    \end{align*}
    where the second equality uses $j_0+2\in J_1$, $j_0+2\notin J_2$, $j_0+2\in J_3$ (hence $j_0+2\notin J_1\Delta J_3$ and $j_0+2\in J_2\Delta J_3$) and $s^{J_3}_{j_0+1}=p-2-r_{j_0+1}$, the third equality follows from Lemma \ref{General Lem Yj}(i) applied to $j=j_0+1$ and (\ref{General Eq ratio i'}) using $j_0+2\in J_3$, and the last equality follows from Proposition \ref{General Prop vector} applied to $(J_2,J_3)$. Therefore, we have $\mu_{J_1,J_3}=(\mu_{J_1,J_4}/\mu_{J_2,J_4})\mu_{J_2,J_3}$, which completes the proof.
    
    (ii). Suppose that $j_0\notin J_1-1$ (which implies $f\geq2$). Similar to (i), by Proposition \ref{General Prop relation 2} applied to $(\un{i},J,J')=\bigbra{\un{e}^{(J_1\cap J_3)^{\nss}},J_1,\cJ\setminus\set{j_0}}$ with $j_0$ as above together with (\ref{General Eq relation 2 tJ}), we have
    \begin{multline*}
        \bbbra{Y_{j_0+1}Y_{j_0}^{2\delta_{j_0\in(J_1\cap J_3)^{\nss}}+p-s^{J_1}_{j_0}}}\smat{p&0\\0&1}\bbra{\un{Y}^{-\un{i}}v_J}\\
        =\frac{\mu_{J_1,J_4}}{\mu_{J_2,J_4}}\bbbra{Y_{j_0+1}Y_{j_0}^{2\delta_{j_0\in(J_1\cap J_3)^{\nss}}+p-s^{J_1}_{j_0}}}Y_{j_0+1}^{p-1-r_{j_0+1}}\smat{p&0\\0&1}\bbra{\un{Y}^{-\un{i}'}v_{J\setminus\set{j_0+2}}},
    \end{multline*}
    where $\un{i}'=\un{e}^{(J_2\cap J_3)^{\nss}}+\delta_{j_0+2\in J_3}e_{j_0+2}$. We claim that $\bigbbbra{\prod\nolimits_{j+1\in J_1\Delta J_3}Y_j^{s^{J_3}_j}\prod\nolimits_{j+1\notin J_1\Delta J_3}Y_j^{p-1}}$ is a multiple of $Y_{j_0+1}Y_{j_0}^{2\delta_{j_0\in(J_1\cap J_3)^{\nss}}+p-s^{J_1}_{j_0}}$. Indeed, since $s^{J_3}_{j_0+1}\geq1$ and $2\delta_{j_0\in(J_1\cap J_3)^{\nss}}+p-s^{J_1}_{j_0}\leq p-1$ by (\ref{General Eq bound s}), the claim follows from the fact that $j_0+1\notin J_1$, $j_0+1\notin J_3$ (hence $j_0+1\notin J_1\Delta J_3$) and $j_0+1\neq j_0$. Once we have the claim, we can argue exactly as in (i) to conclude the proof. 
\end{proof}

To end this section, we extend the definition of $\mu_{J,J'}$ to all $J,J'\subseteq\cJ$ such that $(J-1)^{\ss}=(J')^{\ss}$ as follows:
\begin{equation}\label{General Eq general mu}
    \begin{cases}
        \mu_{J,J'}\eqdef\dfrac{\mu_{((J')^{\ss}\sqcup(\partial J')^{\nss})+1,J'}\mu_{J,(J')^{\ss}}}{\mu_{((J')^{\ss}\sqcup(\partial J')^{\nss})+1,(J')^{\ss}}}&\text{if}~(J_{\rhobar},J')\neq(\emptyset,\cJ)\\
        \mu_{\emptyset,\cJ}\eqdef\dfrac{\mu_{\emptyset,\emptyset}\mu_{\cJ,\cJ}}{\mu_{J,\emptyset}}&\text{if}~J_{\rhobar}=\emptyset
    \end{cases}
\end{equation}
(and $\mu_{J,\cJ}$ as in Proposition \ref{General Prop vector complement} if $J_{\rhobar}=\emptyset$ and $J\neq\emptyset$), where each term on the RHS of (\ref{General Eq general mu}) are defined in either Proposition \ref{General Prop vector} or Proposition \ref{General Prop vector complement}. Then the equation (\ref{General Eq ratio}) holds for arbitrary $J_1,J_2,J_3,J_4\subseteq\cJ$ such that $(J_1-1)^{\ss}=(J_2-1)^{\ss}=J_3^{\ss}=J_4^{\ss}$. In particular, for $J,J'$ such that $(J-1)^{\ss}=(J'-1)^{\ss}$, the quantity $\mu_{J,J''}/\mu_{J',J''}$ does not depend on $J''$ for $(J'')^{\ss}=(J-1)^{\ss}$ and we denote it by $\mu_{J,*}/\mu_{J',*}$. Similarly, for $J,J'$ such that $J^{\ss}=(J')^{\ss}$, the quantity $\mu_{J'',J}/\mu_{J'',J'}$ does not depend on $J''$ for $(J''-1)^{\ss}=J^{\ss}$ and we denote it by $\mu_{*,J}/\mu_{*,J'}$.

\section{Projective systems in \texorpdfstring{$\pi$}.}\label{General Sec xJi}

In this section, we define certain projective systems $x_{J,\un{i}}$ of elements of $\pi$ indexed by $J\subseteq\cJ$ and $\un{i}\in\ZZ^f$, see Theorem \ref{General Thm seq}. They will give rise to a basis of the $A$-module $D_A(\pi)$. The definition of these elements is much more involved than in the semisimple case (compared with \cite[(104)]{BHHMS3}).

\begin{definition}\label{General Def rcepi}
    Let $J\subseteq\cJ$. 
    \begin{enumerate}
    \item 
    We define $\un{r}^J\in\ZZ^f$ by 
    \begin{equation}\label{General Eq rJ}
        r^J_j\eqdef
    \begin{cases}
        0&\text{if}~j\notin J,\ j+1\notin J\\
        -1&\text{if}~j\in J,\ j+1\notin J\\
        r_j+1&\text{if}~j\notin J,\ j+1\in J\\
        r_j&\text{if}~j\in J,\ j+1\in J.
    \end{cases}
    \end{equation}
    \item 
    We define $\un{c}^{J}\in\ZZ^f$ by
    \begin{equation}\label{General Eq cJ}
        c^J_j\eqdef
    \begin{cases}
        p-1&\text{if}~j\notin J,\ j+1\notin J\\
        r_j+1&\text{if}~j\in J,\ j+1\notin J\\
        p-2-r_j&\text{if}~j\notin J,\ j+1\in J\\
        0&\text{if}~j\in J,\ j+1\in J.
    \end{cases}
    \end{equation}
    \item 
    We define $\varepsilon_J\in\set{\pm1}$ by (see Remark \ref{General Rk Delta and partial} for $\partial J$)
    \begin{equation}\label{General Eq epsilonJ}
        \varepsilon_J\eqdef
    \begin{cases}
        (-1)^{f-1}&\text{if}~J_{\rhobar}=\emptyset,~J=\cJ\\
        (-1)^{|(J\setminus\partial J)^{\nss}|}&\text{otherwise}.
    \end{cases}
    \end{equation}
    \end{enumerate}    
\end{definition}

\begin{remark}
\begin{enumerate}
    \item
    By definition, for all $J\subseteq\cJ$ and $j\in\cJ$ we have
    \begin{align}
        r^{J}_j&=\delta_{j+1\in J}(r_j+1)-\delta_{j\in J};\label{General Eq Lem rJ}\\
        c^J_j&=\delta_{j\notin J}(p-2-r_j)+\delta_{j+1\notin J}(r_j+1).\label{General Eq Lem cJ}
    \end{align}
    \item 
    The definition of $\un{c}^J$ is a variant of \cite[(95)]{BHHMS3} (where $\rhobar$ was assumed to be semisimple). Also, by (\ref{General Eq genericity}) we have $\un{0}\leq\un{c}^J\leq\un{p}-\un{1}$.
\end{enumerate}
\end{remark}

\begin{theorem}\label{General Thm seq}
    There exists a unique family of elements $\sset{x_{J,\un{i}}:J\subseteq\cJ,\un{i}\in\ZZ^f}$ of $\pi$ satisfying the following properties:
    \begin{enumerate}
    \item
    For each $J\subseteq\cJ$, we have $x_{J,\un{f}}=\un{Y}^{-(\un{f}-\un{e}^{J^{\sh}})}v_J$ (defined in Proposition \ref{General Prop shift}).
    \item 
    For each $J\subseteq\cJ$, $\un{i}\in\ZZ^f$ and $\un{k}\in\NNN^f$, we have $\un{Y}^{\un{k}}x_{J,\un{i}}=x_{J,\un{i}-\un{k}}$.
    \item 
    For each $J\subseteq\cJ$ and $\un{i}\in\ZZ^f$, we have (see Proposition \ref{General Prop vector} and (\ref{General Eq general mu}) for $\mu_{J+1,J'}$)
    \begin{equation*}
        \smat{p&0\\0&1}x_{J+1,\un{i}}=\sum\limits_{J^{\ss}\subseteq J'\subseteq J}\varepsilon_{J'}\mu_{J+1,J'}x_{J',p\delta(\un{i})+\un{c}^J+\un{r}^{J\setminus J'}}.
    \end{equation*}
    \end{enumerate}
\end{theorem}

\begin{remark}
    The extra term $\un{e}^{J^{\sh}}$ in Theorem \ref{General Thm seq}(i) has the advantage that the constants $\un{c}^J$ and $\un{r}^J$ in Theorem \ref{General Thm seq}(iii) work for arbitrary $J_{\rhobar}$.
\end{remark}

\begin{remark}
    Let $0\leq k\leq f$. Assume that Theorem \ref{General Thm seq} is true for $|J|\leq k$. Then by Theorem \ref{General Thm seq}(ii),(iii), for all $|J|\leq k$, $\un{i}\in\ZZ^f$ and $\un{\ell}\in\NNN^f$ we have 
    \begin{equation}\label{General Eq Seq p001}
        \un{Y}^{\un{\ell}}\smat{p&0\\0&1}x_{J+1,\un{i}}=\sum\limits_{J^{\ss}\subseteq J'\subseteq J}\varepsilon_{J'}\mu_{J+1,J'}x_{J',p\delta(\un{i})+\un{c}^J+\un{r}^{J\setminus J'}-\un{\ell}}.
    \end{equation}
    Moreover, the LHS and each term of the summation in (\ref{General Eq Seq p001}) are $H$-eigenvectors with common $H$-eigencharacter $\chi_J\alpha^{\un{e}^{J^{\sh}}+\un{\ell}-\un{i}}$, see the proof of Corollary \ref{General Cor xJi}(ii) below.
\end{remark}

\begin{proof}[Proof of Theorem \ref{General Thm seq}]
We define the elements $x_{J,\un{i}}\in\pi$ by increasing induction on $\abs{J}$ and on $\max_j{i_j}$. For each $J\subseteq\cJ$ and $\un{i}\leq\un{f}$, we define (see Proposition \ref{General Prop shift})
\begin{equation}\label{General Eq seq small}
    x_{J,\un{i}}\eqdef
    \begin{cases}
        \un{Y}^{-(\un{i}-\un{e}^{J^{\sh}})}v_J&\text{if}~\un{i}\geq\un{e}^{J^{\sh}}\\
        0&\text{otherwise}.
    \end{cases}
\end{equation}

Then we let $\abs{J}=k$ for $0\leq k\leq f$ and $\max_j{i_j}=m>f$. Assume that $x_{J,\un{i}}$ is defined for $\abs{J}\leq k-1$ and all $\un{i}\in\ZZ^f$, and for $\abs{J}=k$ and $\max_j{i_j}\leq m-1$. We write $\un{i}=p\delta(\un{i}')+\un{c}^{J}-\un{\ell}$ for the unique $\un{i}',\un{\ell}\in\ZZ^f$ such that $\un{0}\leq\un{\ell}\leq\un{p}-\un{1}$. Then we claim that $\max_j{i'_j}<\max_j{i_j}=m$. Indeed, for each $j$ we have 
\begin{equation}\label{General Eq max1<max2}
    i'_{j+1}=\bbra{i_j-c^{J}_j+\ell_j}/p\leq\bigbra{m-0+(p-1)}/p<m/p+1<m,
\end{equation}
where the last inequality uses $m>f\geq1$. Then we define $x_{J,\un{i}}$ by the formula
\begin{equation}\label{General Eq Seq def}
    \varepsilon_J\mu_{J+1,J}x_{J,\un{i}}\eqdef\un{Y}^{\un{\ell}}\smat{p&0\\0&1}x_{J+1,\un{i}'}-\sum\limits_{J^{\ss}\subseteq J'\subsetneqq J}\varepsilon_{J'}\mu_{J+1,J'}x_{J',\un{i}+\un{r}^{J\setminus J'}},
\end{equation}
where each term on the RHS of (\ref{General Eq Seq def}) is defined by the induction hypothesis (hence a priori (\ref{General Eq Seq def}) holds for all $J\subseteq\cJ$ and $\un{i}\in\ZZ^f$ such that $\max_ji_j>f$).

\begin{lemma}\label{General Lem Seq reduction}
    Let $0\leq k\leq f$. Assume that Theorem \ref{General Thm seq} is true for $\abs{J}\leq k-1$. If (\ref{General Eq Seq def}) holds for $\abs{J}=k$ and $\un{i}=\un{f}$, then Theorem \ref{General Thm seq} is true for $\abs{J}=k$.
\end{lemma}

\begin{proof}
    By (\ref{General Eq seq small}) and Proposition \ref{General Prop shift}, Theorem \ref{General Thm seq}(ii) is true for all $J\subseteq\cJ$ and $\un{i}\leq\un{f}$. Then we let $J\subseteq\cJ$ such that $|J|=k$. We define $\un{c}^{\prime J}\in\ZZ^f$ by 
    \begin{equation}\label{General Eq c'J}
        c^{\prime J}_j\eqdef
    \begin{cases}
        p-1-f&\text{if}~j\notin J,\ j+1\notin J\\
        r_j+1-f&\text{if}~j\in J,\ j+1\notin J\\
        p-2-r_j-f&\text{if}~j\notin J,\ j+1\in J\\
        p-f&\text{if}~j\in J,\ j+1\in J.
    \end{cases}
    \end{equation}
    In particular, by (\ref{General Eq genericity}) we have $\un{0}\leq\un{c}^{\prime J}\leq\un{p}-\un{1}$, and by \eqref{General Eq c'J} and \eqref{General Eq cJ} we have
    \begin{equation}\label{General Eq cJ c'J}
        \un{f}=p\delta\bigbra{\un{e}^{J\cap(J+1)}}+\un{c}^J-\un{c}^{\prime J}.
    \end{equation}
    
    Since (\ref{General Eq Seq def}) holds for $J$ as above and $\un{i}=\un{f}$ by assumption, using (\ref{General Eq cJ c'J}) we have
    \begin{equation}\label{General Eq seq def i=f}
        \varepsilon_J\mu_{J+1,J}x_{J,\un{f}}=\un{Y}^{\un{c}^{\prime J}}\smat{p&0\\0&1}x_{J+1,\un{e}^{J\cap(J+1)}}-\sum\limits_{J^{\ss}\subseteq J'\subsetneqq J}\varepsilon_{J'}\mu_{J+1,J'}x_{J',\un{f}+\un{r}^{J\setminus J'}}.
    \end{equation}
    For each $\un{i}\leq\un{f}$, we write $\un{i}=p\delta(\un{i}')+\un{c}^{J}-\un{\ell}$ as in (\ref{General Eq Seq def}). In particular, comparing with (\ref{General Eq cJ c'J}) we have  $\un{i}'\leq\un{e}^{J\cap(J+1)}$. By Lemma \ref{General Lem Yj}(i) and Theorem \ref{General Thm seq}(ii) (applied with $(J,\un{i},\un{k})$ there replaced by $\bigbra{J+1,\un{e}^{J\cap(J+1)},\un{e}^{J\cap(J+1)}-\un{i}'}$ and using $\un{e}^{J\cap(J+1)}\leq\un{f}$) we deduce that 
    \begin{equation}\label{General Eq Seq reduction}
        \un{Y}^{\un{f}-\un{i}}\bbbra{\un{Y}^{\un{c}^{\prime J}}\smat{p&0\\0&1}x_{J+1,\un{e}^{J\cap(J+1)}}}=\un{Y}^{\un{\ell}}\smat{p&0\\0&1}x_{J+1,\un{i}'}.
    \end{equation}
    Since Theorem \ref{General Thm seq}(ii) is true for $|J|\leq k-1$, by applying $\un{Y}^{\un{f}-\un{i}}$ to (\ref{General Eq seq def i=f}) and using (\ref{General Eq Seq reduction}) we deduce that (\ref{General Eq Seq def}) is true for $J$ as above and $\un{i}\leq\un{f}$, hence for all $\un{i}\in\ZZ^f$ by definition. 

    Then we use increasing induction on $\max_j{i_j}$ to prove that Theorem \ref{General Thm seq}(ii) is true (for $J$ as above, which satisfies $|J|=k$). We already know that Theorem \ref{General Thm seq}(ii) is true for $\un{i}\leq\un{f}$. Then for each $\un{i}\in\ZZ^f$ and $\un{k}\in\NNN^f$, if we write $\un{i}-\un{k}=p\delta(\un{i}'')+\un{c}^{J}-\un{\ell}'$ for the unique $\un{i}'',\un{\ell}'\in\ZZ^f$ such that $\un{0}\leq\un{\ell}'\leq\un{p}-\un{1}$, then we have
    \begin{align*}
        \un{Y}^{\un{k}}\varepsilon_J\mu_{J+1,J}x_{J,\un{i}}&=\un{Y}^{\un{k}}\bbbra{\un{Y}^{\un{\ell}}\smat{p&0\\0&1}x_{J+1,\un{i}'}-\sum\limits_{J^{\ss}\subseteq J'\subsetneqq J}\varepsilon_{J'}\mu_{J+1,J'}x_{J',\un{i}+\un{r}^{J\setminus J'}}}\\
        &=\un{Y}^{\un{\ell}'}\smat{p&0\\0&1}x_{J+1,\un{i}''}-\sum\limits_{J^{\ss}\subseteq J'\subsetneqq J}\varepsilon_{J'}\mu_{J+1,J'}x_{J',\un{i}-\un{k}+\un{r}^{J\setminus J'}}\\
        &=\varepsilon_J\mu_{J+1,J}x_{J,\un{i}-\un{k}},
    \end{align*}
    where the first and the third equality follow from (\ref{General Eq Seq def}), and the second equality follows in a similar way as (\ref{General Eq Seq reduction}) using the induction hypothesis. This shows that $\un{Y}^{\un{k}}x_{J,\un{i}}=x_{J,\un{i}-\un{k}}$.
    
    Finally, by taking $\un{i}$ such that $\un{\ell}=\un{0}$ in (\ref{General Eq Seq def}), we conclude that Theorem \ref{General Thm seq}(iii) is true (for $J$ as above). 
\end{proof}

Assume that Theorem \ref{General Thm seq} is true for $\abs{J}\leq k-1$. By Lemma \ref{General Lem Seq reduction}, it suffices to prove that (\ref{General Eq Seq def}) is true for $\abs{J}=k$ and $\un{i}=\un{f}$. For $J\subseteq\cJ$, we denote (see (\ref{General Eq c'J}) for $\un{c}^{\prime J}$)
\begin{equation}\label{General Eq def zw}
\begin{aligned}
     z_J&\eqdef\un{Y}^{\un{c}^{\prime J}}\smat{p&0\\0&1}x_{J+1,\un{e}^{J\cap(J+1)}}=\un{Y}^{\un{c}^{\prime J}}\smat{p&0\\0&1}\bbra{\un{Y}^{-\un{e}^{(J\cap(J+1))^{\nss}}}v_{J+1}}\in\pi;\\
     w_J&\eqdef\sum\limits_{J^{\ss}\subseteq J'\subsetneqq J}\varepsilon_{J'}\mu_{J+1,J'}x_{J',\un{f}+\un{r}^{J\setminus J'}}\in\pi.
\end{aligned}
\end{equation}
Then by (\ref{General Eq seq def i=f}), it is equivalent to proving that for $|J|=k$ we have
\begin{equation}\label{General Eq zw}
    z_J-w_J=\varepsilon_J\mu_{J+1,J}x_{J,\un{f}}=\varepsilon_J\mu_{J+1,J}\un{Y}^{-(\un{f}-\un{e}^{J^{\sh}})}v_{J}.
\end{equation}
The proof of (\ref{General Eq zw}) consists of the following three lemmas, where we need to use some results of Appendix \ref{General Sec app vanish}.

\begin{lemma}\label{General Lem lie in D0}
    Let $0\leq k\leq f$. Assume that Theorem \ref{General Thm seq} is true for $\abs{J}\leq k-1$. Then for $\abs{J}\leq k$, we have $z_J,w_J\in D_0(\rhobar)$.
\end{lemma}

\begin{proof}
    (i). First we prove that $z_J\in D_0(\rhobar)$. Let $\un{i}\eqdef\un{e}^{(J\cap(J+1))^{\nss}}$, $J''_1\eqdef J\Delta(J-1)$ and $\un{m}\eqdef\un{m}(\un{i},J+1,J''_1)$ (see (\ref{General Eq mj})). By Lemma \ref{General Lem m} applied with $(J,J')$ there replaced by $(J+1,J)$, we have $m_j=\delta_{j\in J}(-1)^{\delta_{j\notin J}}=\delta_{j\in J}$ for all $j\in\cJ$. Then by Proposition \ref{General Prop Isigma0m}(ii) applied to $(\un{i},J+1,J''_1)$, we have for all $j'\in\cJ$
    \begin{equation*}
        Y_{j'}^{\delta_{J''_1=\emptyset}}\bbbra{\scalebox{1}{$\prod$}_{j\notin J''_1}Y_j^{2i_j+t^{J+1}(J''_1)_j}}\smat{p&0\\0&1}\bbra{\un{Y}^{-\un{i}}v_{J+1}}\in I\bigbra{\sigma_{J^{\ss}},\sigma_{\un{e}^J}}=I\bigbra{\sigma_{J^{\ss}},\sigma_{\un{e}^{J^{\ss}}+\un{e}^{J^{\nss}}}}\subseteq D_{0,\sigma_{J^{\ss}}}(\rhobar),
    \end{equation*}
    where the last inclusion follows from Corollary \ref{General Cor structure of D0}. To prove that $z_J\in D_0(\rhobar)$, it suffices to show that for all $j\in\cJ$ we have
    \begin{equation*}
        c^{\prime J}_j\geq\delta_{j\notin J''_1}\bigbra{2i_j+t^{J+1}(J''_1)_j}+\delta_{J''_1=\emptyset}.
    \end{equation*}
    This is a consequence of Lemma \ref{General Lem r and c}(iv) applied with $J'=J$ and $\un{\delta}=\un{0}$, and (\ref{General Eq cJ c'J}).

    \hspace{\fill}
    
    (ii). Next we prove that $w_J\in D_0(\rhobar)$. To do this, we prove by increasing induction on $|J'|$ that $x_{J',\un{f}+\un{r}^{J\setminus J'}}\in D_0(\rhobar)$ for each $J'$ such that $J^{\ss}\subseteq J'\subsetneqq J$ (which implies $(J')^{\ss}=J^{\ss}$). Let $\un{i}\eqdef\un{e}^{(J'+1)\cap J}-\un{e}^{(J'+1)^{\sh}}=\un{e}^{((J'+1)\cap J)^{\nss}}$ (using $(J')^{\ss}=J^{\ss}$), $J''_2\eqdef J'\Delta(J-1)$ and $\un{m}\eqdef\un{m}(\un{i},J'+1,J''_2)$ (see (\ref{General Eq mj})). By Proposition \ref{General Prop Isigma0m}(ii) applied to $(\un{i},J'+1,J''_2)$, we have for all $j'\in\cJ$
    \begin{equation}\label{General Eq wJ in D0}
        Y_{j'}^{\delta_{J''_2=\emptyset}}\bbbra{\scalebox{1}{$\prod$}_{j\notin J''_2}Y_j^{2i_j+t^{J'+1}(J''_2)_j}}\smat{p&0\\0&1}\bbra{\un{Y}^{-\un{i}}v_{J'+1}}\in I\bigbra{\sigma_{(J')^{\ss}},\sigma_{\un{m}}}.
    \end{equation}
    By Lemma \ref{General Lem m} applied with $(J,J')$ there replaced by $(J'+1,J)$, we have $m_j=\delta_{j\in J}(-1)^{\delta_{j\notin J'}}$ for all $j\in\cJ$. Then a case-by-case examination using $J^{\ss}\subseteq J'\subseteq J$ shows that $m_j-\delta_{j\in(J')^{\ss}}$ equals $0$ if $j\in J_{\rhobar}$ and equals $\delta_{j\in J}(-1)^{\delta_{j\notin J'}}$ if $j\notin J_{\rhobar}$. Hence by Corollary \ref{General Cor structure of D0} we deduce that $I\bigbra{\sigma_{(J')^{\ss}},\sigma_{\un{m}}}\subseteq D_{0,\sigma_{(J')^{\ss}}}(\rhobar)$.

    We let $\un{c}\in\ZZ^f$ be as in Lemma \ref{General Lem r and c}(iv). On one hand, since $\un{Y}^{-\un{i}}v_{J'+1}=x_{J'+1,\un{e}^{(J'+1)\cap J}}$ by (\ref{General Eq seq small}), multiplying (\ref{General Eq wJ in D0}) by a suitable power of $\un{Y}$ and using Lemma \ref{General Lem r and c}(iv) (with $\un{\delta}=\un{0}$) we deduce that 
    \begin{equation*}
        \un{Y}^{\un{c}}\smat{p&0\\0&1}x_{J'+1,\un{e}^{(J'+1)\cap J}}\in D_0(\rhobar).
    \end{equation*} 
    On the other hand, since $\abs{J'}\leq k-1$ by assumption, by (\ref{General Eq Seq p001}) applied to $J'$ and using $(J')^{\ss}=J^{\ss}$ we have
    \begin{align*}
        \un{Y}^{\un{c}}\smat{p&0\\0&1}x_{J'+1,\un{e}^{(J'+1)\cap J}}&=\sum\limits_{J^{\ss}\subseteq J''\subseteq J'}\varepsilon_{J''}\mu_{J'+1,J''}x_{J'',\bigbra{p\delta(\un{e}^{(J'+1)\cap J})+\un{c}^{J'}+\un{r}^{J'\setminus J''}-\un{c}}}\\
        &=\sum\limits_{J^{\ss}\subseteq J''\subseteq J'}\varepsilon_{J''}\mu_{J'+1,J''}x_{J'',\bigbra{\un{r}^{J'\setminus J''}+\un{f}+\un{r}^{J\setminus J'}}}\\
        &=\sum\limits_{J^{\ss}\subseteq J''\subseteq J'}\varepsilon_{J''}\mu_{J'+1,J''}x_{J'',\un{f}+\un{r}^{J\setminus J''}},
    \end{align*}
    where the second equality follows from the definition of $\un{c}$ and the last equality follows from Lemma \ref{General Lem r and c}(ii). By the induction hypothesis, we have $x_{J'',\un{f}+\un{r}^{J\setminus J''}}\in D_0(\rhobar)$ for all $J^{\ss}\subseteq J''\subsetneqq J'$. It follows that $x_{J',\un{f}+\un{r}^{J\setminus J'}}\in D_0(\rhobar)$, which completes the proof.
\end{proof}

\begin{lemma}\label{General Lem z}
    Let $0\leq k\leq f$. Assume that Theorem \ref{General Thm seq} is true for $|J|\leq k-1$. Then for $|J|\leq k$, we have (see Remark \ref{General Rk Delta and partial} for $\partial J$)
    \begin{align*}
        Y_{j_0}^{f+1-\delta_{j_0\in J^{\sh}}}z_J&=
    \begin{cases}
        \sum\limits_{J^{\ss}\subseteq J'\subseteq J\setminus\set{j_0+1}}\varepsilon_{J'}\mu_{J+1,J'}x_{J',\bigbra{\un{r}^{J\setminus J'}+\un{f}-(f+1-\delta_{j_0\in J^{\sh}})e_{j_0}}}&\text{if}~j_0+1\in J^{\nss}\\
        0&\text{if}~j_0+1\notin J^{\nss};
    \end{cases}\\
        \un{Y}^{\un{f}-\un{e}^{J^{\sh}}}z_J&=
    \begin{cases}
        \mu_{J+1,J}v_{J},&\text{if}~J^{\nss}=(\partial J)^{\nss}\\
        0,&\text{if}~J^{\nss}\neq(\partial J)^{\nss}~\text{and}~J^{\nss}\neq\cJ\\
        \mu_{J,J}v_J+\mu_{J,\emptyset}x_{\emptyset,\un{r}},&\text{if}~J^{\nss}=\cJ.
    \end{cases}
    \end{align*}
\end{lemma}

\begin{proof}
    (i) We prove the first equality. Using the decomposition
    \begin{equation}\label{General Eq decomposition of big J}
        \cJ=\bigbra{J^c\cap(J+1)^c}\sqcup\bigbra{\bigbra{(J+1)\Delta J^{\ss}}\cup J^{\nss}}\sqcup(J+1)^{\sh},
    \end{equation}
    we separate the proof into the following five cases. 

    \hspace{\fill}
    
    (a). Suppose that $j_0\notin J$ and $j_0+1\notin J$, which implies $j_0\notin J^{\sh}$. Since $j_0+1\notin \bigbra{J\cap(J+1)}^{\nss}$, by Lemma \ref{General Lem Yj}(i) applied to $j_0$ and Proposition \ref{General Prop shift} we have 
    \begin{equation}\label{General Eq Lem z 2}
        Y_{j_0}^p\smat{p&0\\0&1}\bbra{\un{Y}^{-\un{e}^{(J\cap(J+1))^{\nss}}}v_{J+1}}=0.
    \end{equation}
    Since $c^{\prime J}_{j_0}=p-1-f$ by (\ref{General Eq c'J}), we deduce from (\ref{General Eq def zw}) and (\ref{General Eq Lem z 2}) that $Y_{j_0}^{f+1}z_J=0$.

    \hspace{\fill}

    (b). Suppose that $j_0+1\in (J+1)^{\sh}$, which implies $j_0+1\notin \bigbra{J\cap(J+1)}^{\nss}$. In particular, the equality (\ref{General Eq Lem z 2}) still holds as in (a). 
    Since $c^{\prime J}_{j_0}=p-f$ by (\ref{General Eq c'J}), we deduce from (\ref{General Eq def zw}) and (\ref{General Eq Lem z 2}) that $Y_{j_0}^{f+1-\delta_{j_0\in J^{\sh}}}z_J=0$.

    \hspace{\fill}
    
    (c). Suppose that $j_0+1\in(J+1)\Delta J^{\ss}$ and $j_0+1\notin J^{\nss}$, which implies $f\geq2$ and $j_0\in J\Delta(J-1)$. In particular, we have $j_0\notin J^{\sh}$. By (\ref{General Eq sJ}) we have
    \begin{align*}
        s^{J+1}_j-2\delta_{j\in(J\cap(J+1))^{\nss}}&=
    \begin{cases}
        \bigbra{r_j+\delta_{j-1\in J}}-0&\text{if}~j\notin J\\
        \bigbra{p-1-r_j-\delta_{j-1\notin J}-2\delta_{j\in(J+1)^{\ss}}}-2\delta_{j\in(J+1)^{\nss}}&\text{if}~j\in J
    \end{cases}\\
    &=
    \begin{cases}
        r_j+\delta_{j-1\in J\Delta(J-1)}&\text{if}~j\notin J\\
        p-3-r_j+\delta_{j-1\in J\Delta(J-1)}&\text{if}~j\in J.
    \end{cases}
    \end{align*}
    We let $\un{i}\eqdef\un{e}^{(J\cap(J+1))^{\nss}}$ and $J''\eqdef\bigbra{J\Delta(J-1)}\setminus\set{j_0}$. In particular, we have $i_{j_0+1}=0$. Then by (\ref{General Eq tJJ'}), for $j\neq j_0$ we have
    \begin{equation}\label{General Eq Lem z 4}
    \begin{aligned}
        \delta_{j\notin J''}\bigbra{2i_j+t^{J+1}(J'')_j}&=\delta_{j\notin J''}\bigbra{2\delta_{j\in(J\cap(J+1))^{\nss}}+p-1-s^{J+1}_j+\delta_{j-1\in J''}}\\
        &=
    \begin{cases}
        p-1-r_j-\delta_{j=j_0+1}&\text{if}~j\notin J,~j+1\notin J\\
        r_j+2-\delta_{j=j_0+1}&\text{if}~j\in J,~j+1\in J\\
        0&\text{otherwise},
    \end{cases}
    \end{aligned}    
    \end{equation}
    and $2i_{j_0}+t^{J+1}(J'')_{j_0}$ equals $p-1-r_j$ if $j_0\notin J$ (which implies $j_0+1\in J$), and equals $r_j+2$ if $j_0\in J$ (which implies $j_0+1\notin J$). In particular, by (\ref{General Eq c'J}) we have $2i_{j_0}+t^{J+1}(J'')_{j_0}=f+1+c^{\prime J}_{j_0}$.
    
    By Proposition \ref{General Prop relation 1} applied to $(\un{i},J+1,J'')$ with $j_0$ as above, taking $j'=j_0+1$ in (\ref{General Eq relation 1 statement}) when $J''=\emptyset$ and multiplying $Y_{j_0+1}^{\delta_{J''\neq\emptyset}}$ when $j_0+1\notin J'$, we deduce that 
    \begin{equation}\label{General Eq Lem z 3}
        \bbbra{Y_{j_0}^{f+1+c^{\prime J}_{j_0}}\scalebox{1}{$\prod$}_{j\notin J,j+1\notin J}Y_j^{p-1-r_j}\scalebox{1}{$\prod$}_{j\in J,j+1\in J}Y_j^{r_j+2}}\smat{p&0\\0&1}\bbra{\un{Y}^{-\un{i}}v_{J+1}}=0.
    \end{equation}
    Comparing (\ref{General Eq def zw}) and (\ref{General Eq Lem z 3}), to prove $Y_{j_0}^{f+1}z_J=0$, it is enough to show that 
    \begin{equation}\label{General Eq Lem z 6}
        c^{\prime J}_{j}\geq
    \begin{cases}
        p-1-r_j&\text{if}~j\notin J,~j+1\notin J\\
        r_j+2&\text{if}~j\in J,~j+1\in J,
    \end{cases}
    \end{equation}
    which follows directly from (\ref{General Eq c'J}) and (\ref{General Eq genericity}).

    \hspace{\fill}

    (d). Suppose that $j_0\notin J$ and $j_0+1\in J^{\nss}$, which implies $f\geq2$ and $j_0\in J\Delta(J-1)$. In particular, we have $j_0\notin J^{\sh}$. We let $\un{i}\eqdef\un{e}^{(J\cap(J+1))^{\nss}}$ and $J''\eqdef\bigbra{J\Delta(J-1)}\setminus\set{j_0}$. In particular, we have $i_{j_0+1}=0$. As in (c), the equality (\ref{General Eq Lem z 4}) still holds and we have $2i_{j_0}+t^{J+1}(J'')_{j_0}=f+1+c^{\prime J}_{j_0}$. We denote
    \begin{equation*}
        Z\eqdef\prod\limits_{j+1\notin J,j+2\notin J}Y_j^{p-1-r_j}\prod\limits_{j+1\in J,j+2\in J}Y_j^{r_j+2}\in\FF\ddbra{N_0}.
    \end{equation*}
    Then by Proposition \ref{General Prop relation 2} applied to $(\un{i},J+1,J'')$ with $j_0$ as above, using (\ref{General Eq relation 2 tJ}) and together with Lemma \ref{General Lem Yj}(i) applied to $j_0+1$ if moreover $j_0+1\notin J''$, we have
    \begin{multline}\label{General Eq Lem z 5}
        \bbbra{Y_{j_0+1}Y_{j_0}^{f+1+c^{\prime J}_{j_0}}Z}\smat{p&0\\0&1}\bbra{\un{Y}^{-\un{i}}v_{J+1}}\\
        =\frac{\mu_{J+1,*}}{\mu_{(J+1)\setminus\set{j_0+2},*}}Y_{j_0+1}^{p-1-r_{j_0+1}}\bbbra{Y_{j_0+1}Y_{j_0}^{f+1+c^{\prime J}_{j_0}}Z}\smat{p&0\\0&1}\bbra{\un{Y}^{-\un{i}'}v_{(J+1)\setminus\set{j_0+2}}},
    \end{multline}
    where $\un{i}'\eqdef\un{i}+\delta_{j_0+2\in J^{\ss}}e_{j_0+2}$. Then using (\ref{General Eq Lem z 6}) together with $c^{\prime J}_{j_0+1}\geq1$ by (\ref{General Eq c'J}) and (\ref{General Eq genericity}), we deduce from (\ref{General Eq def zw}) and (\ref{General Eq Lem z 5}) that
    \begin{equation}\label{General Eq Lem z 7}
        Y_{j_0}^{f+1}z_J=\frac{\mu_{J+1,*}}{\mu_{(J+1)\setminus\set{j_0+2},*}}\bbbra{Y_{j_0}^{f+1}Y_{j_0+1}^{p-1-r_{j_0+1}}\un{Y}^{\un{c}^{\prime J}}}\smat{p&0\\0&1}\bbra{\un{Y}^{-\un{i}'}v_{(J+1)\setminus\set{j_0+2}}}.
    \end{equation}
    Since we have
    \begin{equation*}
        \un{i}'+\un{e}^{((J+1)\setminus\set{j_0+2})^{\sh}}=\un{i}+\un{e}^{(J+1)^{\sh}}=\un{e}^{J\cap(J+1)},
    \end{equation*}
    by (\ref{General Eq seq small}) we have $\un{Y}^{-\un{i}'}v_{(J+1)\setminus\set{j_0+2}}=x_{(J+1)\setminus\set{j_0+2},\un{e}^{J\cap(J+1)}}$. Then by (\ref{General Eq Seq p001}) applied to $J\setminus\set{j_0+1}$ and $\un{i}=\un{e}^{J\cap(J+1)}$, and using $j_0+1\notin J^{\ss}$, we deduce from (\ref{General Eq Lem z 7}) that
    \begin{align*}
        Y_{j_0}^{f+1}z_J\!=\!\frac{\mu_{J+1,*}}{\mu_{(J+1)\setminus\set{j_0+2},*}}\!\bbbra{\sum\limits_{J^{\ss}\subseteq J'\subseteq J\setminus\set{j_0+1}}\!\mu_{(J+1)\setminus\set{j_0+2},J'}x_{J',\un{c}(J')}}\!=\!\sum\limits_{J^{\ss}\subseteq J'\subseteq J\setminus\set{j_0+1}}\!\mu_{J+1,J'}x_{J',\un{c}(J')}
    \end{align*}
    with $\un{c}(J')\in\ZZ^f$ such that 
    \begin{align*}
        c(J')_j&=p\delta_{j+1\in J\cap(J+1)}+c^{J\setminus\set{j_0+1}}_j+r^{(J\setminus\set{j_0+1})\setminus J'}_j-c^{\prime J}_j-(p-1-r_{j_0+1})\delta_{j=j_0+1}-(f+1)\delta_{j=j_0}\\
        &=c^{J\setminus\set{j_0+1}}_j+r^{(J\setminus\set{j_0+1})\setminus J'}_j-c^{J}_j-(p-1-r_{j_0+1})\delta_{j=j_0+1}+f-(f+1)\delta_{j=j_0}\\
        &=r^{(J\setminus\set{j_0+1})\setminus J'}_j+r^{\set{j_0+1}}_j+f-(f+1)\delta_{j=j_0}\\
        &=r^{J\setminus J'}_j+f-(f+1)\delta_{j=j_0},
    \end{align*}
    where the second equality follows from \eqref{General Eq cJ c'J}, the third equality follows from Lemma \ref{General Lem r and c}(v) applied to $J'=J\setminus\set{j_0+1}$, and the last equality follows from Lemma \ref{General Lem r and c}(ii). This proves the desired formula.

    \hspace{\fill}

    (e). Suppose that $j_0\in J$ and $j_0+1\in J^{\nss}$. The proof is similar to (d), except that we take $\un{i}\eqdef\un{e}^{(J\cap(J+1))^{\nss}\setminus\set{j_0+1}}$ and $J''\eqdef\cJ$.

    \hspace{\fill}

    (ii). We prove the second equality. By (\ref{General Eq def zw}) we have
    \begin{align}
        \un{Y}^{\un{f}-\un{e}^{J^{\sh}}}z_J&=\un{Y}^{\un{f}-\un{e}^{J^{\sh}}+\un{c}^{\prime J}}\smat{p&0\\0&1}x_{J+1,\un{e}^{J\cap(J+1)}}\notag\\
        &=\un{Y}^{\un{c}^J+(p-1)\un{e}^{J^{\sh}}}\smat{p&0\\0&1}\bbra{\un{Y}^{\un{e}^{(J\cap(J+1))^{\nss}+1}}x_{J+1,\un{e}^{J\cap(J+1)}}}\label{General Eq Lem z 9}\\
        &=\un{Y}^{\un{c}^J+(p-1)\un{e}^{J^{\sh}}}\smat{p&0\\0&1}x_{J+1,\un{e}^{(J^{\sh}+1)}},\label{General Eq Lem z 10}
    \end{align}    
    where the second equality follows from \eqref{General Eq cJ c'J} and Lemma \ref{General Lem Yj}(i), and the last equality follows from (\ref{General Eq seq small}) and the equality $\bigbra{J\cap(J+1)}\setminus\bigbra{\bra{J\cap(J-1)}^{\nss}+1}=J^{\sh}+1$.

    Suppose that $J^{\nss}=(\partial J)^{\nss}$, or equivalently $\bigbra{J\cap(J-1)}^{\nss}=\emptyset$. Then using (\ref{General Eq sJ}) and (\ref{General Eq cJ}), a case-by-case examination shows that
    \begin{equation*}
        \un{c}^J+(p-1)\un{e}^{J^{\sh}}=
    \begin{cases}
        s^J_{j}&\text{if}~j\in J\Delta(J-1)\\
        p-1&\text{if}~j\notin J\Delta(J-1).
    \end{cases}
    \end{equation*}
    We also have $x_{J+1,\un{e}^{J\cap(J+1)}}=\un{Y}^{-\un{e}^{(J\cap(J+1))^{\nss}}}v_{J+1}$ by (\ref{General Eq seq small}). Then by Proposition \ref{General Prop vector} applied to $(J+1,J)$, we deduce from (\ref{General Eq Lem z 9}) that $\un{Y}^{\un{f}-\un{e}^{J^{\sh}}}z_J=\mu_{J+1,J}v_J$.

    Suppose that $J^{\nss}\neq(\partial J)^{\nss}$ and $J^{\nss}\neq\cJ$, which implies $J^{\ss}\sqcup(\partial J)^{\nss}\subsetneqq J$. Then by Lemma \ref{General Lem lem for zw} applied to $(J,J)$ and using $\un{r}^{\emptyset}=\un{0}$ by (\ref{General Eq rJ}), we deduce from (\ref{General Eq Lem z 10}) that $\un{Y}^{\un{f}-\un{e}^{J^{\sh}}}z_J=0.$

    Suppose that $J^{\nss}=\cJ$, or equivalently $(J,J_{\rhobar})=(\cJ,\emptyset)$, which implies $J+1=J$ and $J^{\sh}=\emptyset$. Then by Proposition \ref{General Prop vector complement} applied to $\cJ$ and using $\un{c}^{\cJ}=\un{0}$ by (\ref{General Eq cJ}), we deduce from (\ref{General Eq Lem z 10}) that $\un{Y}^{\un{f}-\un{e}^{J^{\sh}}}z_J=\mu_{J,J}v_J+\mu_{J,\emptyset}x_{\emptyset,\un{r}}$.
\end{proof}

\begin{lemma}\label{General Lem w}
    Let $0\leq k\leq f$. Assume that Theorem \ref{General Thm seq} is true for $|J|\leq k-1$. Then for $|J|\leq k$, we have 
    \begin{align*}
        Y_{j_0}^{f+1-\delta_{j_0\in J^{\sh}}}w_J&=
    \begin{cases}
        \sum\limits_{J^{\ss}\subseteq J'\subseteq J\setminus\set{j_0+1}}\varepsilon_{J'}\mu_{J+1,J'}x_{J',\bigbra{\un{r}^{J\setminus J'}+\un{f}-(f+1-\delta_{j_0\in J^{\sh}})e_{j_0}}}&\text{if}~j_0+1\in J^{\nss}\\
        0&\text{if}~j_0+1\notin J^{\nss};
    \end{cases}\\
        \un{Y}^{\un{f}-\un{e}^{J^{\sh}}}w_J&=
    \begin{cases}
        0&\text{if}~J^{\nss}=(\partial J)^{\nss}\\
        (-1)^{|(J\setminus\partial J)^{\nss}|+1}\mu_{J+1,J}v_{J}&\text{if}~J^{\nss}\neq(\partial J)^{\nss}~\text{and}~J^{\nss}\neq\cJ\\
        \bigbra{1+(-1)^f}\mu_{J,J}v_J+\mu_{J,\emptyset}x_{\emptyset,\un{r}}&\text{if}~J^{\nss}=\cJ.
    \end{cases}
    \end{align*}
\end{lemma}

\begin{proof}
    (i). We prove the first equality. By definition and Theorem \ref{General Thm seq}(ii) (applied to $J'$ such that $J^{\ss}\subseteq J'\subsetneqq J$, which implies $|J'|\leq k-1$), we have
    \begin{equation*}
        Y_{j_0}^{f+1-\delta_{j_0\in J^{\sh}}}w_J=\sum\limits_{J^{\ss}\subseteq J'\subsetneqq J}\varepsilon_{J'}\mu_{J+1,J'}x_{J',\bigbra{\un{r}^{J\setminus J'}+\un{f}-(f+1-\delta_{j_0\in J^{\sh}})e_{j_0}}}.
    \end{equation*}
    By Corollary \ref{General Cor w 0}, we have $x_{J',\bigbra{\un{r}^{J\setminus J'}+\un{f}-(f+1-\delta_{j_0\in J^{\sh}})e_{j_0}}}=0$ if $J^{\ss}\subseteq J'\subseteq J$ and $j_0+1\notin J\setminus J'$. Then we easily conclude.

    \hspace{\fill}

    (ii). We prove the second equality. We assume that $J^{\nss}\neq\cJ$, the case $J^{\nss}=\cJ$ being similar. Then by Theorem \ref{General Thm seq}(ii) (applied to $J'$ such that $J^{\ss}\subseteq J'\subsetneqq J$) and Proposition \ref{General Prop xJ' xJ}(i) we have
    \begin{align*}
        \un{Y}^{\un{f}-\un{e}^{J^{\sh}}}w_J&=\sum\limits_{J^{\ss}\subseteq J'\subsetneqq J}\varepsilon_{J'}\mu_{J+1,J'}x_{J',\un{r}^{J\setminus J'}+\un{e}^{J^{\sh}}}=\bbbra{\sum\limits_{J^{\ss}\sqcup(\partial J)^{\nss}\subseteq J'\subsetneqq J}(-1)^{\abs{(J'\setminus\partial J)^{\nss}}}}\mu_{J+1,J}v_{J}\\
        &=
        \begin{cases}
            0&\text{if}~J^{\nss}=(\partial J)^{\nss}\\
            (-1)^{\abs{(J\setminus\partial J)^{\nss}}+1}\mu_{J+1,J}v_J&\text{if}~J^{\nss}\neq(\partial J)^{\nss},
        \end{cases}
    \end{align*}
    where the last equality follows as in (\ref{General Eq xJ'xJ 7}).
\end{proof}

In particular, by Lemma \ref{General Lem z} and Lemma \ref{General Lem w} with a case-by-case examination, we have
\begin{align*}
    Y_j^{f+1-\delta_{j\in J^{\sh}}}(z_J-w_J)&=0~\forall\,j\in\cJ;\\
    \un{Y}^{\un{f}-\un{e}^{J^{\sh}}}(z_J-w_J)&=\varepsilon_Jv_J.
\end{align*}
Then (\ref{General Eq zw}) is a consequence of Lemma \ref{General Lem lie in D0} and Proposition \ref{General Prop shift}. This proves the existence of the family $\set{x_{J,\un{i}}:J\subseteq\cJ,\un{i}\in\ZZ^f}$. Finally, the uniqueness is clear from the construction. This completes the proof of Theorem \ref{General Thm seq}. 
\end{proof}

\begin{corollary}\label{General Cor xJi}
    Let $J\subseteq\cJ$ and $\un{i}\in\ZZ^f$. Then $H$ acts on $x_{J,\un{i}}$ (possibly zero) by the character $\chi'_J\alpha^{-\un{i}}$, where $\chi'_J\eqdef\chi_J\alpha^{\un{e}^{J^{\sh}}}$ (see \S\ref{General Sec sw} for $\chi_J$).
\end{corollary}

\begin{proof}
    We prove the result by increasing induction on $|J|$ and $\max_ji_j$. If $\un{i}\leq\un{f}$, then the claim follows from (\ref{General Eq seq small}) and Proposition \ref{General Prop shift}. Next we assume that $\max_ji_j>f$ and write $\un{i}=p\delta(\un{i}')+\un{c}^{J}-\un{\ell}$ for the unique $\un{i}',\un{\ell}\in\ZZ^f$ such that $\un{0}\leq\un{\ell}\leq\un{p}-\un{1}$. In particular, we have $\max_ji'_j<\max_ji_j$ (see \ref{General Eq max1<max2}). By the induction hypothesis and Lemma \ref{General Lem Yj}(ii), $H$ acts on $\un{Y}^{\un{\ell}}\smat{p&0\\0&1}x_{J+1,\un{i}'}$ by the character 
    \begin{equation*}
        \chi'_{J+1}\alpha^{-\un{i}'+\un{\ell}}=\chi'_{J+1}\alpha^{-p\delta(\un{i}')+\un{\ell}}=\chi'_J\alpha^{\un{r}^J-\un{r}^{J+1}-p\delta(\un{i}')+\un{\ell}}=\chi'_J\alpha^{-\un{c}^J-p\delta(\un{i}')+\un{\ell}}=\chi'_J\alpha^{-\un{i}},
    \end{equation*}
    where the second equality follows from Lemma \ref{General Lem r and c}(i) and the third equality follows from Lemma \ref{General Lem r and c}(iii). By the induction hypothesis, for each $J'$ such that $J^{\ss}\subseteq J'\subsetneqq J$,  $H$ acts on $x_{J',\un{i}+\un{r}^{J\setminus J'}}$ by the character 
    \begin{equation*}
        \chi'_{J'}\alpha^{-\un{i}-\un{r}^{J\setminus J'}}=\chi'_{J}\alpha^{\un{r}^J-\un{r}^{J'}-\un{i}-\un{r}^{J\setminus J'}}=\chi'_J\alpha^{-\un{i}},
    \end{equation*}
    where the first equality follows from Lemma \ref{General Lem r and c}(i) and the second equality follows from Lemma \ref{General Lem r and c}(ii). Hence we deduce from (\ref{General Eq Seq def}) that $H$ acts on $x_{J,\un{i}}$ by the character $\chi'_J\alpha^{-\un{i}}$.
\end{proof}

\section{The finiteness condition}\label{General Sec finiteness}

In this section, we prove the crucial finiteness condition for the family of elements $(x_{J,\un{i}})_{J,\un{i}}$ of Theorem \ref{General Thm seq} to give rise to a basis of $D_A(\pi)$. The main result is Theorem \ref{General Thm finiteness}. The proof is by induction on $\un{i}$, the base case being a consequence of the vanishing result: Proposition \ref{General Prop more 0}.

\begin{lemma}\label{General Lem lem for finiteness induction}
    Let $J\subseteq\cJ$ and $\un{i}\in\ZZ^f$. Suppose that $x_{J',\un{i}}=0$ for all $|J'|\leq|J|$. Then we have $x_{J,p\delta(\un{i})+\un{c}^{J}}=0$.
\end{lemma}

\begin{proof}
    We use increasing induction on $|J|$. By Theorem \ref{General Thm seq}(iii) we have
    \begin{equation*}
        0=\smat{p&0\\0&1}x_{J+1,\un{i}}=\sum\limits_{J^{\ss}\subseteq J'\subseteq J}\varepsilon_{J'}\mu_{J+1,J'}x_{J',p\delta(\un{i})+\un{c}^{J}+\un{r}^{J\setminus J'}}.
    \end{equation*}
    By the induction hypothesis, for all $J'\subsetneqq J$ we have $x_{J',p\delta(\un{i})+\un{c}^{J'}}=0$.
    Since $\un{c}^{J'}\geq\un{c}^{J}+\un{r}^{J\setminus J'}$ by Lemma \ref{General Lem r and c}(v), we deduce from Theorem \ref{General Thm seq}(ii) that $x_{J',p\delta(\un{i})+\un{c}^{J}+\un{r}^{J\setminus J'}}=0$ for all $J'\subsetneqq J$. Hence we conclude that $x_{J,p\delta(\un{i})+\un{c}^{J}}=0$.
\end{proof}

\begin{proposition}\label{General Prop degree+finiteness}
    Let $\un{i}\in\ZZ^f$ satisfying
    \begin{enumerate}
    \item
    $\norm{\un{i}}\leq f$;
    \item 
    $i_{j}\leq-1$ for some $j\in\cJ$.
    \end{enumerate}
    Then we have $x_{J,\un{i}}=0$ for all $J\subseteq\cJ$.
\end{proposition}

\begin{proof}
    If $f=1$, then we conclude using (\ref{General Eq seq small}). From now on we assume that $f\geq2$. We use increasing induction on $\max_ji_j$. If $\un{i}\leq\un{f}$, then the lemma follows from (\ref{General Eq seq small}). Then we assume that $\max_ji_j>f$ and let $J_0\subseteq\cJ$. We write $\un{i}=p\delta(\un{i}')+\un{c}^{J_0}-\un{\ell}$ for the unique $\un{i}',\un{\ell}\in\ZZ^f$ such that $\un{0}\leq\un{\ell}\leq\un{p}-\un{1}$. In particular, we have $\max_j{i'_j}<\max_j{i_j}$ (see (\ref{General Eq max1<max2})). Since $\un{c}^{J_0}\geq\un{0}$ by (\ref{General Eq cJ}), we also have
    \begin{equation*}
        \norm{\un{i}'}=\bbra{\norm{\un{i}}-\norm{\un{c}^{J_0}}+\norm{\un{\ell}}}/p\leq(f-0+(p-1)f)/p=f.
    \end{equation*}
    
    Suppose that $i'_{j}<0$ for some $j$. Then by the induction hypothesis, we have $x_{J,\un{i}'}=0$ for all $J\subseteq\cJ$ (in particular, for all $|J|\leq|J_0|$). Since $p\delta(\un{i}')+\un{c}^{J_0}\geq\un{i}$, we deduce from Lemma \ref{General Lem lem for finiteness induction} and Theorem \ref{General Thm seq}(ii) that $x_{J_0,\un{i}}=0$.

    Suppose that $i'_j\geq0$ for all $j$, which implies $i_j\geq-(p-1)$ for all $j$. Since $\norm{i}\leq f$, we deduce that $i_j\leq(f-1)(p-1)+f$ for all $j$. We write $\min_ji_j=-m'$ with $1\leq m'\leq p-1$ and fix $j_0\in\cJ$ such that $i_{j_0}=-m'$. Then we let $\un{n}\in\ZZ^f$ with $n_{j_0+1}=0$, $n_j=m'$ for $j\neq j_0+1$ if $1\leq m'\leq2f-2$, and $n_j=2f-1$ for $j\neq j_0+1$ if $2f-1\leq m'\leq p-1$. In particular, $\un{n}$ satisfies the conditions in Definition \ref{General Def aJn}(ii) for $J_0$ and $j_0$. By Proposition \ref{General Prop more 0} applied to $J_0$ and $j_0$ we have $x_{J_0,\un{a}^{J_0}(\un{n})-e_{j_0+1}}=0$ (see Definition \ref{General Def aJn} for $\un{a}^{J_0}(\un{n})$). Then the result follows from Theorem \ref{General Thm seq}(ii) and the Claim below.

    \hspace{\fill}

    \noindent\textbf{Claim.} We have $\un{a}^{J_0}(\un{n})-e_{j_0+1}\geq\un{i}$.

    \proof We assume that $1\leq m'\leq 2f-2$, the case $2f-1\leq m'\leq p-1$ being similar. Since $\norm{\un{i}}\leq f$ and $i_j\geq-m'$ for all $j$, we deduce that $i_j\leq(f-1)m'+f$ for all $j$. Hence it suffices to show that 
    \begin{equation*}
    \begin{cases}
        a^{J_0}(\un{n})_{j_0}\geq-m'\\
        a^{J_0}(\un{n})_j-\delta_{j=j_0+1}\geq(f-1)m'+f&\text{if}~j\neq j_0.
    \end{cases}
    \end{equation*}
    
    By Definition \ref{General Def aJn}(ii), we have $a^{J_0}(\un{n})_{j_0}=0\geq-m'$ if $j_0\in J_0^{\sh}$, and $a^{J_0}(\un{n})_{j_0}=t^{J_0}_{j_0}(0)-n_{j_0}=-m'$ (since $n_{j_0+1}=0$) if $j_0\notin J_0^{\sh}$. For $j\neq j_0$, by Definition \ref{General Def aJn}(i),(ii) we have
    \begin{align*}
        a^{J_0}(\un{n})_j-\delta_{j=j_0+1}&=t^{J_0}_j(m')-n_j-\delta_{j=j_0+1}\\
        &=p[m'/2]+\delta_{2\nmid m'}\bigbra{\delta_{j+1\notin J}(r_j+1)+\delta_{j+1\in J}(p-1-r_j)}-(n_j+\delta_{j=j_0+1})\\
        &\geq(4f+4)[m'/2]+\delta_{2\nmid m'}(2f+2)-1\\
        &=(2f+2)m'-1=fm'+fm'+(2m'-1)>(f-1)m'+f,
    \end{align*}
    where the first inequality uses (\ref{General Eq genericity}) and $p\geq 4f+4$, and the last inequality uses $m'\geq1$. Here for $x\in\RR$, we denote by $[x]$ the largest integer which is smaller than or equal to $x$. This proves the claim.
\end{proof}

\begin{corollary}\label{General Cor degree}
    Let $J\subseteq\cJ$, $\un{i}\in\ZZ^f$ and $\un{k}\in\NNN^f$.
    \begin{enumerate}
    \item 
    If $\norm{\un{k}}>\norm{\un{i}}-|J^{\sh}|$, then we have $\un{Y}^{\un{k}}x_{J,\un{i}}=0$.
    \item 
    If $\norm{\un{k}}=\norm{\un{i}}-|J^{\sh}|$ and $\un{Y}^{\un{k}}x_{J,\un{i}}\neq0$, then we have $\un{k}=\un{i}-\un{e}^{J^{\sh}}$.
    \end{enumerate}
\end{corollary}

\begin{proof}
    By Theorem \ref{General Thm seq}(ii) we have $\un{Y}^{\un{k}}x_{J,\un{i}}=x_{J,\un{\ell}}$ with $\un{\ell}\eqdef\un{i}-\un{k}$. In both cases, we have $\norm{\un{\ell}}\leq f$ since $\abs{J^{\sh}}\leq f$. If $\un{\ell}\geq\un{0}$, then we have $\un{\ell}\leq\un{f}$, and the result follows from (\ref{General Eq seq small}). If $\un{\ell}\ngeq\un{0}$, then the result follows from Proposition \ref{General Prop degree+finiteness}.
\end{proof}

\begin{proposition}\label{General Prop finiteness}
    Let $m\in\NNN$ and $\un{i}\in\ZZ^f$ satisfying
    \begin{enumerate}
    \item
    $\norm{\un{i}}\leq p^m+f-1$;
    \item 
    $i_{j}\leq-p^m$ for some $j\in\cJ$.
    \end{enumerate}
    Then we have $x_{J,\un{i}}=0$ for all $J\subseteq\cJ$.
\end{proposition}

\begin{proof}
    We prove the result by increasing induction on $m$. For $m=0$, this is exactly Proposition \ref{General Prop degree+finiteness}. Then we let $m\geq1$ and fix $J\subseteq\cJ$. We write $\un{i}=p\delta(\un{i}')+\un{c}^{J}-\un{\ell}$ for the unique $\un{i}',\un{\ell}\in\ZZ^f$ such that $\un{0}\leq\un{\ell}\leq\un{p}-\un{1}$ and fix $j_0\in\cJ$ such that $i_{j_0}\leq-p^m$. Since $\un{c}^J\geq\un{0}$ by (\ref{General Eq cJ}), we have
    \begin{equation*}
        \norm{\un{i}'}=\bbra{\norm{\un{i}}-\norm{\un{c}^{J+1}}+\norm{\un{\ell}}}/p\leq\bigbra{(p^m+f-1)-0+(p-1)f}/p=p^{m-1}+f-1/p,
    \end{equation*}
    which implies $\norm{\un{i}'}\leq p^{m-1}+f-1$. We also have 
    \begin{equation*}
        i'_{j_0+1}=\bbra{i_{j_0}-c^{J+1}_{j_0}+\ell_{j_0}}/p\leq\bigbra{-p^m-0+(p-1)}/p=-p^{m-1}+(p-1)/p,
    \end{equation*}
    which implies $i'_{j_0+1}\leq-p^{m-1}$. By the induction hypothesis, we have $x_{J',\un{i}'}=0$ for all $J'\subseteq\cJ$. Since $p\delta(\un{i}')+\un{c}^{J}\geq\un{i}$, we conclude from Lemma \ref{General Lem lem for finiteness induction} and Theorem \ref{General Thm seq}(ii) that $x_{J,\un{i}}=0$.
\end{proof}

\begin{theorem}[Finiteness condition]\label{General Thm finiteness}
    For $J\subseteq\cJ$ and $M\in\ZZ$, the set $\sset{\un{i}\in\ZZ^f:x_{J,\un{i}}\neq0,\norm{\un{i}}=M}$ is finite.
\end{theorem}

\begin{proof}
    We choose $m$ large enough such that $p^m+f-1\geq M$. If $i_{j_0}\leq-p^m$ for some $j_0$, then by Proposition \ref{General Prop finiteness} we have $x_{J,\un{i}}=0$. Otherwise, we have $i_j>-p^m$ for all $j$. Together with the restriction $\norm{\un{i}}=M$, this set is finite.
\end{proof}

\section{An explicit basis of \texorpdfstring{$\Hom_A(D_A(\pi),A)$}.}\label{General Sec basis}

In this section, we construct an explicit basis of $\Hom_A(D_A(\pi),A)$. In particular, we prove that $D_A(\pi)$ has rank $2^f$, see Theorem \ref{General Thm rank 2^f}. As a corollary, we finish the proof of Theorem \ref{General Thm main1}.

First we recall the definition of the ring $A$ and the $A$-module $D_A(\pi)$. We let $\fm_{N_0}$ be the maximal ideal of $\FF\ddbra{N_0}$. Then we have $\FF\ddbra{N_0}=\FF\ddbra{Y_0,\ldots,Y_{f-1}}$ and $\fm_{N_0}=(Y_0,\ldots,Y_{f-1})$. Consider the multiplicative subset $S\eqdef\sset{(Y_0\cdots Y_{f-1})^n:n\geq0}$ of $\FF\ddbra{N_0}$. Then $A\eqdef\wh{\FF\ddbra{N_0}_S}$ is the completion of the localization $\FF\ddbra{N_0}_S$ with respect to the $\fm_{N_0}$-adic filtration 
\begin{equation*}
    F_n\bbra{\FF\ddbra{N_0}_S}=\bigcup\limits_{k\geq0}\frac{1}{(Y_0\cdots Y_{f-1})^k}\fm_{N_0}^{kf-n},
\end{equation*}
where $\fm_{N_0}^m\eqdef\FF\ddbra{N_0}$ if $m\leq0$. We denote by $F_nA$ ($n\in\ZZ$) the induced filtration on $A$ and endow $A$ with the associated topology (\cite[\S1.3]{LvO96}). There is an $\FF$-linear action of $\OK\x$ on $\FF\ddbra{N_0}$ given by multiplication on $N_0\cong\OK$, and an $\FF$-linear Frobenius $\varphi$ on $\FF\ddbra{N_0}$ given by multiplication by $p$ on $N_0\cong\OK$. They extend canonically by continuity to commuting continuous $\FF$-linear actions of $\varphi$ and $\OK\x$ on $A$ which satisfies (for each $j\in\cJ$)
\begin{equation}\label{General Eq action on Yj}
\begin{aligned}
    \varphi(Y_j)&=Y_{j-1}^p;\\
    [a](Y_j)&=a^{p^j}Y_j\ \forall\,a\in\Fq\x.
\end{aligned}    
\end{equation}

We let $\pi^{\vee}$ be the $\FF$-linear dual of $\pi$, which is a finitely generated $\FF\ddbra{I_1}$-module and is endowed with the $\fm_{I_1}$-adic topology, where $\fm_{I_1}$ is the maximal ideal of $\FF\ddbra{I_1}$. We define $D_A(\pi)$ to be the completion of $\FF\ddbra{N_0}_S\otimes_{\FF\ddbra{N_0}}\pi^{\vee}$ with respect to the tensor product topology. The $\OK\x$-action on $\pi^{\vee}$ given by $f\mapsto f\circ\smat{a&0\\0&1}$ (for $a\in\OK\x$) extends by continuity to $D_A(\pi)$, and the $\psi$-action on $\pi^{\vee}$ given by $f\mapsto f\circ\smat{p&0\\0&1}$ induces a continuous $A$-linear map $\beta:D_A(\pi)\to A\otimes_{\varphi,A}D_A(\pi)$. Moreover, $D_A(\pi)$ is a finite free $A$-module by \cite[Remark~3.3.2.6(ii)]{BHHMS2}, \cite[Cor.~3.1.2.9]{BHHMS2} and \cite[Remark.~2.6.2]{BHHMS3}.

As in \cite[(87)]{BHHMS3}, there exists an injective $A$-linear map
\begin{equation}\label{General Eq mu*}
    \mu_{*}:\Hom_A(D_A(\pi),A)\into\Hom_{\FF}^{\cont}(D_A(\pi),\FF).
\end{equation}
By \cite[Prop.~3.2.3]{BHHMS3}, $\Hom_{\FF}^{\cont}(D_A(\pi),\FF)$ is identified with the set of sequences $(x_k)_{k\geq0}$ with $x_k\in\pi$ and
\begin{enumerate}
    \item 
    $\un{Y}^{\un{1}}x_k=x_{k-1}$ for all $k\geq1$;
    \item 
    there exists $d\in\ZZ$ such that $x_k\in\pi[\fm_{I_1}^{fk+d+1}]$ for all $k\geq0$ (where $\pi[\fm_{I_1}^j]\eqdef0$ if $j\leq0$),
\end{enumerate}
and the $A$-module structure on $\Hom_{\FF}^{\cont}(D_A(\pi),\FF)$ is given as follows: for $a\in A$ and $x=(x_k)_{k\geq0}\in\Hom_{\FF}^{\cont}(D_A(\pi),\FF)$, we have $a(x)=(y_k)_{k\geq0}$ with
\begin{equation}\label{General Eq formula A-module on seq}
    y_k=(\un{Y}^{\un{\ell}-\un{k}}a)x_{\ell}
\end{equation}
for $\ell\gg_k0$. See \cite[Remark~3.8.2]{BHHMS3} for the explanation of (\ref{General Eq formula A-module on seq}).

\hspace{\fill}

For $v\in\pi$, we define $\deg(v)\eqdef\min\set{n\geq-1:v\in\pi[\fm_{I_1}^{n+1}]}\in\ZZ_{\geq-1}$ as in \cite[\S3.5]{BHHMS3}. The following proposition is a generalization of \cite[Prop.~3.5.1]{BHHMS3} (where $\rhobar$ was assumed to be semisimple), and we refer to \S\ref{General Sec app degree} for its proof (see Proposition \ref{General Prop degree appendix}).

\begin{proposition}\label{General Prop degree}
    For $J\subseteq\cJ$ and $\un{i}\in\ZZ^f$ such that $\un{i}\geq\un{e}^{J^{\sh}}$, we have $\deg\bigbra{x_{J,\un{i}}}=\norm{\un{i}}-\abs{J^{\sh}}$.
\end{proposition}

For $J\subseteq\cJ$, we define the sequence $x_J=(x_{J,k})_{k\geq0}$ by $x_{J,k}\eqdef x_{J,\un{k}}$, which is defined in Theorem \ref{General Thm seq}. Since 
$x_{J,k}\in\pi\bigbbbra{\fm_{I_1}^{kf-\abs{J^{\sh}}+1}}$ for all $k\geq0$ by Proposition \ref{General Prop degree}, we have $x_J\in\Hom_{\FF}^{\cont}(D_A(\pi),\FF)$. Then we have the following generalization of \cite[Thm.~3.7.1]{BHHMS3} (where $\rhobar$ was assumed to be semisimple).

\begin{theorem}\label{General Thm rank 2^f}
    The sequences $\sset{x_J:J\subseteq\cJ}$ are contained in the image of the injection
    \begin{equation*}
        \mu_{*}:\Hom_A(D_A(\pi),A)\to\Hom_{\FF}^{\cont}(D_A(\pi),\FF)
    \end{equation*}
    and form an $A$-basis of $\Hom_A(D_A(\pi),A)$. In particular, $D_A(\pi)$ is a free $A$-module of rank $2^f$.
\end{theorem}

\begin{proof}
    We follow closely the proof of \cite[Thm.~3.7.1]{BHHMS3} and use without comment the notation of \loc.

    First, the proof of \loc\,using Theorem \ref{General Thm finiteness} shows that each sequence $x_J\in\Hom_{\FF}^{\cont}(D_A(\pi),\FF)$ comes from an element of $\Hom_A(D_A(\pi),A)$, and we still denote it by $x_J$.

    For each $J\subseteq\cJ$, we define another sequence $x'_J=\bigbra{x'_{J,k}}_{k\geq0}$ by $x'_{J,k}\eqdef x_{J,\un{k}+\un{e}^{J^{\sh}}}$. In particular, we have $x'_{J,0}=v_J$ by (\ref{General Eq seq small}). By (\ref{General Eq formula A-module on seq}) we have $x'_J=\un{Y}^{-\un{e}^{J^{\sh}}}x_J$, which implies that $x'_J\in\Hom_{\FF}^{\cont}(D_A(\pi),\FF)$ and comes from an element of $\Hom_A(D_A(\pi),A)$ (recall $\un{Y}^{-\un{e}^{J^{\sh}}}\in A$), and we still denote it by $x'_J$.

    Since $\un{Y}^{-\un{e}^{J^{\sh}}}$ is invertible in $A$, to prove that $\set{x_J,J\subseteq\cJ}$ form an $A$-basis of $\Hom_A(D_A(\pi),A)$, it suffices to show that $\set{x'_J,J\subseteq\cJ}$ form an $A$-basis of $\Hom_A(D_A(\pi),A)$. As in the proof of \cite[Thm.~3.7.1]{BHHMS3}, it suffices to show that the elements $\set{\gr(x'_J):J\subseteq\cJ}$ form a $\gr A$-basis of $\Hom_{\gr A}(\gr D_A(\pi),\gr A)$.
    
    Since $\pi^{I_1}=D_0(\rhobar)^{I_1}$ is multiplicity-free by Lemma \ref{General Lem Diamond diagram}(ii) (see the assumptions on $\pi$ above Theorem \ref{General Thm main1}), there exist unique $I$-eigenvectors $v_J^*\in(\pi^{I_1})^{\vee}=\gr_0(\pi^{\vee})$ for $J\subseteq\cJ$ such that $\ang{v_J,v_{J'}^*}=\delta_{J=J'}$. As in the proof of \cite[Lemma~3.7.2]{BHHMS3}, we know that $\gr D_A(\pi)$ is a free $\gr A$-module, and that there exists a surjection of $\gr A$-modules
    \begin{equation}\label{General Eq basis surjection}
        \bigoplus\limits_{J\subseteq\cJ}\gr A\onto\gr D_A(\pi),
    \end{equation}
    sending the standard basis element indexed by $J$ on the left to the image of $v_J^*$ in $\gr D_A(\pi)$ (still denoted $v_J^*$). To complete the proof, it is enough to show that $\ang{\gr(x'_J),v_{J'}^*}=\delta_{J=J'}\un{y}^{-\un{1}}$ in $\gr A$ for all $J,J'\subseteq\cJ$, which implies that the surjection (\ref{General Eq basis surjection}) is an isomorphism. The argument here is completely analogous to that of \cite[Thm.~3.7.1]{BHHMS3}, using Corollary \ref{General Cor degree} and Proposition \ref{General Prop degree}. 
\end{proof}

\hspace{\fill}

By definition, $D_A(\pi)$ is a finite free $(\psi,\OK\x)$-module over $A$ in the sense of \cite[Def.~3.1.2.1]{BHHMS2}. Then the construction of \cite[\S3.2]{BHHMS3} makes $\Hom_A(D_A(\pi),A)(1)$ a $(\varphi,\OK\x)$-module, which is \'etale if and only if $\beta$ as in (\ref{General Eq beta}) is an isomorphism. Here for $D$ a $(\varphi,\OK\x)$-module over $A$, we write $D(1)$ to be $D$ with the action of $\varphi$ unchanged and the action of $a\in\OK\x$ multiplied by $N_{\Fq/\Fp}(\ovl{a})$. Then by \cite[Lemma~3.8.1(ii)]{BHHMS3} and \cite[(114)]{BHHMS3}, under the injection (\ref{General Eq mu*}) the actions of $\varphi$ and $\OK\x$ on $\Hom_A(D_A(\pi),A)(1)$ can be expressed in terms of sequences as follows: 
\begin{enumerate}
    \item 
    for $k\geq0$ and $p\ell\geq k$, we have
    \begin{equation}\label{General Eq formula phi-action}
        \bra{\varphi(x_J)}_k=(-1)^{f-1}\un{Y}^{p\un{\ell}-\un{k}}\smat{p&0\\0&1}x_{J,\ell};
    \end{equation}
    \item
    for $a\in\OK\x$, $k\geq0$ and $\ell\gg_k0$, we have
    \begin{equation}\label{General Eq formula a-action}
        \bra{a(x_J)}_k=\frac{a(\un{Y}^{\un{\ell}})}{\un{Y}^{\un{k}}}\smat{a&0\\0&1}x_{J,\ell}.
    \end{equation}
\end{enumerate}

We denote by $\Mat(\varphi)$ and $\Mat(a)$ ($a\in\OK\x$) the matrices of the actions of $\varphi$ and $\OK\x$ on $\Hom_A(D_A(\pi),A)(1)$ with respect to the basis $\sset{x_J:J\subseteq\cJ}$ of Theorem \ref{General Thm rank 2^f}, whose rows and columns are indexed by the subsets of $\cJ$. For $J,J'\subseteq\cJ$ such that $(J-1)^{\ss}=(J')^{\ss}$, we let 
\begin{equation}\label{General Eq gamma}
    \gamma_{J,J'}\eqdef(-1)^{f-1}\varepsilon_{J'}\mu_{J,J'},
\end{equation}
where $\varepsilon_{J'}$ is defined in (\ref{General Eq epsilonJ}) and $\mu_{J,J'}$ is defined in (\ref{General Eq general mu}). Then by definition and the sentence after (\ref{General Eq general mu}), for $J_1,J_2,J_3,J_4\subseteq\cJ$ such that $(J_1-1)^{\ss}=(J_2-1)^{\ss}=J_3^{\ss}=J_4^{\ss}$ we have $\gamma_{J_1,J_3}/\gamma_{J_1,J_4}=\gamma_{J_2,J_3}/\gamma_{J_2,J_4}$. We define $\gamma_{*,J}/\gamma_{*,J'}$ for $J^{\ss}=(J')^{\ss}$ in a similar way as $\mu_{*,J}/\mu_{*,J'}$.

We give a preliminary result on the actions of $\varphi$ and $\OK\x$ on $\Hom_A(D_A(\pi),A)(1)$ which suffices to finish the proof of Theorem \ref{General Thm main1}. We refer to Appendix \ref{General Sec app okx} for a more detailed study. 

\begin{proposition}\label{General Prop phi-OK* action}
    \begin{enumerate}
    \item 
    We have (see (\ref{General Eq cJ}) for $\un{c}^J$ and (\ref{General Eq rJ}) for $\un{r}^{J\setminus J'}$)
    \begin{equation*}
        \Mat(\varphi)_{J',J+1}=
    \begin{cases}
        \gamma_{J+1,J'}\un{Y}^{-(\un{c}^{J}+\un{r}^{J\setminus J'})}&\text{if}~J^{\ss}\subseteq J'\subseteq J\\
        0&\text{otherwise}.
    \end{cases}
    \end{equation*}
    \item 
    For $a\in[\Fq\x]$, $\Mat(a)$ is a diagonal matrix with $\Mat(a)_{J,J}=\ovl{a}^{\un{r}^{J^c}}$.
    \end{enumerate}
\end{proposition}

\begin{proof}
    (i). Let $J\subseteq\cJ$. For $k\geq0$ and $p\ell\geq k$, by (\ref{General Eq formula phi-action}) and (\ref{General Eq Seq p001}) we have
    \begin{align*}
        \bra{\varphi(x_{J+1})}_k&=(-1)^{f-1}\un{Y}^{p\un{\ell}-\un{k}}\smat{p&0\\0&1}x_{J+1,\un{\ell}}=\sum\limits_{J^{\ss}\subseteq J'\subseteq J}(-1)^{f-1}\varepsilon_{J'}\mu_{J+1,J'}x_{J',\bigbra{p\un{\ell}+\un{c}^{J}+\un{r}^{J\setminus J'}-(p\un{\ell}-\un{k})}}\\
        &=\sum\limits_{J^{\ss}\subseteq J'\subseteq J}\gamma_{J+1,J'}x_{J',\un{c}^{J}+\un{r}^{J\setminus J'}+\un{k}}.
    \end{align*}
    Then using (\ref{General Eq formula A-module on seq}) one easily checks that 
    \begin{equation*}
        \varphi(x_{J+1})=\sum\limits_{J^{\ss}\subseteq J'\subseteq J}\gamma_{J+1,J'}\un{Y}^{-(\un{c}^{J}+\un{r}^{J\setminus J'})}(x_{J'}),
    \end{equation*}
    which proves (i).

    \hspace{\fill}

    (ii). Let $J\subseteq\cJ$ and $a\in[\Fq\x]$. By Corollary \ref{General Cor xJi}(ii) we have
    \begin{equation*}
        \smat{a&0\\0&1}x_{J,\un{i}}=\chi'_J\bigbra{\!\smat{a&0\\0&1}\!}\ovl{a}^{-\un{i}}x_{J,\un{i}}=\ovl{a}^{\un{r}-\un{r}^{J}-\un{i}}x_{J,\un{i}}=\ovl{a}^{\un{r}^{\cJ}-\un{r}^{J}-\un{i}}x_{J,\un{i}}=\ovl{a}^{\un{r}^{J^c}-\un{i}}x_{J,\un{i}},
    \end{equation*}
    where the second equality follows from Lemma \ref{General Lem r and c}(i), the third equality follows from (\ref{General Eq rJ}), and the last equality follows from Lemma \ref{General Lem r and c}(ii). Then for $k\geq0$, by (\ref{General Eq formula a-action}) we have
    \begin{equation*}
        \bra{a(x_J)}_k=\frac{a\bigbra{\un{Y}^{\un{k}}}}{\un{Y}^{\un{k}}}\smat{a&0\\0&1}x_{J,k}=\ovl{a}^{\un{k}}\bbra{\ovl{a}^{\un{r}^{J^c}-\un{k}}x_{J,k}}=\ovl{a}^{\un{r}^{J^c}}x_{J,k},
    \end{equation*}
    where the second equality follows from (\ref{General Eq action on Yj}). Using (\ref{General Eq formula A-module on seq}), we conclude that $a(x_J)=\ovl{a}^{\un{r}^{J^c}}(x_J)$, which proves (ii).
\end{proof}

\begin{proof}[Proof of Theorem \ref{General Thm main1}]
    By Theorem \ref{General Thm rank 2^f}, $\pi$ is in $\cC$ and $D_A(\pi)$ has rank $2^f$. Moreover, by Proposition \ref{General Prop phi-OK* action}(i) we have $\Mat(\varphi)\in\GL_{2^f}(A)$, hence $\Hom_A(D_A(\pi),A)$ is an \'etale $(\varphi,\OK\x)$-module over $A$, which implies that $\beta$ as in (\ref{General Eq beta}) is an isomorphism.
\end{proof}

\section{On the subrepresentations of \texorpdfstring{$\pi$}.}\label{General Sec subrep}

In this section, we finish the proof of Theorem \ref{General Thm main2}, see Theorem \ref{General Thm rank sub}. This theorem is crucially needed to prove that $\pi$ is of finite length in the non-semisimple case in \cite{BHHMS4}. As a corollary, we prove that $\pi$ is generated by $D_0(\rhobar)$ under the assumption that $\pi^{\vee}$ is essentially self-dual of grade $2f$ in the sense of \cite[(176)]{BHHMS2} (which is satisfied for those $\pi$ coming from the cohomology of towers of Shimura curves by \cite[Thm.~8.2]{HW22}), see Corollary \ref{General Cor finite generation}, which gives another proof of \cite[Thm.~1.6]{HW22} (but under a stronger genericity condition).

\begin{lemma}\label{General Lem pi1^K1}
    Let $\pi_1$ be a subrepresentation of $\pi$. Then there exists a set $S$ of subsets of $\cJ$ which is stable under $J\mapsto J-1$, and is moreover stable under taking subsets if $J_{\rhobar}\neq\cJ$, such that
    \begin{equation}\label{General Eq pi1^K1 statement}
    \begin{aligned}
        \JH(\pi_1^{K_1})\cap W(\rhobar^{\ss})&=\sset{\sigma_{\un{e}^J}:J\in S};\\
        \JH(\pi_1^{K_1})&=\sset{\sigma_{\un{b}}\in\JH\bigbra{D_0(\rhobar)}:\set{j:b_j\geq1}\in S},
    \end{aligned}
    \end{equation}
    where $\rhobar^{\ss}$ is the semisimplification of $\rhobar$, $\sigma_{\un{b}}$ and $\un{e}^J$ are defined in \S\ref{General Sec sw}, and see Lemma \ref{General Lem Diamond diagram}(i) for $\JH\bigbra{D_0(\rhobar)}$.
\end{lemma}

\begin{proof}
    We recall from Corollary \ref{General Cor structure of D0} that for each $J\subseteq\cJ$ we have $\un{\varepsilon}^J\in\set{\pm1}^f$ with $\varepsilon^J_j=(-1)^{\delta_{j\notin J}}$. We also recall from (\ref{General Eq aJ}) that $\sigma_{J}=\sigma_{\un{e}^J}$ for $J\subseteq J_{\rhobar}$.
    
    \hspace{\fill}
    
    \noindent \textbf{Claim 1.} If $\sigma_{\un{e}^J}\in\JH\bra{\pi_1^{K_1}}$ for some $J\subseteq\cJ$, then $\pi_1^{K_1}$ contains $I\bigbra{\sigma_{\un{e}^{(J-1)^{\ss}}},\sigma_{\un{e}^{(J-1)^{\ss}}+\un{\varepsilon}^{J-1}}}$ and $I\bigbra{\sigma_{\un{e}^{J^{\ss}}},\sigma_{\un{e}^{J^{\ss}}+\un{\varepsilon}^J}}$ (see Lemma \ref{General Lem p001}(iii) for the notation). 

    \proof We prove the claim by increasing induction on $|J|$. Fix $J\subseteq\cJ$ and assume that $\sigma_{\un{e}^J}\in\JH\bra{\pi_1^{K_1}}$. Since $\pi_1^{K_1}$ is a $\GL_2(\OK)$-subrepresentation of $\pi^{K_1}=D_0(\rhobar)$ and $\sigma_{\un{e}^J}\in\JH\bigbra{D_{0,\sigma_{J^{\ss}}}(\rhobar)}$ by Lemma \ref{General Lem Diamond diagram}(i), we deduce from Corollary \ref{General Cor structure of D0} that $\pi_1^{K_1}$ contains $I\bigbra{\sigma_{\un{e}^{J^{\ss}}},\sigma_{\un{e}^J}}$, which is a subrepresentation of $I\bigbra{\sigma_{\un{e}^{J^{\ss}}},\sigma_{\un{e}^{J^{\ss}}+\un{\varepsilon}^J}}$ by Lemma \ref{General Lem p001}(iii). In particular, for each $J'\subseteq\cJ$ such that $J^{\ss}\subseteq J'\subsetneqq J$, we have $\sigma_{\un{e}^{J'}}\in\JH\bra{\pi_1^{K_1}}$. Then by the induction hypothesis and using $(J')^{\ss}=J^{\ss}$, we deduce that $\pi_1^{K_1}$ contains $I\bigbra{\sigma_{\un{e}^{J^{\ss}}},\sigma_{\un{e}^{J^{\ss}}+\varepsilon^{J'}}}$. In particular, by Lemma \ref{General Lem p001}(iii), $\JH\bra{\pi_1^{K_1}}$ contains all $\sigma_{\un{b}}$ with
    \begin{equation*}
    \begin{cases}
        b_j=\delta_{j\in J^{\ss}}&\text{if}~j\notin J^{\nss}\\
        b_j\in\set{0,-1}&\text{if}~j\in J^{\nss},~j\notin J'\\
        b_j\in\set{0,1}&\text{if}~j\in J^{\nss},~j\in J'.
    \end{cases}    
    \end{equation*}
    By varying $J'$ such that $J^{\ss}\subseteq J'\subsetneqq J$ and using $\sigma_{\un{e}^J}\in\JH\bra{\pi_1^{K_1}}$, we deduce that $\JH\bra{\pi_1^{K_1}}$ contains $\sigma_{\un{b}}$ with
    \begin{equation*}
    \begin{cases}
        b_j=\delta_{j\in J^{\ss}}&\text{if}~j\notin J^{\nss}\\
        b_j\in\set{-1,0,1}&\text{if}~j\in J^{\nss}.
    \end{cases}    
    \end{equation*}
    Hence we have $\un{Y}^{-\un{1}}v_J\in\pi_1^{K_1}$ by Proposition \ref{General Prop shift}. 

    By Lemma \ref{General Lem p001}(i),(iii), we have
    \begin{equation}\label{General Eq pi1^K1 claim 3}
        \bang{\GL_2(\OK)\smat{p&0\\0&1}\un{Y}^{-\un{1}}v_J}=I\bigbra{\sigma_{\un{e}^{(J-1)^{\ss}}},\sigma_{\un{e}^{(J-1)^{\ss}}+\un{c}}}\subseteq\pi_1
    \end{equation}
    with $c_j=(-1)^{\delta_{j+1\notin J}}\bigbra{3-\delta_{j\in J\Delta(J-1)^{\ss}}}$. Since $\varepsilon^{J-1}_j=(-1)^{\delta_{j+1\notin J}}$ and $3-\delta_{j\in J\Delta(J-1)^{\ss}}\geq1$, we deduce from (\ref{General Eq pi1^K1 claim 3}) and Corollary \ref{General Cor structure of D0} that
    \begin{equation}\label{General Eq pi1^K1 claim 1}
        I\bigbra{\sigma_{\un{e}^{(J-1)^{\ss}}},\sigma_{\un{e}^{(J-1)^{\ss}}+\un{\varepsilon}^{J-1}}}\subseteq \pi_1\cap\pi^{K_1}=\pi_1^{K_1},
    \end{equation}
    which proves the first part of the claim.

    By Lemma \ref{General Lem p001}(iii) and (\ref{General Eq pi1^K1 claim 1}), we have $\sigma_{\un{e}^{J-1}}\in\JH(\pi_1^{K_1})$. Continuing the above process with $J$ replaced with $J-1$ and so on, we deduce that $\pi_1^{K_1}$ contains $I\bigbra{\sigma_{\un{e}^{(J-i)^{\ss}}},\sigma_{\un{e}^{(J-i)^{\ss}}+\un{\varepsilon}^{J-i}}}$ for all $i\geq0$. In particular, the second part of the claim follows by taking $i=f$ since $J-f=J$.\qed

    \hspace{\fill}

    \noindent \textbf{Claim 2.} Suppose that $J_{\rhobar}\neq\cJ$. If $\sigma_{\un{e}^J}\in\JH\bra{\pi_1^{K_1}}$ for some $J\subseteq\cJ$, then $\sigma_{\un{e}^{J'}}\in\JH\bra{\pi_1^{K_1}}$ for all $J'\subseteq J$.

    \proof Without loss of generality, we may assume that $|J\setminus J'|=1$ and write $J\setminus J'=\set{j_0}$ for some $j_0\in J$. Since $J_{\rhobar}\neq\cJ$, by replacing $(J,J',j_0)$ with $(J+i,J'+i,j_0+i)$ for some $0\leq i\leq f-1$ using Claim 1, we may assume that $j_0\notin J_{\rhobar}$, which implies $J^{\ss}\subseteq J'\subsetneqq J$. Then we have $\sigma_{\un{e}^{J'}}\in\JH\bra{\pi_1^{K_1}}$ by the first paragraph of the proof of Claim 1.\qed

    \hspace{\fill}

    We let $S\eqdef\bigset{J\subseteq\cJ:\sigma_{\un{e}^J}\in\JH\bra{\pi_1^{K_1}}}$. Then by Claim 1 and Claim 2, $S$ is stable under $J\mapsto J-1$, and is moreover stable under taking subsets if $J_{\rhobar}\neq\cJ$. By (\ref{General Eq SW of rhobar}), we have $W(\rhobar^{\ss})=\bigset{\sigma_{\un{e}^J}:J\subseteq\cJ}$, hence the first formula of (\ref{General Eq pi1^K1 statement}) follows from the definition of $S$. 

    Then by Claim 1, we have
    \begin{equation}\label{General Eq pi1^K1 inclusion}
        \pi'\eqdef\sum\limits_{J\in S}I\bigbra{\sigma_{\un{e}^{J^{\ss}}},\sigma_{\un{e}^{J^{\ss}}+\un{\varepsilon}^J}}\subseteq\pi_1^{K_1}\subseteq\pi^{K_1}=D_0(\rhobar).
    \end{equation}
    Since $\JH(\pi')=\bigset{\sigma_{\un{b}}\in\JH\bigbra{D_0(\rhobar)}:\set{j:b_j\geq1}\in S}$ by Lemma \ref{General Lem p001}(iii), to prove the second formula of (\ref{General Eq pi1^K1 statement}), it suffices to show that the first inclusion in \eqref{General Eq pi1^K1 inclusion} is an equality. 

    Suppose on the contrary that the first inclusion in \eqref{General Eq pi1^K1 inclusion} is strict, then there exists $\sigma_{\un{b}}\in\pi_1^{K_1}\subseteq D_0(\rhobar)$ such that $J_0\eqdef\set{j:b_j\geq1}\notin S$. By Corollary \ref{General Cor structure of D0}, $\pi_1^{K_1}$ must contain $I\bigbra{\sigma_{\un{e}^{J_0^{\ss}}},\sigma_{\un{b}}}$, which contains $\sigma_{\un{e}^{J_0}}$ as a constituent by Lemma \ref{General Lem p001}(iii). This contradicts the definition of $S$.
\end{proof}

\begin{theorem}\label{General Thm rank sub}
    Let $\pi_1$ be a subrepresentation of $\pi$. Then we have 
    \begin{equation*}
        \rank_AD_A(\pi_1)=\babs{\JH\bra{\pi_1^{K_1}}\cap W(\rhobar^{\ss})}.
    \end{equation*}
\end{theorem}

\begin{proof}
    Recall from Theorem \ref{General Thm rank 2^f} that $\Hom_A(D_A(\pi),A)$ has rank $2^f$ with $A$-basis $\set{x_J:J\subseteq\cJ}$, where $x_J$ is defined before Theorem \ref{General Thm rank 2^f}. Let $S$ be the set of subsets of $\cJ$ in Lemma \ref{General Lem pi1^K1}. It suffices to show that $\set{x_J:J\in S}$ form an $A$-basis of $\Hom_A(D_A(\pi_1),A)\into\Hom_A(D_A(\pi),A)$.

    \hspace{\fill}

    First we prove that $x_J$ is an element of $\Hom_A(D_A(\pi_1),A)$ for all $J\in S$. By Proposition \ref{General Prop shift}, for all $J\in S$ and $\un{0}\leq\un{i}\leq\un{f}-\un{e}^{J^{\sh}}$, the element $\un{Y}^{-\un{i}}v_J\in D_0(\rhobar)$ lies in the subrepresentation of $D_0(\rhobar)$ with constituents $\sigma_{\un{b}}$ for $\un{b}$ as in (\ref{General Eq range of b}). Hence we have $\un{Y}^{-\un{i}}v_J\in\pi_1^{K_1}$ for all $J\in S$ and $\un{0}\leq\un{i}\leq\un{f}-\un{e}^{J^{\sh}}$ by the second equality in \eqref{General Eq pi1^K1 statement}, which implies that $x_{J,\un{i}}\in\pi_1$ for all $J\in S$ and $\un{i}\leq\un{f}$ by (\ref{General Eq seq small}). Since $S$ is stable under $J\mapsto J-1$, and is moreover stable under taking subsets if $J_{\rhobar}\neq\cJ$, using (\ref{General Eq Seq def}), an increasing induction on $|J|$ and on $\max_ji_j$ shows that $x_{J,i}\in\pi_1$ for all $J\in S$ and $\un{i}\in\ZZ^f$ (if $J_{\rhobar}=\cJ$ then we have $J^{\ss}=J$ for all $J\subseteq\cJ$ and we only use increasing induction on $\max_ji_j$). By the definition of $x_J$ and Proposition \ref{General Prop degree}, we have $x_J\in\Hom_{\FF}^{\cont}(D_A(\pi_1),\FF)$ for all $J\in S$. Then as in the proof of \ref{General Thm rank 2^f}, we deduce from Theorem \ref{General Thm finiteness} that $x_J\in\Hom_A(D_A(\pi_1),A)$ for all $J\in S$.

    \hspace{\fill}

    Next we prove that any element of $\Hom_A(D_A(\pi_1),A)$ is an $A$-linear combination of $x_J$ for $J\in S$. Suppose on the contrary that $\sum\nolimits_{i=1}^ma_{J_i} x_{J_i}\in\Hom_A(D_A(\pi_1),A)$ for $J_i\notin S$ distinct and $a_{J_i}\in A\setminus\set{0}$. We let $J_0$ be a maximal (under inclusion) element among those $J_i$ such that $\deg(a_{J_i})-|\partial J_i|$ is minimal for $1\leq i\leq m$ (see the proof of Lemma \ref{General Lem theta map}(iii) for the definition of the degree and see Remark \ref{General Rk Delta and partial} for $\partial J_i$). Up to rescaling $\sum\nolimits_{i=1}^ma_{J_i}x_{J_i}$ by a suitable $\lambda\un{Y}^{\un{s}}\in A\x$ with $\lambda\in\FF\x$ and $\un{s}\in\ZZ^f$, we may assume that 
    \begin{equation}\label{General Eq rank sub 3}
        a_{J_0}=\un{Y}^{-\un{e}^{J_0\setminus\partial J_0}}+\bigbra{\text{terms of degree}\geq-|J_0\setminus\partial J_0|~\text{and not in}~\FF\un{Y}^{-\un{e}^{J_0\setminus\partial J_0}}},
    \end{equation}
    which has degree $-|J_0\setminus\partial J_0|$, hence we have $\deg(a_{J_0})-|\partial J_0|=-|J_0|$. Then by the assumption on $J_0$, we have
    \begin{equation}\label{General Eq rank sub 4}
    \begin{cases}
        \deg(a_{J_i})\geq-|J_0|+|\partial J_i|+1=-\bigbra{|J_i\setminus\partial J_i|-|J_i\setminus J_0|-1}&\text{if}~J_i\supsetneqq J_0\\
        \deg(a_{J_i})\geq-|J_0|+|\partial J_i|\geq-f&\text{if}~J_i\nsupseteq J_0.
    \end{cases}
    \end{equation}
    
    We define the following $\GL_2(\OK)$-subrepresentation of $D_0(\rhobar)$ (see Corollary \ref{General Cor structure of D0}):
    \begin{equation*}
        V\eqdef
    \begin{cases}
        \sum\limits_{J\nsupseteq J_0}I\bigbra{\sigma_{\un{e}^{J^{\ss}}},\sigma_{\un{e}^{J^{\ss}}+\un{\varepsilon}^J}}&\text{if}~J_{\rhobar}\neq\cJ\\
        \sum\limits_{J\neq J_0}I\bigbra{\sigma_{\un{e}^{J^{\ss}}},\sigma_{\un{e}^{J^{\ss}}+\un{\varepsilon}^J}}=\bigoplus\limits_{J\neq J_0}D_{0,\sigma_{J}}(\rhobar)&\text{if}~J_{\rhobar}=\cJ,
    \end{cases}
    \end{equation*}
    By Lemma \ref{General Lem p001}(iii), $V$ has constituents $\sigma_{\un{b}}$ with $\sset{j:b_j\geq1}\nsupseteq J_0$ if $J_{\rhobar}\neq\cJ$, and $\sset{j:b_j\geq1}\neq J_0$ if $J_{\rhobar}=\cJ$. In particular, since $J_0\notin S$ and $S$ is stable under taking subsets if $J_{\rhobar}\neq\cJ$, we deduce from the second equality in (\ref{General Eq pi1^K1 statement}) that $\pi_1^{K_1}\subseteq V$. 
    
    \hspace{\fill}
    
    \noindent\textbf{Claim.} We have the following properties:
    \begin{enumerate}
    \item
    If $J\nsupseteq J_0$ and $\norm{\un{i}}\leq f$, then $x_{J,\un{i}}\in V$.
    \item 
    If $J\supsetneqq J_0$ and $\norm{\un{i}}\leq|J\setminus\partial J|-|J\setminus J_0|-1$, then $x_{J,\un{i}}\in V$.
    \item
    If $J=J_0$ and $\norm{\un{i}}\leq|J_0\setminus\partial J_0|$, then we have $x_{J,\un{i}}\in V$ if $\un{i}\neq\un{e}^{J_0\setminus\partial J_0}$, and $x_{J,\un{e}^{J_0\setminus\partial J_0}}\in D_0(\rhobar)\setminus V$.
    \end{enumerate}
    
    \proof (i). By (\ref{General Eq seq small}), we may assume that $\un{i}\geq\un{e}^{J^{\sh}}$, in which case we have $x_{J,i}=\un{Y}^{-\un{i}'}v_J\in D_0(\rhobar)$ with $\un{i}'\eqdef\un{i}-\un{e}^{J^{\sh}}$, which is defined in Proposition \ref{General Prop shift} since $\un{0}\leq\un{i}'\leq\un{f}-\un{e}^{J^{\sh}}$. We let $\sigma_{\un{b}}\in\JH\bigbra{D_0(\rhobar)}$ be an arbitrary constituent of the $\GL_2(\OK)$-subrepresentation of $D_0(\rhobar)$ generated by $x_{J,\un{i}}$. Then by \eqref{General Eq range of b} we have
    \begin{equation}\label{General Eq rank sub 1}
    \begin{aligned}
        \sset{j:b_j\geq1}&\subseteq J^{\ss}\sqcup\sset{j\in J^{\nss}:i'_j>0}\sqcup\sset{j:j\in(\partial J)^{\nss}:i'_j=0}\\
        &=J\setminus\sset{j\in(J\setminus\partial J)^{\nss}:i'_j=0}\subseteq J.
    \end{aligned}
    \end{equation}
    In particular, if $J\nsupseteq J_0$, then we deduce from (\ref{General Eq rank sub 1}) that $\set{j:b_j\geq1}\nsupseteq J_0$, hence $\sigma_{\un{b}}$ is a constituent of $V$, which proves that $x_{J,\un{i}}\in V$.

    (ii). By (\ref{General Eq rank sub 1}), we also have
    \begin{equation}\label{General Eq rank sub 2}
        \bigabs{\set{j:b_j\geq1}}\leq|J|-\babs{\sset{j\in(J\setminus\partial J)^{\nss}:i'_j=0}}\leq|J|-\bigbra{\babs{(J\setminus\partial J)^{\nss}}-\norm{\un{i}'}},
    \end{equation}
    where the second inequality becomes an equality if and only if $\un{i}'=\un{e}^{J_1}$ for some $J_1\subseteq(J\setminus\partial J)^{\nss}$, which is equivalent to $\un{i}=\un{e}^{J_2}$ for some $J^{\sh}\subseteq J_2\subseteq J\setminus\partial J$. If $\norm{\un{i}}\leq|J\setminus\partial J|-|J\setminus J_0|-1$, which implies $\norm{\un{i}'}\leq\bigabs{(J\setminus\partial J)^{\nss}}-|J\setminus J_0|-1$, then we deduce from (\ref{General Eq rank sub 2}) that $\bigabs{\set{j:b_j\geq1}}\leq|J_0|-1$, which implies $\set{j:b_j\geq1}\nsupseteq J_0$, hence $\sigma_{\un{b}}$ is a constituent of $V$, which proves that $x_{J,\un{i}}\in V$.

    (iii). Suppose that $J=J_0$ and $\norm{\un{i}}\leq|J_0\setminus\partial J_0|$. Then by \eqref{General Eq rank sub 2} we have 
    \begin{equation*}
        \bigabs{\set{j:b_j\geq1}}\leq\abs{J_0}-\bigbra{\abs{(J_0\setminus\partial J_0)^{\nss}}-\norm{\un{i}'}}=\abs{J_0}-\bigbra{\abs{J_0\setminus\partial J_0}-\norm{\un{i}}}\leq|J_0|,
    \end{equation*}
    and at least one inequality is strict if $\un{i}\neq\un{e}^{J_0\setminus\partial J_0}$, in which case we have $\set{j:b_j\geq1}\nsupseteq J_0$, hence $\sigma_{\un{b}}$ is a constituent of $V$. This proves  $x_{J_0,\un{i}}\in V$ if $\norm{\un{i}}\leq|J_0\setminus\partial J_0|$ and $\un{i}\neq\un{e}^{J_0\setminus\partial J_0}$. Finally, the $\GL_2(\OK)$-subrepresentation of $D_0(\rhobar)$ generated by $\un{Y}^{-\un{e}^{(J_0\setminus\partial J_0)^{\nss}}}v_{J_0}$ has $\sigma_{\un{e}^{J_0}}$ as a constituent by (\ref{General Eq range of b}), hence it follows from (\ref{General Eq seq small}) and the description of the constituents of $V$ that $x_{J_0,\un{e}^{J_0\setminus\partial J_0}}=\un{Y}^{-\un{e}^{(J_0\setminus\partial J_0)^{\nss}}}v_{J_0}\in D_0(\rhobar)\setminus V$.\qed

    \hspace{\fill}

    Using (\ref{General Eq mu*}), we identify $\sum\nolimits_{i=1}^ma_{J_i}x_{J_i}\in\Hom_A\bra{D_A(\pi_1),A}$ as $(z_k)_{k\geq0}\in\Hom_{\FF}^{\cont}\bra{D_A(\pi_1),\FF}$, which is a sequence of elements of $\pi_1$. By writing each $a_{J_i}\in A$ as an infinite sum of monomials in $\un{Y}$ together with (\ref{General Eq rank sub 3}) and (\ref{General Eq rank sub 4}), we deduce from (\ref{General Eq formula A-module on seq}), \cite[Remark~3.8.2]{BHHMS3} and the definition of $x_{J_i}$ that the zeroth term $z_0$ of the sequence $\sum\nolimits_{i=1}^ma_{J_i}x_{J_i}$ is a linear combination of $x_{J,\un{i}}$ satisfying the assumptions of (i),(ii),(iii) of the claim above and with exactly one of the terms equals $x_{J_0,\un{e}^{J_0\setminus\partial J_0}}$, hence is an element of $D_0(\rhobar)\setminus V$. By definition, we also have $z_0\in\pi_1$, hence $z_0\in\pi_1\cap D_0(\rhobar)=\pi_1^{K_1}\subseteq V$, which is a contradiction.
\end{proof}

\begin{corollary}\label{General Cor finite generation}
    Assume moreover that $\pi^{\vee}$ is essentially self-dual of grade $2f$ in the sense of \cite[(176)]{BHHMS2}. Then as a $\GL_2(K)$-representation, $\pi$ is generated by $D_0(\rhobar)$.
\end{corollary}

\begin{proof}
    We use the notation of \cite[Prop.~3.3.5.3]{BHHMS2}. By Theorem \ref{General Thm main1} and the proof of \cite[Prop.~3.3.5.3(i)]{BHHMS2}, we deduce that $\dim_{\FF\dbra{X}}D_{\xi}^{\vee}(\pi')=\fm_{\fp_0}(\pi^{\prime\vee})$ for any subquotient $\pi'$ of $\pi$. By Theorem \ref{General Thm rank sub} and \cite[Remark~3.3.5.4(ii)]{BHHMS2}, we deduce that $D_{\xi}^{\vee}(\pi')\neq0$ for $\pi'$ a subrepresentation of $\pi$. Then the proof of \cite[Prop.~3.3.5.3(iii)]{BHHMS2} shows that $D_{\xi}^{\vee}(\pi')\neq0$ for $\pi'$ a quotient of $\pi$, see \cite[Remark~3.3.5.4(i)]{BHHMS2}. Then we can conclude as in the proof of \cite[Thm.~3.3.5.5]{BHHMS2}.
\end{proof}

\appendix
\numberwithin{equation}{section}
\section*{Appendix}
\section{Some vanishing results}\label{General Sec app vanish}

In this appendix, we give a careful study of the elements $x_{J,\un{i}}$ defined in \S\ref{General Sec xJi} for small $\un{i}$. In particular, we use the results of \S\ref{General Sec relation} to prove that many of these elements are zero. The main results are Proposition \ref{General Prop more 0}, Corollary \ref{General Cor w 0} and Proposition \ref{General Prop xJ' xJ}. The results of this appendix are needed in the proof of Theorem \ref{General Thm seq}, and to prove the finiteness results in \S\ref{General Sec finiteness}.

We refer to Definition \ref{General Def rcepi} for $\un{r}^J,\un{c}^J\in\ZZ^f$ and $\varepsilon_J\in\set{\pm1}$, and refer to (\ref{General Eq seq small}) and (\ref{General Eq Seq def}) for the definition of $x_{J,\un{i}}\in\pi$.

\begin{lemma}\label{General Lem simple 0}
    Let $0\leq k\leq f$. Assume that Theorem \ref{General Thm seq} is true for $\abs{J}\leq k$. Let $J\subseteq\cJ$ with $|J|\leq k$, $j_0\in\cJ$ and $\un{i}\in\ZZ^f$. Suppose that $j_0\notin J$, $j_0+1\notin J$ and $i_{j_0}<0$. Then we have $x_{J,\un{i}}=0$.
\end{lemma}

\begin{proof}
    We prove the result by increasing induction on $|J|\leq k$ and on $\max_ji_j$ (the base case being $|J|=-1$, which is automatic). We let $J\subseteq\cJ$ such that $|J|\leq k$ and $\un{i}\in\ZZ^f$. If $\un{i}\leq\un{f}$, then the lemma follows directly from (\ref{General Eq seq small}). If $\max_ji_j>f$, then we write $\un{i}=p\delta(\un{i}')+\un{c}^{J+1}-\un{\ell}$ for the unique $\un{i}',\un{\ell}\in\ZZ^f$ such that $\un{0}\leq\un{\ell}\leq\un{p}-\un{1}$. In particular, we have $\max_j{i'_j}<\max_j{i_j}$ (see (\ref{General Eq max1<max2})). Since $i_{j_0}<0$ and $c^{J+1}_{j_0}=p-1$ by (\ref{General Eq cJ}), we have $i'_{j_0+1}<0$, hence by the induction hypothesis on $\max_j{i_j}$ we deduce that $x_{J+1,\un{i}'}=0$. For each $J'\subsetneqq J$, by (\ref{General Eq rJ}) we have $i_{j_0}+r^{J\setminus J'}_{j_0}=i_{j_0}<0$, hence by the induction hypothesis on $|J|$ we deduce that $x_{J',\un{i}+\un{r}^{J\setminus J'}}=0$ for $J'\subsetneqq J$. Then by (\ref{General Eq Seq def}) we conclude that $x_{J,\un{i}}=0$.
\end{proof}

To prove more vanishing results, we need some variants of \cite[Lemma~3.4.2]{BHHMS3} and \cite[Lemma~3.4.4]{BHHMS3} (where $\rhobar$ was assumed to be semisimple). Since there could be overlaps between different $\GL_2(\OK)$-subrepresentations $\bang{\GL_2(\OK)\smat{p&0\\0&1}\un{Y}^{-\un{i}}v_J}$ of $\pi$ (see Proposition \ref{General Prop relation 2}), we need to be more precise about the region where the elements $x_{J,\un{i}}$ vanish. This motivates the following somewhat technical definition.

\begin{definition}\label{General Def aJn}
    Let $J\subseteq\cJ$.
    \begin{enumerate}
    \item 
    Let $j\in\cJ$ and $x\in\ZZ$. We write $x=2n+\delta$ with $n\in\ZZ$ and $\delta\in\set{0,1}$. Then we define $t^J_j(x)\eqdef np+\delta\bbra{\delta_{j+1\notin J}(r_j+1)+\delta_{j+1\in J}(p-1-r_j)}$.
    \item 
    Let $\un{n}\in\ZZ^f$. Suppose that there exists $j_0\in\cJ$ such that
    \begin{enumerate}
    \item
    $n_{j_0+1}=0$;
    \item 
    $1\leq n_j\leq 2f-\delta_{j\in J}$ if $j\neq j_0+1$,
    \end{enumerate}
    then we define  $\un{a}^J(\un{n})\in\ZZ^f$ by
    \begin{equation*}
        a^J(\un{n})_j\eqdef 
    \begin{cases}
        t^J_{j_0}(n_{j_0+1})=0&\text{if}~j=j_0~\text{and}~j_0\in J^{\sh}\\
        t^J_j(n_{j+1})-n_j&\text{otherwise}.
    \end{cases}
    \end{equation*}
    \end{enumerate}
\end{definition}

\begin{lemma}\label{General Lem a=a+r}
    Let $J\subseteq\cJ$, $\un{n}\in\ZZ^f$ and $j_0\in\cJ$ as in Definition \ref{General Def aJn}(ii). Let $J'\subseteq J$ such that $j_0+1\notin J\setminus J'$. Suppose that either $j_0\notin J^{\sh}$ or $J^{\ss}\cup\set{j_0+1}\subseteq J'$, then we have (see \S\ref{General Sec sw} for $\un{e}^{J\setminus J'}$)
    \begin{equation*}
        \un{a}^{J}(\un{n})+\un{r}^{J\setminus J'}=\un{a}^{J'}\bigbra{\un{n}+\un{e}^{J\setminus J'}}.
    \end{equation*}
\end{lemma}

\begin{proof}
    Since $j_0+1\notin J\setminus J'$, we have $\bigbra{\un{n}+\un{e}^{J\setminus J'}}_{j_0+1}=n_{j_0+1}+\delta_{j_0+1\in J\setminus J'}=0$. By definition, we also have for $j\neq j_0+1$
    \begin{equation*}
        1\leq n_j\leq n_j+\delta_{j\in J\setminus J'}\leq(2f-\delta_{j\in J})+\delta_{j\in J\setminus J'}=2f-\delta_{j\in J'}.
    \end{equation*}
    Hence $\un{a}^{J'}\bigbra{\un{n}+\un{e}^{J\setminus J'}}$ is well-defined.

    First we suppose that $j_0\notin J^{\sh}$. We need to prove that for each $j\in\cJ$ we have
    \begin{equation*}
        t^{J}_j(n_{j+1})-n_j+r^{J\setminus J'}_j=t^{J'}_j\bigbra{n_{j+1}+\delta_{j+1\in J\setminus J'}}-(n_j+\delta_{j\in J\setminus J'}).
    \end{equation*}
    Since $r^{J\setminus J'}_j=\delta_{j+1\in J\setminus J'}(r_j+1)-\delta_{j\in J\setminus J'}$ by (\ref{General Eq rJ}) (see (\ref{General Eq Lem rJ}) below), it suffices to show that for each $j\in\cJ$ we have
    \begin{equation}\label{General Eq t=t+r}
        t^{J}_j(n_{j+1})+\delta_{j+1\in J\setminus J'}(r_j+1)=t^{J'}_j\bigbra{n_{j+1}+\delta_{j+1\in J\setminus J'}}.
    \end{equation}
    We fix $j\in\cJ$ and write $n_{j+1}=2n+\delta$ with $n\in\ZZ$ and $\delta\in\set{0,1}$. If $\delta=0$, then we have (as $\delta_{j+1\in J\setminus J'}\delta_{j+1\in J'}=0$)
    \begin{align*}
        t^{J'}_j\bigbra{n_{j+1}+\delta_{j+1\in J\setminus J'}}&=np+\delta_{j+1\in J\setminus J'}\bigbra{\delta_{j+1\notin J'}(r_j+1)+\delta_{j+1\in J'}(p-1-r_j)}\\
        &=np+\delta_{j+1\in J\setminus J'}(r_j+1)\\
        &=t^J_j(n_{j+1})+\delta_{j+1\in J\setminus J'}(r_j+1).
    \end{align*}
    If $\delta=1$, then we have
    \begin{align*}
    &\begin{aligned}
        t^J_j(n_{j+1})+\delta_{j+1\in J\setminus J'}(r_j+1)&=np+\delta_{j+1\notin J}(r_j+1)+\delta_{j+1\in J}(p-1-r_j)+\delta_{j+1\in J\setminus J'}(r_j+1)\\
        &=np+\delta_{j+1\notin J'}(r_j+1)+\delta_{j+1\in J}(p-1-r_j);
    \end{aligned}\\
    &\begin{aligned}
        t^{J'}_j\bigbra{n_{j+1}+\delta_{j+1\in J\setminus J'}}&=np+\delta_{j+1\notin J'}(r_j+1)+\delta_{j+1\in J'}(p-1-r_j)\\
        &\hspace{1.5cm}+\delta_{j+1\in J\setminus J'}\bigbra{\delta_{j+1\notin J'}(p-1-r_j)+\delta_{j+1\in J'}(r_j+1)}\\
        &=np+\delta_{j+1\notin J'}(r_j+1)+\delta_{j+1\in J'}(p-1-r_j)+\delta_{j+1\in J\setminus J'}(p-1-r_j)\\
        &=np+\delta_{j+1\notin J'}(r_j+1)+\delta_{j+1\in J}(p-1-r_j).
    \end{aligned}   
    \end{align*}
    
    Then it remains to show that $a^{J}(\un{n})_{j_0}+r^{J\setminus J'}_{j_0}=a^{J'}\bigbra{\un{n}+\un{e}^{J\setminus J'}}_{j_0}$ when $j_0\in J^{\sh}$ and $J^{\ss}\cup\set{j_0+1}\subseteq J'\subseteq J$. By assumption we have $j_0\in(J')^{\sh}$, hence by definition we have $a^J(\un{n})_{j_0}=\un{a}^{J'}\bigbra{\un{n}+\un{e}^{J\setminus J'}}_{j_0}=0$. By assumption we also have $j_0,j_0+1\in J'$, hence $j_0,j_0+1\notin J\setminus J'$ and $r^{J\setminus J'}_{j_0}=0$ by (\ref{General Eq rJ}). This completes the proof.
\end{proof}

\begin{proposition}\label{General Prop more 0}
    Let $0\leq k\leq f$. Assume that Theorem \ref{General Thm seq} is true for $|J|\leq k$. Let $J\subseteq\cJ$ with $|J|\leq k$. Let $\un{n}\in\ZZ^f$ and $j_0\in\cJ$ be as in Definition \ref{General Def aJn}(ii). Then we have $x_{J,\un{a}^J(\un{n})-e_{j_0+1}}=0$.
\end{proposition}

\begin{proof}
    If $f=1$, then we have $a^J(0)-1=-1$, and the proposition follows directly from (\ref{General Eq seq small}). Hence in the rest of the proof we assume that $f\geq2$, and we prove the result by increasing induction on $|J|$. We let $\un{i}\in\ZZ^f$ and $J'\subseteq\cJ$ be the unique pair such that 
    \begin{equation}\label{General Eq more 0 n1}
        n_j=2i_j-\delta_{j\in J+1}+\delta_{j\notin J}+\delta_{j-1\in J'}
    \end{equation}
    for all $j\in\cJ$. In particular, we have $\un{i}\leq\un{f}$ since $n_j\leq2f-\delta_{j\in J}$ for all $j\in\cJ$. 

    \hspace{\fill}

    \noindent\textbf{Claim 1.} We let $\un{i}'\eqdef\un{i}-\un{e}^{(J+1)^{\sh}}$. Then we have for all $j\in\cJ$
    \begin{equation}\label{General Eq more 0 n2}
        n_j=2i'_j-\delta_{j\in(J+1)\Delta J^{\ss}}+\delta_{j\notin J^{\nss}}+\delta_{j-1\in J'}.
    \end{equation}
    Indeed, this follows from (\ref{General Eq more 0 n1}) and the following computation:
    \begin{equation*}
    \begin{aligned}
        &-2\delta_{j\in(J+1)^{\sh}}-\delta_{j\in(J+1)\Delta J^{\ss}}+\delta_{j\notin J^{\nss}}\\
        &\hspace{1.5cm}=-2\delta_{j\in J+1}\delta_{j\in J^{\ss}}-\bigbra{\delta_{j\in J+1}+\delta_{j\in J^{\ss}}-2\delta_{j\in J+1}\delta_{j\in J^{\ss}}}+\bigbra{\delta_{j\notin J}-\delta_{j\in J^{\ss}}}\\
        &\hspace{1.5cm}=-\delta_{j\in J+1}+\delta_{j\notin J}.
    \end{aligned}
    \end{equation*}

    \hspace{\fill}
    
    \noindent\textbf{Claim 2}: We let $\un{c}\in\ZZ^f$ such that (see (\ref{General Eq tJJ'}) for $t^{J+1}(J')_j$) 
    \begin{equation}\label{General Eq more 0 c}
        c_j=pi_{j+1}+c^J_j-\delta_{j\notin J'}\bigbra{2i'_j+t^{J+1}(J')_j}.
    \end{equation}
    If either $j_0\notin J^{\sh}$ or $j_0\in J'$, then we have
    \begin{equation}\label{General Eq more 0 claim statement}
    \begin{cases}
        c_j\geq a^J(\un{n})_j-1&\text{if}~j=j_0+1,~j_0+1\in J'~\text{and}~j_0+1\notin J\\
        c_j\geq a^J(\un{n})_j&\text{otherwise}.
    \end{cases}
    \end{equation}
    
    \proof Indeed, by (\ref{General Eq sJ}) and a case-by-case examination we have
    \begin{equation}\label{General Eq more 0 claim 1}
    \begin{aligned}
        2i'_j+t^{J+1}(J')_j&=2i_j+p-1-\bigbra{s^{J+1}_j+2\delta_{j\in(J+1)^{\sh}}}+\delta_{j-1\in J'}\\
        &=2i_j+\delta_{j\notin J}(p-1-r_j)+\delta_{j\in J}(r_j+1)-\delta_{j\in J+1}+\delta_{j-1\in J'}.
    \end{aligned}
    \end{equation}
    If $j\in J'$, then by definition and a case-by-case examination we have
    \begin{equation}\label{General Eq more 0 claim 2}
    \begin{aligned}
        t^J_j(n_{j+1})&=t^J_j\bigbra{2i_{j+1}-\delta_{j\in J}+\delta_{j+1\notin J}+1}=t^J_j\bigbra{2i_{j+1}+\delta_{j\notin J}+\delta_{j+1\notin J}}\\
        &=pi_{j+1}+\delta_{j\notin J}(p-1-r_j)+\delta_{j+1\notin J}(r_j+1).
    \end{aligned}    
    \end{equation}
    Combining (\ref{General Eq Lem cJ}), (\ref{General Eq more 0 claim 1}) and (\ref{General Eq more 0 claim 2}) we deduce that
    \begin{equation*}
        c_j-a^J(\un{n})_j=n_j-\delta_{j\notin J}\geq-\delta_{j=j_0+1,j_0+1\notin J},
    \end{equation*}
    unless when $j=j_0$ and $j_0\in J^{\sh}$, in which case we have $c_{j_0}-a^J(\un{n})_{j_0}=-\delta_{j_0\notin J}=0$. If $j\notin J'$, then by assumption we have either $j\neq j_0$ or $j_0\notin J^{\sh}$. By definition and a case-by-case examination we have
    \begin{equation}\label{General Eq more 0 claim 3}
        t^J_j(n_{j+1})=t^J_j\bigbra{2i_{j+1}-\delta_{j\in J}+\delta_{j+1\notin J}}
        =pi_{j+1}-\delta_{j\in J}(r_j+1)+\delta_{j+1\notin J}(r_j+1).    
    \end{equation}
    Combining (\ref{General Eq Lem cJ}), (\ref{General Eq more 0 n1}), (\ref{General Eq more 0 claim 1}) and (\ref{General Eq more 0 claim 3}) we deduce that $c_j=a^J(\un{n})_j$.\qed

    \hspace{\fill}
    
    Using the decomposition
    (\ref{General Eq decomposition of big J}), we separate the proof into the following four cases.

    \hspace{\fill}

    (a). Suppose that $j_0\notin J$ and $j_0+1\notin J$ (which implies $j_0\notin J^{\sh}$). Since $n_{j_0+1}=0$ and $n_{j_0}>0$ by assumption (recall that $f\geq2$), we have $\bigbra{\un{a}^J(\un{n})-e_{j_0+1}}_{j_0}=-n_{j_0}-0<0$. Then we deduce from Lemma \ref{General Lem simple 0} that $x_{J,\un{a}^J(\un{n})-e_{j_0+1}}=0$.

    \hspace{\fill}
    
    (b). Suppose that $j_0+1\in(J+1)\Delta J^{\ss}$ and $j_0+1\notin J^{\nss}$ (which implies $j_0\notin J^{\sh}$). Using (\ref{General Eq more 0 n2}), we deduce from $n_{j_0+1}=0$ that $i'_{j_0+1}=0$ and $j_0\notin J'$, and deduce from $\un{i}\leq\un{f}$ that $\un{0}\leq\un{i}'\leq\un{f}-\un{e}^{(J+1)^{\sh}}$. By Proposition \ref{General Prop relation 1} applied to $\bigbra{\un{i}',J+1,J'}$ with $j_0$ as above, we have
    \begin{equation}\label{General Eq more 0 case b 1}
        Y_{j_0+1}^{\delta_{J'=\emptyset}}\bbbra{\scalebox{1}{$\prod$}_{j\notin J'}Y_j^{2i'_j+t^{J+1}(J')_j}}\smat{p&0\\0&1}\bbra{\un{Y}^{-\un{i}'}v_{J+1}}=0.
    \end{equation}
    Multiplying (\ref{General Eq more 0 case b 1}) by $Y_{j_0+1}^{\delta_{j_0+1\notin J'}}$ when $J'\neq\emptyset$ and using (\ref{General Eq seq small}), we deduce that
    \begin{equation*}
        Y_{j_0+1}^{\delta_{j_0+1\notin J'}}\bbbra{\scalebox{1}{$\prod$}_{j\notin J'}Y_j^{2i'_j+t^{J+1}(J')_j}}\smat{p&0\\0&1}x_{J+1,\un{i}}=0.
    \end{equation*}
    Then by (\ref{General Eq Seq p001}) we have (see (\ref{General Eq more 0 c}) for \un{c})
    \begin{equation}\label{General Eq more 0 case b 2}
        \sum\limits_{J^{\ss}\subseteq J_1\subseteq J}\varepsilon_{J_1}\mu_{J+1,J_1}x_{J_1,\bigbra{\un{c}+\un{r}^{J\setminus J_1}-\delta_{j_0+1\notin J'e_{j_0+1}}}}=0.
    \end{equation}
    By (\ref{General Eq more 0 claim statement}) we have $\un{c}\geq\un{a}^J(\un{n})-\delta_{j_0+1\in J'\setminus J}e_{j_0+1}\geq\un{a}^J(\un{n})-\delta_{j_0+1\in J'}e_{j_0+1}$. Moreover, for each $J_1\subseteq\cJ$ such that $J^{\ss}\subseteq J_1\subseteq J$, we have $j_0+1\notin J\setminus J_1$ (since $j_0+1\notin J^{\nss}$), hence by Lemma \ref{General Lem a=a+r} (recall that $j_0\notin J^{\sh}$) we have $\un{a}^{J}(\un{n})+\un{r}^{J\setminus J_1}=\un{a}^{J_1}\bigbra{\un{n}+\un{e}^{J\setminus J_1}}$. In particular, multiplying (\ref{General Eq more 0 case b 2}) by a suitable power of $\un{Y}$ and using Theorem \ref{General Thm seq}(ii) (applied to $J_1$ such that $J^{\ss}\subseteq J_1\subseteq J$) we deduce that
    \begin{equation*}
        \sum\limits_{J^{\ss}\subseteq J_1\subseteq J}\varepsilon_{J_1}\mu_{J+1,J_1}x_{J_1,\un{a}^{J_1}(\un{n}+\un{e}^{J\setminus J_1})-e_{j_0+1}}=0.
    \end{equation*}
    By the induction hypothesis, we have $x_{J_1,\un{a}^{J_1}(\un{n}+\un{e}^{J\setminus J_1})-e_{j_0+1}}=0$ for all $J^{\ss}\subseteq J_1\subsetneqq J$, hence we conclude that $x_{J,\un{a}^J(\un{n})-e_{j_0+1}}=0$.

    \hspace{\fill}
    
    (c). Suppose that $j_0+1\in J^{\nss}$. Using (\ref{General Eq more 0 n2}), we deduce from $n_{j_0+1}=0$ that $i'_{j_0+1}=0$, and $j_0\in J'$ if and only if $j_0+1\in J+1$. Then we deduce from (\ref{General Eq more 0 n2}) and $\un{i}\leq\un{f}$ that $\un{0}\leq\un{i}'\leq\un{f}-\un{e}^{(J+1)^{\sh}}$. By Proposition \ref{General Prop relation 2} applied to $\bigbra{\un{i}',J+1,J'}$ with $j_0$ as above, we have (see the end of \S\ref{General Sec relation} for $\mu_{J_1,*}/\mu_{J_2,*}$)
    \begin{multline*}
        Y_{j_0+1}^{\delta_{j_0\notin J}}\bbbra{\scalebox{1}{$\prod$}_{j\notin J'}Y_j^{2i'_j+t^{J+1}(J')_j}}\smat{p&0\\0&1}\bbra{\un{Y}^{-\un{i}'}v_{J+1}}\\
        =\frac{\mu_{J+1,*}}{\mu_{(J+1)\setminus\set{j_0+2},*}}Y_{j_0+1}^{\delta_{j_0\notin J}}\bbbra{\scalebox{1}{$\prod$}_{j\notin J''}Y_j^{2i''_j+t^{(J+1)\setminus\set{j_0+2}}(J'')_j}}\smat{p&0\\0&1}\bbra{\un{Y}^{-\un{i}''}v_{(J+1)\setminus\set{j_0+2}}},
    \end{multline*}
    where $\un{i}''\eqdef\un{i}'-\delta_{j_0+1\in J'}e_{j_0+2}+\delta_{j_0+2\in J^{\ss}}e_{j_0+2}$ and $J''\eqdef J'\Delta\set{j_0+1}$. As in (b), using (\ref{General Eq seq small}), (\ref{General Eq Seq p001}) and $j_0+1\notin J^{\ss}$, we deduce that
    \begin{multline}\label{General Eq more 0 case c 1}
        \sum\limits_{J^{\ss}\subseteq J_1\subseteq J}\varepsilon_{J_1}\mu_{J+1,J_1}x_{J_1,\bigbra{\un{c}+\un{r}^{J\setminus J_1}-\delta_{j_0\notin J}e_{j_0+1}}}\\
        =\sum\limits_{J^{\ss}\subseteq J_2\subseteq J\setminus\set{j_0+1}}\varepsilon_{J_2}\mu_{J+1,J_2}x_{J_2,\bigbra{\un{c}'+\un{r}^{(J\setminus\set{j_0+1})\setminus J_2}-\delta_{j_0\notin J}e_{j_0+1}}},
    \end{multline}
    where $\un{c}$ is defined in (\ref{General Eq more 0 c}) and $c'_j\eqdef pi'''_{j+1}+c^{J\setminus\set{j_0+1}}_j-\delta_{j\notin J''}\bigbra{2i''_j+t^{(J+1)\setminus\set{j_0+2}}(J'')_j}$ 
    with 
    \begin{multline*}
        \un{i}'''\eqdef\un{i}''+\un{e}^{((J+1)\setminus\set{j_0+2})^{\sh}}=\un{i}''+\un{e}^{(J+1)^{\sh}}-\delta_{j_0+2\in J^{\ss}}e_{j_0+2}\\
        =\un{i}'+\un{e}^{(J+1)^{\sh}}-\delta_{j_0+1\in J'}e_{j_0+2}=\un{i}-\delta_{j_0+1\in J'}e_{j_0+2}.
    \end{multline*}

    \hspace{\fill}

    \noindent\textbf{Claim 3}: We have $\un{c}'=\un{c}+\un{r}^{\set{j_0+1}}$. 
    
    \proof By (\ref{General Eq relation 2 tJ}) we have $c_j=c'_j$ for $j\neq j_0$ and $j\neq j_0+1$. If $j=j_0$, then by (\ref{General Eq relation 2 tJ}) and (\ref{General Eq cJ}) we deduce that $c'_{j_0}-c_{j_0}=c^{J\setminus\set{j_0+1}}_{j_0}-c^J_{j_0}=r_{j_0}+1$. If $j=j_0+1$, we assume that $j_0+2\in J$, the case $j_0+2\notin J$ being similar. Then by (\ref{General Eq relation 2 tJ}) and (\ref{General Eq cJ}) we deduce that
    \begin{align*}
        c'_{j_0+1}&=p\bigbra{i_{j_0+2}-\delta_{j_0+1\notin J'}}+(p-2-r_{j_0+1})-\delta_{j_0+1\in J'}(p-1-r_{j_0+1})\\
        &=pi_{j_0+2}-\delta_{j_0+1\notin J'}(r_{j_0+1}+1)-(\delta_{j_0+1\notin J'}-1+\delta_{j_0+1\in J'})(p-1-r_{j_0+1})-1\\
        &=pi_{j_0+2}-\delta_{j_0+1\notin J'}(r_{j_0+1}+1)-1=c_{j_0+1}-1.
    \end{align*}
    The claim then follows from (\ref{General Eq rJ}).\qed

    \hspace{\fill}
    
    By Lemma \ref{General Lem r and c}(ii), for each $J_2\subseteq J\setminus\set{j_0+1}$ we have $\un{r}^{\set{j_0+1}}+\un{r}^{(J\setminus\set{j_0+1})\setminus J_2}=\un{r}^{J\setminus J_2}$, hence the RHS of (\ref{General Eq more 0 case c 1}) cancels with the terms in the LHS of (\ref{General Eq more 0 case c 1}) for the $J_1$ such that $j_0+1\notin J_1$. Since $j_0+1\in J$, and $j_0\in J'$ whenever $j_0\in J^{\sh}$, by (\ref{General Eq more 0 claim statement}) we have $\un{c}\geq\un{a}^J(\un{n})$. Moreover, for each $J_1\subseteq\cJ$ such that $J^{\ss}\cup\set{j_0+1}\subseteq J_1\subseteq J$, by Lemma \ref{General Lem a=a+r} we have $\un{a}^J(\un{n})+\un{r}^{J\setminus J_1}=\un{a}^{J_1}\bigbra{\un{n}+\un{e}^{J\setminus J_1}}$. Then multiplying (\ref{General Eq more 0 case c 1}) by a suitable power of $\un{Y}$ and using Theorem \ref{General Thm seq}(ii) we deduce that
    \begin{equation*}
        \sum\limits_{J^{\ss}\cup\set{j_0+1}\subseteq J_1\subseteq J}\varepsilon_{J_1}\mu_{J+1,J_1}x_{J_1,\un{a}^{J_1}(\un{n}+\un{e}^{J\setminus J_1})-e_{j_0+1}}=0.
    \end{equation*}
    By the induction hypothesis, we have $x_{J_1,\un{a}^{J_1}(\un{n}+\un{e}^{J\setminus J_1})-e_{j_0+1}}=0$ for all $J^{\ss}\cup\set{j_0+1}\subseteq J_1\subsetneqq J$, hence we conclude that $x_{J,\un{a}^J(\un{n})-e_{j_0+1}}=0$.

    \hspace{\fill}
    
    (d). Suppose that $j_0+1\in(J+1)^{\sh}$. By (\ref{General Eq more 0 n2}) we have $i'_{j_0+1}=-1$ and $j_0\in J'$, hence $x_{J+1,\un{i}}=0$ by (\ref{General Eq seq small}). Then by (\ref{General Eq Seq p001}), we have
    \begin{equation}\label{General Eq more 0 case d 1}
        0=\smat{p&0\\0&1}x_{J+1,\un{i}}=\sum\limits_{J^{\ss}\subseteq J_1\subseteq J}\varepsilon_{J_1}\mu_{J+1,J_1}x_{J_1,p\delta(\un{i})+\un{c}^J+\un{r}^{J\setminus J_1}}.
    \end{equation}
    By (\ref{General Eq more 0 claim statement}) we have $p\delta(\un{i})+\un{c}^J\geq\un{c}\geq\un{a}^{J}(\un{n})-e_{j_0+1}$. Moreover, for each $J_1\subseteq\cJ$ such that $J^{\ss}\subseteq J_1\subseteq J$, by assumption we have $j_0+1\in(J+1)^{\sh}\subseteq J^{\ss}\subseteq J_1$. Then as in (c), we deduce from (\ref{General Eq more 0 case d 1}), Lemma \ref{General Lem a=a+r} and Theorem \ref{General Thm seq}(ii) that
    \begin{equation*}
        \sum\limits_{J^{\ss}\subseteq J_1\subseteq J}\varepsilon_{J_1}\mu_{J+1,J_1}x_{J_1,\un{a}^{J_1}(\un{n}+\un{e}^{J\setminus J_1})-e_{j_0+1}}=0.
    \end{equation*}
    By the induction hypothesis, we have $x_{J_1,\un{a}^{J_1}(\un{n}+\un{e}^{J\setminus J_1})-e_{j_0+1}}=0$ for all $J^{\ss}\subseteq J_1\subsetneqq J$, hence we conclude that $x_{J,\un{a}^J(\un{n})-e_{j_0+1}}=0$.
\end{proof}

\begin{corollary}\label{General Cor w 0}
    Let $0\leq k\leq f$. Suppose that Theorem \ref{General Thm seq} is true for $|J|\leq k-1$. Let $J,J'\subseteq\cJ$ with $|J|\leq k$ and $J^{\ss}\subseteq J'\subsetneqq J\subseteq\cJ$. Let $j_0\in\cJ$ such that $j_0+1\notin J\setminus J'$. Then we have $x_{J',\bigbra{\un{r}^{J\setminus J'}+\un{f}-(f+1-\delta_{j_0\in J^{\sh}})e_{j_0}}}=0$.
\end{corollary}

\begin{proof}
    If $f=1$, then the assumption is never satisfied. Hence in the rest of the proof we assume that $f\geq2$. We let $\un{n}\in\ZZ^f$ such that $n_{j_0+1}=0$, $n_{j_0}=1+\delta_{j_0\in J\setminus J'}$ and $n_j=2$ for $j\neq j_0,j_0+1$. In particular, $\un{n}$ satisfies the conditions in Definition \ref{General Def aJn}(ii) for $J'$ and $j_0$. Since $|J'|\leq k-1$ by assumption, we deduce from Proposition \ref{General Prop more 0} applied to $J'$ and $j_0$ that $x_{J',\un{a}^{J'}(\un{n})-e_{j_0+1}}=0$. Then the result follows Theorem \ref{General Thm seq}(ii) (applied to $J'$) and the Claim below.

    \hspace{\fill}

    \noindent\textbf{Claim.} We have
    \begin{equation*}
        \un{a}^{J'}(\un{n})-e_{j_0+1}\geq\un{r}^{J\setminus J'}+\un{f}-\bigbra{f+1-\delta_{j_0\in J^{\sh}}}e_{j_0}.
    \end{equation*}
    
    \proof If either $j\neq j_0,j_0-1$, or $j=j_0-1$ and $j_0\in J\setminus J'$, then we have $n_{j+1}=2$. Hence
    \begin{equation*}
    \begin{aligned}
        a^{J'}(\un{n})_j-\delta_{j=j_0+1}&=t^{J'}_j(2)-n_j-\delta_{j=j_0+1}\\
        &\geq p-2-1\geq(p-2-2f)+f\geq r_j+1+f\geq r^{J\setminus J'}_j+f,
    \end{aligned}       
    \end{equation*}
    where the third inequality follows from (\ref{General Eq genericity}) and the last inequality follows from (\ref{General Eq rJ}). 
    
    If $j=j_0-1$ and $j_0\notin J\setminus J'$, then we have $n_{j_0}=1$, and $r^{J\setminus J'}_{j_0}\leq0$ by (\ref{General Eq rJ}). Hence
    \begin{align*}
        a^{J'}(\un{n})_{j_0-1}-\delta_{j_0-1=j_0+1}&=t^{J'}_{j_0-1}(1)-n_{j_0-1}-\delta_{j_0-1=j_0+1}\\
        &\geq\bigbra{\delta_{j_0\notin J'}(r_{j_0-1}+1)+\delta_{j_0\in J'}(p-1-r_{j_0-1})}-2-1\\
        &\geq(2f+2)-3\geq f\geq r^{J\setminus J'}_{j_0-1}+f,
    \end{align*}
    where the second inequality follows from (\ref{General Eq genericity}). 
    
    Finally, we let $j=j_0$. Since $j_0+1\notin J\setminus J'$, by (\ref{General Eq rJ}) we have $r^{J\setminus J'}_{j_0}=-\delta_{j_0\in J\setminus J'}$. If $j_0\notin J^{\sh}$, then we have (using $n_{j_0+1}=0$)
    \begin{equation*}
        a^{J'}(\un{n})_{j_0}-\delta_{j_0=j_0+1}=t^{J'}_{j_0}(0)-n_{j_0}=-n_{j_0}=-\delta_{j_0\in J\setminus J'}-1=r^{J\setminus J'}_{j_0}+f-(f+1).
    \end{equation*}
    If $j_0\in J^{\sh}$, then $J^{\sh}\subseteq J^{\ss}\subseteq J'$ implies $j_0\in J'$, hence $j_0\notin J\setminus J'$, which implies $r^{J\setminus J'}_{j_0}=0$. Then we have 
    \begin{equation*}
        a^{J'}(\un{n})_{j_0}-\delta_{j_0=j_0+1}=0=r^{J\setminus J'}_{j_0}+f-f,
    \end{equation*}
    which completes the proof.
\end{proof}

\begin{lemma}\label{General Lem lem for zw}
    Let $J,J'\subseteq\cJ$ such that $J^{\ss}\sqcup(\partial J)^{\nss}\subsetneqq J'\subseteq J$ (see Remark \ref{General Rk Delta and partial} for $\partial J$) and $J^{\nss}\neq\cJ$. Then we have
    \begin{equation}\label{General Eq Y^c-r}
        \un{Y}^{\un{c}^{J'}+(p-1)\un{e}^{J^{\sh}}-\un{r}^{J\setminus J'}}\smat{p&0\\0&1}x_{J'+1,\un{e}^{(J^{\sh}+1)}}=0.
    \end{equation}
\end{lemma}

\begin{proof}
    Our assumption implies $f\geq2$. By Lemma \ref{General Lem r and c}(v), the LHS of (\ref{General Eq Y^c-r}) is well-defined (since $\un{c}^J\geq\un{0}$) and it suffices to show that 
    \begin{equation}\label{General Eq Y^c-r 2}
        \bbbra{\un{Y}^{\un{c}^J+(p-1)\un{e}^{J^{\sh}}}\scalebox{1}{$\prod$}_{j\in J\setminus J'}Y_j^{p-1-r_j}}\smat{p&0\\0&1}x_{J'+1,\un{e}^{(J^{\sh}+1)}}=0.
    \end{equation}

    We let $\un{i}\eqdef\un{e}^{J^{\sh}+1}-\un{e}^{(J'+1)^{\sh}}$. If $\un{i}\ngeq\un{0}$, then by (\ref{General Eq seq small}) we have $x_{J'+1,\un{e}^{J^{\sh}+1}}=0$, which proves (\ref{General Eq Y^c-r 2}). From now on we assume that $\un{i}\geq\un{0}$, which implies $(J'+1)^{\sh}\subseteq J^{\sh}+1$. Then we claim that
    \begin{equation}\label{General Eq Y^c-r 5}
        (J')^{\nss}\cap\bigbra{J^{\ss}-1}=\emptyset.
    \end{equation}
    Otherwise, there exists $j_1\in(J')^{\nss}$ such that $j_1+1\in J^{\ss}=(J')^{\ss}$, which implies $j_1+1\in(J'+1)^{\sh}\subseteq J^{\sh}+1$. Hence $j_1\in J^{\sh}$, which is a contradiction since $j_1\notin J_{\rhobar}$. 

    Since $J^{\nss}\neq\cJ$ by assumption, we divide $J^{\nss}$ into a disjoint union of intervals not adjacent to each other. Since $(\partial J)^{\nss}\subsetneqq(J')^{\nss}\subseteq J^{\nss}$ by assumption, we choose an interval $I$ as above such that $(J'\setminus\partial J)^{\nss}\cap I\neq\emptyset$ and denote by $j_0$ the right boundary of $I$. Since $j_0+1\notin J^{\nss}$ by construction, we have either $j_0+1\in J^{\ss}=(J')^{\ss}$, which implies $j_0\notin J'$ by (\ref{General Eq Y^c-r 5}), or $j_0+1\notin J$, which implies $j_0\in(\partial J)^{\nss}\subseteq J'$. In particular, in both cases we have $j_0\notin(J'\setminus\partial J)^{\nss}$ and $j_0+1\in(J'+1)\Delta(J')^{\ss}$.

    \hspace{\fill}
    
    (a). First we suppose that $j_0+1\notin J$, which implies $j_0\in(\partial J)^{\nss}\subseteq J'$. We let $1\leq w\leq f-1$ be minimal such that $j_0-w\in(J'\setminus\partial J)^{\nss}$. By the construction of $I$ we have $j_0-w,j_0-w+1,\ldots,j_0\in J^{\nss}$ and $j_0-w+1,\ldots,j_0-1\notin J'$. Then by (\ref{General Eq sJ}) we have if $w\geq2$
    \begin{equation}\label{General Eq Y^c-r 4}
        s^{J'+1}_j=
    \begin{cases}
        p-2-r_j&\text{if}~j=j_0\\
        r_j&\text{if}~j=j_0-w+2,\ldots,j_0-1~(\text{and $w\geq3$})\\
        r_j+1&\text{if}~j=j_0-w+1,
    \end{cases}
    \end{equation}
    and $s^{J'+1}_{j_0}=p-1-r_{j_0}$ if $w=1$.
    
    Then we let $J''\eqdef\cJ\setminus\set{j_0-w+1,\ldots,j_0}$. For each $j=j_0-w+1,\ldots,j_0+1$ we have $j-1\notin J_{\rhobar}$ by construction, hence $j\notin J^{\sh}+1$, which implies $i_j\leq0$ and hence $i_j=0$ (since $\un{i}\geq\un{0}$). Then by (\ref{General Eq tJJ'}) and (\ref{General Eq Y^c-r 4}) we have
    \begin{equation*}
    \begin{aligned}
        2i_j+t^{J'+1}(J'')_j&=p-1-s^{J'+1}_j+\delta_{j-1\in J''}\\
    &=\begin{cases}
        r_j+1&\text{if}~j=j_0\\
        p-1-r_j&\text{if}~j=j_0-w+1,\ldots,j_0-1~(\text{and $w\geq2$}).
    \end{cases}
    \end{aligned}
    \end{equation*}
    By Proposition \ref{General Prop relation 1} applied to $(\un{i},J'+1,J'')$ with $j_0$ as above and using (\ref{General Eq seq small}), we have
    \begin{equation*}
        \bbbra{Y_{j_0}^{r_{j_0}+1}\scalebox{1}{$\prod$}_{j=j_0-w+1}^{j_0-1}Y_j^{p-1-r_j}}\smat{p&0\\0&1}x_{J'+1,\un{e}^{(J^{\sh}+1)}}=0.
    \end{equation*}
    Since $j_0-w+1,\ldots,j_0-1\in J\setminus J'$, to prove (\ref{General Eq Y^c-r 2}) it is enough to show that $c^J_{j_0}+(p-1)\delta_{j_0\in J^{\sh}}\geq r_{j_0}+1$, which follows from (\ref{General Eq cJ}) since $j_0\in J$ and $j_0+1\notin J$.

    \hspace{\fill}

    (b). Then we suppose that $j_0+1\in J^{\ss}=(J')^{\ss}$, which implies $j_0\notin J'$. We use the same definition of $w$, $J''$ as in (a). In particular, we still have $j_0-w,j_0-w+1,\ldots,j_0\in J^{\nss}$ and $j_0-w+1,\ldots,j_0-1\notin J'$. Then by (\ref{General Eq sJ}) and (\ref{General Eq tJJ'}) we have
    \begin{equation*}
        s^{J'+1}_j=
    \begin{cases}
        r_j&\text{if}~j=j_0-w+2,\ldots,j_0\\
        r_j+1&\text{if}~j=j_0-w+1
    \end{cases}
    \end{equation*}
    and $2i_j+t^{J'+1}(J'')_j=p-1-r_j$ for $j=j_0-w+1,\ldots,j_0$. By Proposition \ref{General Prop relation 1} applied to $(\un{i},J'+1,J'')$ with $j_0$ as above and using (\ref{General Eq seq small}), we have
    \begin{equation*}
        \bbbra{\scalebox{1}{$\prod$}_{j=j_0-w+1}^{j_0}Y_j^{p-1-r_j}}\smat{p&0\\0&1}x_{J'+1,\un{e}^{(J^{\sh}+1)}}=0.
    \end{equation*}
    Since $j_0-w+1,\ldots,j_0\in J\setminus J'$, this completes the proof of (\ref{General Eq Y^c-r 2}).
\end{proof}


\begin{proposition}\label{General Prop xJ' xJ}
    Let $0\leq k\leq f$. Assume that Theorem \ref{General Thm seq} is true for $|J|\leq k-1$. Let $J,J'\subseteq\cJ$ such that $|J|\leq k$ and $J^{\ss}\subseteq J'\subsetneqq J$.
    \begin{enumerate}
    \item 
    If $J^{\nss}\neq\cJ$, then we have (see the end of \S\ref{General Sec relation} for $\mu_{*,J}/\mu_{*,J'}$)
    \begin{equation*}
        \varepsilon_{J'}x_{J',\un{r}^{J\setminus J'}+\un{e}^{J^{\sh}}}=
    \begin{cases}
        0,&\text{if}~J'\nsupseteq J^{\ss}\sqcup(\partial J)^{\nss}\\
        (-1)^{|(J'\setminus\partial J)^{\nss}|}\frac{\mu_{*,J}}{\mu_{*,J'}}v_J,&\text{if}~J'\supseteq J^{\ss}\sqcup(\partial J)^{\nss}.
    \end{cases}
    \end{equation*}
    \item 
    If $J^{\nss}=\cJ$  (i.e.\,$(J,J_{\rhobar})=(\cJ,\emptyset)$, which implies $k=f$) and $J'\neq\emptyset$, then we have
    \begin{equation*}
        \varepsilon_{J'}x_{J',\un{r}^{J\setminus J'}}=(-1)^{|J'|+1}\frac{\mu_{*,J}}{\mu_{*,J'}}v_J.
    \end{equation*}
    \end{enumerate}
\end{proposition}

\begin{proof}
    (i). We let $J^{\nss}\neq\cJ$ and separate the following cases:

    \hspace{\fill}
    
    (a). Suppose that $J'\nsupseteq J^{\ss}\sqcup(\partial J)^{\nss}$. Since $J'\supseteq J^{\ss}$, we have $J'\nsupseteq\partial J$. Then we let $j_0\in\partial J$ (i.e.\,$j_0\in J$ and $j_0+1\notin J$) and $j_0\notin J'$. This implies $j_0\in J\setminus J'$ and $j_0+1\notin J\setminus J'$, hence $r^{J\setminus J'}_{j_0}=-1$ by (\ref{General Eq rJ}). We also have $j_0\notin J^{\sh}$. Then we deduce from Lemma \ref{General Lem simple 0} that $x_{J',\un{r}^{J\setminus J'}+\un{e}^{J^{\sh}}}=0$.

    \hspace{\fill}

    (b). Suppose that $J^{\ss}\sqcup(\partial J)^{\nss}\subseteq J'\subsetneqq J$. We use increasing induction on $\babs{J'\setminus\!\bigbra{J^{\ss}\sqcup(\partial J)^{\nss}}}$, which equals $\babs{(J'\setminus\partial J)^{\nss}}$. 
    
    First we assume that $J'=J^{\ss}\sqcup(\partial J)^{\nss}$. By Proposition \ref{General Prop vector} applied to $(J'+1,J)$, we have
    \begin{equation}\label{General Eq xJ'xJ 1}
    \begin{aligned}
        \mu_{J'+1,J}v_J&=\bbbra{\scalebox{1}{$\prod$}_{j\in J_0}Y_j^{s^J_j}\scalebox{1}{$\prod$}_{j\notin J_0}Y_j^{p-1}}\smat{p&0\\0&1}\bbra{\un{Y}^{-\un{e}^{((J'+1)\cap J)^{\nss}}}v_{J'+1}}\\
        &=\bbbra{\scalebox{1}{$\prod$}_{j\in J_0}Y_j^{s^J_j}\scalebox{1}{$\prod$}_{j\notin J_0}Y_j^{p-1}}\smat{p&0\\0&1}x_{J'+1,\un{e}^{J_1}},
    \end{aligned}
    \end{equation}
    where the second equality follows from (\ref{General Eq seq small}) and
    \begin{align*}
    &\begin{aligned}
        J_0&\eqdef\bigbra{(J'+1)\Delta J}-1=J'\Delta(J-1)=\bigbra{J^{\ss}\sqcup(\partial J)^{\nss}}\Delta\bigbra{(J-1)^{\ss}\sqcup(J-1)^{\nss}}\\
        &\hspace{1.5cm}=\bigbra{J\Delta(J-1)}^{\ss}\sqcup\bigbra{(\partial J)\Delta(J-1)}^{\nss}=\bigbra{J\Delta(J-1)}^{\ss}\sqcup\bigbra{J\cup(J-1)}^{\nss};
    \end{aligned}\\
    &\begin{aligned}
        J_1&\eqdef\bigbra{(J'+1)\cap J}^{\nss}\sqcup(J'+1)^{\sh}=\bigbra{(J'+1)\cap J^{\nss}}\sqcup\bigbra{(J'+1)\cap(J')^{\ss}}\\
        &\hspace{1.5cm}=\bigbra{(J'+1)\cap J^{\nss}}\sqcup\bigbra{(J'+1)\cap J^{\ss}}=(J'+1)\cap J=\bigbra{J'\cap(J-1)}+1.
    \end{aligned}
    \end{align*}
    We write $\un{s}\in\ZZ^f$ with $s_j\eqdef s^J_j$ if $j\in J_0$ and $s_j\eqdef p-1$ if $j\notin J_0$. 

    \hspace{\fill}
    
    \noindent\textbf{Claim.} We have
    \begin{equation}\label{General Eq xJ'xJ 2}
        p\un{e}^{J_1-1}+\un{c}^{J'}-\un{s}=\un{r}^{J\setminus J'}+\un{e}^{J^{\sh}}.
    \end{equation}
    
    \proof Fix $j\in\cJ$. We assume that $j\in J_{\rhobar}$, the case $j\notin J_{\rhobar}$ being similar. In particular, we have $j\in J'$ if and only if $j\in J$, which implies
    \begin{equation}\label{General Eq xJ'xJ 3}
        p\delta_{j\in J_1-1}-\delta_{j\in J^{\sh}}=(p-1)\delta_{j\in J\cap(J-1)}.
    \end{equation}
    Since $j\in J'$ if and only if $j\in J$, and $j\in J_0$ if and only if $j\in J\Delta(J-1)$, by (\ref{General Eq rJ}), (\ref{General Eq cJ}) and (\ref{General Eq sJ}) with a case-by-case examination we have
    \begin{equation}\label{General Eq xJ'xJ 4}
    \begin{aligned}
        r^{J\setminus J'}_j&=\delta_{j+1\in J\setminus J'}(r_j+1);\\
        c^{J'}_j&=\delta_{j\notin J'}(p-2-r_j)+\delta_{j+1\notin J'}(r_j+1)=\delta_{j\notin J}(p-2-r_j)+\delta_{j+1\notin J'}(r_j+1);\\
        s_j&=\delta_{j\notin J}(p-2-r_j)+\delta_{j+1\notin J}(r_j+1)+(p-1)\delta_{j\in J\cap(J-1)}.
    \end{aligned}
    \end{equation}
    Combining (\ref{General Eq xJ'xJ 3}) and (\ref{General Eq xJ'xJ 4}) we get (\ref{General Eq xJ'xJ 2}).\qed

    \hspace{\fill}
    
    By (\ref{General Eq Seq p001}) applied to $J'+1$ and using $(J')^{\ss}=J^{\ss}$, we deduce from (\ref{General Eq xJ'xJ 1}) that 
    \begin{align*}
        \mu_{J'+1,J}v_J&=\sum\limits_{J^{\ss}\subseteq J''\subseteq J'}\varepsilon_{J''}\mu_{J'+1,J''}x_{J'',\bigbra{p\un{e}^{J_1-1}+\un{c}^{J'}+\un{r}^{J'\setminus J''}-\un{s}}}\\
        &=\sum\limits_{J^{\ss}\subseteq J''\subseteq J'}\varepsilon_{J''}\mu_{J'+1,J''}x_{J'',\bigbra{\un{r}^{J\setminus J'}+\un{r}^{J'\setminus J''}+\un{e}^{J^{\sh}}}}\\
        &=\sum\limits_{J^{\ss}\subseteq J''\subseteq J'}\varepsilon_{J''}\mu_{J'+1,J''}x_{J'',\un{r}^{J\setminus J''}+\un{e}^{J^{\sh}}},
    \end{align*}
    where the second equality follows from (\ref{General Eq xJ'xJ 2}) and the last equality follows from Lemma \ref{General Lem r and c}(ii). We know from (a) that  $x_{J'',\un{r}^{J\setminus J''}+\un{e}^{J^{\sh}}}=0$ for $J^{\ss}\subseteq J''\subsetneqq J'$, hence we conclude that
    \begin{equation*}
        \varepsilon_{J'}\mu_{J'+1,J'}x_{J',\un{r}^{J\setminus J'}+\un{e}^{J^{\sh}}}=\mu_{J'+1,J}v_J,
    \end{equation*}
    which proves (i) when $J'=J^{\ss}\sqcup(\partial J)^{\nss}$.

    \hspace{\fill}

    Next we assume that $J^{\ss}\sqcup(\partial J)^{\nss}\subsetneqq J'\subsetneqq J$. By Lemma \ref{General Lem lem for zw}, (\ref{General Eq Seq p001}) applied to $J'+1$ and using $(J')^{\ss}=J^{\ss}$, we have
    \begin{align*}
        0&=\un{Y}^{\un{c}^{J'}+(p-1)\un{e}^{J^{\sh}}-\un{r}^{J\setminus J'}}\smat{p&0\\0&1}x_{J'+1,\un{e}^{(J^{\sh}+1)}}\\
        &=\sum\limits_{J^{\ss}\subseteq J''\subseteq J'}\varepsilon_{J''}\mu_{J'+1,J''}x_{J'',\bigbra{p\un{e}^{J^{\sh}}+\un{c}^{J'}+\un{r}^{J'\setminus J''}-\bigbra{\un{c}^{J'}+(p-1)\un{e}^{J^{\sh}}-\un{r}^{J\setminus J'}}}}\\
        &=\sum\limits_{J^{\ss}\subseteq J''\subseteq J'}\varepsilon_{J''}\mu_{J'+1,J''}x_{J'',\un{r}^{J\setminus J''}+\un{e}^{J^{\sh}}},
    \end{align*}
    where the last equality follows from Lemma \ref{General Lem r and c}(ii). We know from (a) that $x_{J'',\un{r}^{J\setminus J''}+\un{e}^{J^{\sh}}}=0$ for $J''\nsupseteq J^{\ss}\sqcup(\partial J)^{\nss}$, hence we have
    \begin{equation}\label{General Eq xJ'xJ 5}
        \sum\limits_{J^{\ss}\sqcup(\partial J)^{\nss}\subseteq J''\subseteq J'}\varepsilon_{J''}\mu_{J'+1,J''}x_{J'',\un{r}^{J\setminus J''}+\un{e}^{J^{\sh}}}=0.
    \end{equation}
    By the induction hypothesis, we have for $J^{\ss}\sqcup(\partial J)^{\nss}\subseteq J''\subsetneqq J'$
    \begin{equation}\label{General Eq xJ'xJ 6}
        \varepsilon_{J''}x_{J'',\un{r}^{J\setminus J''}+\un{e}^{J^{\sh}}}=    (-1)^{|(J''\setminus\partial J)^{\nss}|}\frac{\mu_{*,J}}{\mu_{*,J''}}v_J.
    \end{equation}
    Moreover, if we denote $m\eqdef|(J'\setminus\partial J)^{\nss}|$, then (by the definition of $\mu_{*,J}/\mu_{*,J''}$) we have
    \begin{equation}\label{General Eq xJ'xJ 7}
    \begin{aligned}
        \sum\limits_{J^{\ss}\sqcup(\partial J)^{\nss}\subseteq J''\subsetneqq J'}(-1)^{|(J''\setminus\partial J)^{\nss}|}\frac{\mu_{*,J}}{\mu_{*,J''}}\mu_{J'+1,J''}&=\bbbra{\sum\limits_{(J''\setminus\partial J)^{\nss}\subsetneqq(J'\setminus\partial J)^{\nss}}(-1)^{|(J''\setminus\partial J)^{\nss}|}}\mu_{J'+1,J}\\
        &=\bbbra{\sum\limits_{i=0}^{m-1}(-1)^i\binom{m}{i}}\mu_{J'+1,J}=(-1)^{m+1}\mu_{J'+1,J}.
    \end{aligned}
    \end{equation}
    Combining (\ref{General Eq xJ'xJ 5}), (\ref{General Eq xJ'xJ 6}) and (\ref{General Eq xJ'xJ 7}), we conclude that
    \begin{equation*}
        \varepsilon_{J'}x_{J',\un{r}^{J\setminus J'}+\un{e}^{J^{\sh}}}=    (-1)^{|(J'\setminus\partial J)^{\nss}|}\frac{\mu_{*,J}}{\mu_{*,J'}}v_J,
    \end{equation*}
    which proves (i).

    \hspace{\fill}
    
    (ii). Let $(J,J_{\rhobar})=(\cJ,\emptyset)$ and $\emptyset\neq J'\neq\cJ$. The proof is by increasing induction on $|J'|$ and is similar to (i), but using Proposition \ref{General Prop vector complement} instead of Proposition \ref{General Prop vector}. We leave it as an exercise.
\end{proof}

\section{The degree function}\label{General Sec app degree}

In this appendix, we study the degree function on $\pi$ and give a proof of Proposition \ref{General Prop degree}, see Proposition \ref{General Prop degree appendix}. It guarantees that the elements $x_{J,\un{i}}$ defined in \S\ref{General Sec xJi} give rise to elements in the dual module $\Hom_A(D_A(\pi),A)$, and is needed to prove Theorem \ref{General Thm rank 2^f}.

Let $Z_1\cong1+p\OK$ be the center of $I_1$. Since $\pi$ has a central character, $Z_1$ acts trivially on $\pi$. We still denote by $\fm_{I_1}$ the maximal ideal of $\FF\ddbra{I_1/Z_1}$ when there is no possible confusion. For $0\leq j\leq f-1$, we view $Y_j$ as an element of $\FF\ddbra{I_1/Z_1}$ and we define
\begin{equation*}
    Z_j\eqdef\sum\limits_{\lambda\in\Fq\x}\lambda^{-p^j}\pmat{1&0\\p[\lambda]&1}\in\FF\ddbra{I_1/Z_1}.
\end{equation*}
Since $Z_j$ commutes with each other, for $\un{i}\in\NNN^f$ we write $\un{Z}^{\un{i}}$ for $\prod_{j=0}^{f-1}Z_j^{i_j}$. For $0\leq j\leq f-1$, we denote by $y_j,z_j\in\gr\bra{\FF\ddbra{I_1/Z_1}}$ (the graded ring for the $\fm_{I_1}$-adic filtration) the associated elements of $Y_j,Z_j\in\FF\ddbra{I_1/Z_1}$. We define the $\gr\bra{\FF\ddbra{I_1/Z_1}}$-module
\begin{equation*}
    \gr(\pi)\eqdef\bigoplus\limits_{n\geq0}\pi[\fm_{I_1}^{n+1}]/\pi[\fm_{I_1}^n].
\end{equation*}
By the proof of \cite[Cor.~5.3.5]{BHHMS1} using condition (ii) on $\pi$ (see above Theorem \ref{General Thm main1}) and taking $\FF$-linear dual, the $\gr(\FF\ddbra{I_1/Z_1})$-module $\gr(\pi)$ is annihilated by the ideal $(y_jz_j,z_jy_j;\,0\leq j\leq f-1)$, hence becomes a graded module over $R\eqdef\gr(\FF\ddbra{I_1/Z_1})/(y_jz_j,z_jy_j;\,0\leq j\leq f-1)$, which is a commutative ring, isomorphic to $\FF[y_j,z_j]/(y_jz_j;\,0\leq j\leq f-1)$ with $y_j,z_j$ of degree $-1$ (see \cite[Thm.~5.3.4]{BHHMS1}). For $v\in\pi$, as in \cite[\S3.5]{BHHMS3} we define 
\begin{equation*}
    \deg(v)\eqdef\min\set{n\geq-1:v\in\pi[\fm_{I_1}^{n+1}]}\in\ZZ_{\geq-1}.
\end{equation*}
We denote $\gr(v)\in\pi[\fm_{I_1}^{\deg(v)+1}]/\pi[\fm_{I_1}^{\deg(v)}]\subseteq\gr(\pi)$ if $v\neq0$  and $\gr(v)=0$ if $v=0$ the associated graded element of $v$.

\begin{lemma}\label{General Lem degree v}
    Let $v\in\pi$ with $\deg(v)=d\geq0$.
\begin{enumerate}
    \item 
    For $j\in\cJ$, we have $\deg(Y_jv)\leq\deg(v)-1$. If moreover $d>0$, then the equality holds if and only if $y_j\gr(v)=\gr(Y_jv)\neq0$ in $\gr(\pi)$. Similar statements hold for $Z_j$.
    \item 
    There exists $\un{a},\un{b}\in\NNN^f$ satisfying $\norm{\un{a}}+\norm{\un{b}}=d$ such that $0\neq\un{Y}^{\un{a}}\un{Z}^{\un{b}}v\in\pi^{I_1}$.
    \item 
    We have $\deg\bigbra{\!\smat{p&0\\0&1}v}\leq pd+(p-1)f$.
\end{enumerate}
\end{lemma}

\begin{proof}
    (i). This follows from the fact that $y_j,z_j\in R$ has degree $-1$.

    (ii). If $d=0$, then the statement is trivial, so we let $d>0$. Since $y_0,\ldots,y_{f-1},z_0,\ldots,z_{f-1}$ form an $\FF$-basis of the degree $-1$ part of $R$, which equals $\fm_{I_1}/\fm_{I_1}^2$, there exists one of them, say $y_j$, such that $y_j\gr(v)\neq0$ (otherwise, $\gr(v)$ is annihilated by $\fm_{I_1}/\fm_{I_1}^2$ in $\gr(\pi)$, so $\fm_{I_1}v\subseteq\pi[\fm_{I_1}^{d-1}]$, i.e.\,$v\in\pi[\fm_{I_1}^{d}]$, a contradiction). By (i), we have $\deg(Y_jv)=d-1$. If $d-1>0$, continue this process to $Y_jv\in\pi$ and so on.
    
    In particular, there exist $W_1,\ldots,W_d\in\set{Y_0,\ldots,Y_{f-1},Z_0,\ldots,Z_{f-1}}$ such that $W_1\cdots W_d\,v\in\pi$ has degree $0$ and $w_1\cdots w_d\gr(v)\neq0$ in $\gr(\pi)$, where $w_i\in R$ is the associated graded element of $W_i$ for $1\leq i\leq d$. We let $W_1',\ldots,W_d'$ be a permutation of $W_1,\ldots,W_d$ such that $W_1'\cdots W_d'$ is of the form $\un{Y}^{\un{a}}\un{Z}^{\un{b}}$ as in the statement. Since $R$ is commutative, we have $w_1'\cdots w_d'\gr(v)\neq0$ in $\gr(\pi)$. As a consequence, $W_1'\cdots W_d'\,v\neq0$ and has degree zero by (i), hence belongs to $\pi^{I_1}$.

    (iii). By (ii), it suffices to show that $\un{Y}^{\un{a}}\un{Z}^{\un{b}}\smat{p&0\\0&1}v=0$ for all $\un{a},\un{b}\in\NNN^f$ such that $\norm{\un{a}}+\norm{\un{b}}\geq pd+(p-1)f+1$. We write $\un{a}=p\un{c}+\un{\ell}$ for the unique $\un{c}\geq\un{0}$ and $\un{0}\leq\un{\ell}\leq\un{p}-\un{1}$. One easily checks that $Z_j\smat{p&0\\0&1}=\smat{p&0\\0&1}Z_{j-1}^p$ for all $j\in\cJ$. Together with Lemma \ref{General Lem Yj}(i), we have
    \begin{equation*}
        \un{Y}^{\un{a}}\un{Z}^{\un{b}}\smat{p&0\\0&1}v=\un{Y}^{\un{\ell}}\un{Y}^{p\un{c}}\smat{p&0\\0&1}\un{Z}^{p\delta(\un{b})}v=\un{Y}^{\un{\ell}}\smat{p&0\\0&1}\un{Y}^{\delta\inv(\un{c})}\un{Z}^{p\delta(\un{b})}v.
    \end{equation*}
    Since $\deg(v)=d$, using (i) it suffices to show that $\norm{\delta\inv(\un{c})}+\norm{p\delta(\un{b})}>d$. Indeed, we have
    \begin{equation*}
    \begin{aligned}
        \norm{\delta\inv(\un{c})}+\norm{p\delta(\un{b})}&=\norm{\un{c}}+p\norm{\un{b}}=\bra{\norm{a}-\norm{\ell}}/p+p\norm{b}\\
        &\geq\bra{\norm{a}-(p-1)f}/p+p\norm{b}\geq\bra{\norm{a}+\norm{b}-(p-1)f}/p\geq(pd+1)/p>d,
    \end{aligned}
    \end{equation*}
    which completes the proof.
\end{proof}

Recall that we have constructed $x_{J,\un{i}}\in\pi$ for $J\subseteq\cJ$ and $\un{i}\in\ZZ^f$ in Theorem \ref{General Thm seq}.

\begin{lemma}\label{General Lem degree xJi}
    Let $J\subseteq\cJ$ and $\un{i}\in\ZZ^f$ such that $\un{i}\geq\un{e}^{J^{\sh}}$ (see \S\ref{General Sec sw} for $\un{e}^{J^{\sh}}$ and note that $\norm{\un{e}^{J^{\sh}}}=|J^{\sh}|$).
\begin{enumerate}
    \item 
    If $z_j\gr(x_{J,\un{i}})=0$ for all $j\in\cJ$, then we have $\deg(x_{J,\un{i}})=\norm{\un{i}}-|J^{\sh}|$.
    \item 
    If $\deg(x_{J,\un{i}})>\norm{\un{i}}-\abs{J^{\sh}}$, then there exists $j_0\in\cJ$ such that $\deg(x_{J,\un{i}+e_{j_0}})\geq\deg(x_{J,\un{i}})+2$.
\end{enumerate}
\end{lemma}

\begin{proof}
    (i). By the second paragraph of the proof of \cite[Prop.~3.5.1]{BHHMS3}, there exists $\un{a}\in\NNN^f$ such that $0\neq\un{Y}^{\un{a}}x_{J,\un{i}}\in\pi^{I_1}$ and $\deg(x_{J,\un{i}})=\norm{\un{a}}$. By Theorem \ref{General Thm seq}(ii) and (\ref{General Eq Seq def}), we have $\un{Y}^{\un{i}-\un{e}^{J^{\sh}}}x_{J,\un{i}}=x_{J,\un{e}^{J^{\sh}}}=v_J\neq0$, hence $\norm{\un{a}}=\deg(x_{J,\un{i}})\geq\norm{\un{i}}-\abs{J^{\sh}}$ by Lemma \ref{General Lem degree v}(i). Then by Corollary \ref{General Cor degree}(i),(ii), we must have $\un{a}=\un{i}-\un{e}^{J^{\sh}}$, hence $\deg(x_{J,\un{i}})=\norm{\un{i}}-\abs{J^{\sh}}$.

    (ii). By (i), there exists $j_0\in\cJ$ such that $z_{j_0}\gr(x_{J,\un{i}})\neq0$. Since $Y_{j_0}x_{J,\un{i}+e_{j_0}}=x_{J,\un{i}}$ by Theorem \ref{General Thm seq}(ii), we have $\deg(x_{J,\un{i}+e_{j_0}})\geq\deg(x_{J,\un{i}})+1$ by Lemma \ref{General Lem degree v}(i). Assume on the contrary that $\deg(x_{J,\un{i}+e_{j_0}})=\deg(x_{J,\un{i}})+1$, then by Lemma \ref{General Lem degree v}(i) we have $y_{j_0}\gr(x_{J,\un{i}+e_{j_0}})=\gr(x_{J,\un{i}})$, hence $z_{j_0}y_{j_0}\gr(x_{J,\un{i}+e_{j_0}})=z_{j_0}\gr(x_{J,\un{i}})\neq0$. This is a contradiction since $z_{j_0}y_{j_0}=0$ in $R$.
\end{proof}

\begin{lemma}\label{General Lem degree f+1}
    For $J\subseteq\cJ$ and $\un{i}\in\ZZ^f$ such that $\un{i}\leq\un{f}+\un{1}$, we have $x_{J,\un{i}}\in\pi^{K_1}$.
\end{lemma}

\begin{proof}
    By Theorem \ref{General Thm seq}(ii), it suffices to show that $x_{J,\un{f}+\un{1}}\in\pi^{K_1}$. 
    
    Recall that $\un{c}^{\prime J}\in\ZZ^f$ is defined in (\ref{General Eq c'J}), which satisfies $\un{1}\leq\un{c}^{\prime J}\leq\un{p}-\un{1}$ by (\ref{General Eq genericity}). By the proof of Lemma \ref{General Lem lie in D0} except that we apply Lemma \ref{General Lem r and c}(iv) with $\un{\delta}=\un{1}$ instead of $\un{\delta}=\un{0}$, we have
    \begin{enumerate}
        \item 
        $\un{Y}^{\un{c}^{\prime J}-\un{1}}\smat{p&0\\0&1}x_{J+1,\un{e}^{J\cap(J+1)}}\in\pi^{K_1}$;
        \item 
        $x_{J',\un{f}+\un{1}+\un{r}^{J\setminus J'}}\in\pi^{K_1}$ for each $J'\subseteq\cJ$ such that $J^{\ss}\subseteq J'\subsetneqq J$ (see \eqref{General Eq rJ} for $\un{r}^{J\setminus J'}$).
    \end{enumerate}
    Moreover, by (\ref{General Eq cJ c'J}) we have (see (\ref{General Eq cJ}) for $\un{c}^J$)
    \begin{equation*}
        \un{f}+\un{1}=p\delta\bigbra{\un{e}^{J\cap(J+1)}}+\un{c}^J-\bra{\un{c}^{\prime J}-\un{1}}.
    \end{equation*}
    Hence we deduce from (\ref{General Eq Seq def}) (with $\un{i}=\un{f}+\un{1}$) that $x_{J,\un{f}+\un{1}}\in\pi^{K_1}$.    
\end{proof}

The following proposition is a generalization of \cite[Prop.~3.5.1]{BHHMS3} (where $\rhobar$ was assumed to be semisimple).

\begin{proposition}\label{General Prop degree appendix}
    For $J\subseteq\cJ$ and $\un{i}\in\ZZ^f$, we have $\deg\bigbra{x_{J,\un{i}}}\leq\norm{\un{i}}-\abs{J^{\sh}}$. If moreover $\un{i}\geq\un{e}^{J^{\sh}}$, then we have $\deg\bigbra{x_{J,\un{i}}}=\norm{\un{i}}-\abs{J^{\sh}}$.
\end{proposition}

\begin{proof}
    First we make the following observation. Let $J\subseteq\cJ$ and $m\geq1$. Assume that $\deg(x_{J,\un{i}})=\norm{\un{i}}-|J^{\sh}|$ for all $\un{e}^{J^{\sh}}\leq\un{i}\leq\un{m}$, then we have $\deg(x_{J,\un{i}})\leq\norm{\un{i}}-|J^{\sh}|$ for all $\un{i}\leq\un{m}$. Indeed, by Theorem \ref{General Thm seq}(ii) we have $x_{J,\un{i}}=\un{Y}^{\un{m}-\un{i}}x_{J,\un{m}}$, hence $\deg(x_{J,\un{i}})\leq(\norm{\un{m}}-\abs{J^{\sh}})-\norm{\un{m}-\un{i}}=\norm{\un{i}}-\abs{J^{\sh}}$ by Lemma \ref{General Lem degree v}(i). In particular, we only need to prove the result for $\un{i}\geq\un{e}^{J^{\sh}}$.

    We prove the result by increasing induction on $|J|$ and on $\max_ji_j$. For $J\subseteq\cJ$ and $\un{i}\in\ZZ^f$ such that $\un{e}^{J^{\sh}}\leq\un{i}\leq\un{f}+\un{1}$, by Lemma \ref{General Lem degree f+1} we have $Z_jx_{J,\un{i}}=0$ for all $j\in\cJ$, hence $z_j\gr(x_{J,\un{i}})=0$ for all $j\in\cJ$. By Lemma \ref{General Lem degree xJi}(i) we deduce that $\deg(x_{J,\un{i}})=\norm{\un{i}}-|J^{\sh}|$.
    
    Then we let $0\leq k\leq f-1$ and $m\geq f+1$. Assume that the result is true for 
    \begin{enumerate}
        \item[(a)]
        $|J|\leq k-1$ and $\un{i}\in\ZZ^f$;
        \item[(b)]
        $|J|=k$ and $\max_ji_j\leq m$,
    \end{enumerate}
    we prove the result for $|J|=k$ and $\un{i}\geq\un{e}^{J^{\sh}}$ such that $\max_ji_j=m+1$.

    \hspace{\fill}

    \noindent\textbf{Claim.} For $J\subseteq\cJ$ such that $|J|=k$ and $\un{i}\in\ZZ^f$ such that $\max_ji_j\leq pm$, we have
    \begin{equation*}
        \deg(x_{J,\un{i}})\leq\norm{\un{i}}+(p-1)f.
    \end{equation*}

    \proof We write $\un{i}=p\delta(\un{i}')+\un{c}^{J}-\un{\ell}$ for the unique $\un{i}',\un{\ell}\in\ZZ^f$ such that $\un{0}\leq\un{\ell}\leq\un{p}-\un{1}$ (see \eqref{General Eq cJ} for $\un{c}^J$). Then we claim that $\max_j{i'_j}\leq m$. Indeed, for each $j$ we have 
    \begin{equation*}
        i'_{j+1}=\bbra{i_j-c^{J}_j+\ell_j}/p\leq\bigbra{pm-0+(p-1)}/p<m+1,
    \end{equation*}
    hence $i'_{j+1}\leq m$. Since $|J+1|=|J|=k$, by (b) we have $\deg(x_{J+1,\un{i}'})\leq\norm{\un{i}'}-|(J+1)^{\sh}|\leq\norm{\un{i}'}$. Then by Lemma \ref{General Lem degree v}(i),(iii) we have
    \begin{equation}\label{General Eq Prop degree 1}
        \deg\bbra{\un{Y}^{\un{\ell}}\smat{p&0\\0&1}x_{J+1,\un{i}'}}\leq p\norm{\un{i}'}+(p-1)f-\norm{\un{\ell}}=\norm{\un{i}}-\norm{\un{c}^J}+(p-1)f\leq\norm{\un{i}}+(p-1)f,
    \end{equation}
     where the last inequality uses $\un{c}^{J}\geq\un{0}$. For $J'\subseteq\cJ$ such that $J^{\ss}\subseteq J'\subsetneqq J$, by (a) we have (see (\ref{General Eq rJ}) for $\un{r}^{J\setminus J'}$)
    \begin{equation}\label{General Eq Prop degree 2}
        \deg\bigbra{x_{J',\un{i}+\un{r}^{J\setminus J'}}}\leq\norm{\un{i}}+\norm{\un{r}^{J\setminus J'}}-|(J')^{\sh}|\leq\norm{\un{i}}+(p-1)f,
    \end{equation}
    where the last inequality uses $r^{J\setminus J'}_j\leq p-1$ for all $j\in\cJ$ by (\ref{General Eq genericity}).
    Combining \eqref{General Eq Seq def}, \eqref{General Eq Prop degree 1} and \eqref{General Eq Prop degree 2}, we deduce that $\deg(x_{J,\un{i}})\leq\norm{\un{i}}+(p-1)f$.\qed

    \hspace{\fill}

    Assume on the contrary that $\deg(x_{J_0,\un{i}^{(1)}})\geq\norm{\un{i}^{(1)}}-|J_0^{\sh}|+1$ for some $|J_0|=k$ and $\un{i}^{(1)}\geq\un{e}^{J^{\sh}}$ such that $\max_ji_j^{(1)}=m+1$. By Lemma \ref{General Lem degree xJi}(ii), there exists $\un{i}^{(2)}=\un{i}^{(1)}+e_{j_1}$ for some $j_1\in\cJ$ such that $\deg\bigbra{x_{J_0,\un{i}^{(2)}}}\geq\norm{\un{i}^{(1)}}-|J_0^{\sh}|+3=\norm{\un{i}^{(2)}}-|J_0^{\sh}|+2$. Moreover, we have $\max_ji_j^{(2)}\leq m+2$. Similarly, there exists $\un{i}^{(3)}=\un{i}^{(2)}+e_{j_2}$ for some $j_2\in\cJ$ such that $\deg\bigbra{x_{J_0,\un{i}^{(3)}}}\geq\norm{\un{i}^{(3)}}-|J_0^{\sh}|+3$, which moreover satisfies $\max_ji_j^{(3)}\leq m+3$. Continue this process, there exists $\un{i}^{((p-1)m)}\in\ZZ^f$ such that $\max_ji_j^{((p-1)m)}\leq pm$ and
    \begin{equation*}
        \deg\bigbra{x_{J_0,\un{i}^{((p-1)m)}}}\geq\norm{\un{i}^{((p-1)m)}}-|J_0^{\sh}|+(p-1)m.
    \end{equation*}
    By the Claim above, we also have
    \begin{equation*}
        \deg\bigbra{x_{J_0,\un{i}^{((p-1)m)}}}\leq\norm{\un{i}^{((p-1)m)}}+(p-1)f.
    \end{equation*}
    This is a contradiction since $m\geq f+1$ and $|J_0^{\sh}|\leq f\leq p-2$ by (\ref{General Eq genericity}).
\end{proof}

\section{The actions of \texorpdfstring{$\varphi$}. and \texorpdfstring{$\OK\x$}. on \texorpdfstring{$\Hom_A(D_A(\pi),A)$}.}\label{General Sec app okx}

In this appendix, we determine the actions of $\varphi$ and $\OK\x$ on $\Hom_A(D_A(\pi),A)(1)$. The main results are Proposition \ref{General Prop phi-OK* action new} and Corollary \ref{General Cor a-action}. Here the $\OK\x$-action is much more technical to compute explicitly than in the semisimple case (see \cite[Prop.~3.8.3]{BHHMS3}). Instead, we give a congruence relation which uniquely determines the $\OK\x$-action. The results of this appendix will be used in \cite{Wang3}.

\begin{lemma}\label{General Lem a varphiq a}
    Let $a\in A$, $\lambda\in\FF\x$ and $\un{s}\in\ZZ^f$ such that $a=\lambda\un{Y}^{\un{s}}\varphi_q(a)$. If $\un{s}=(q-1)\un{t}$ for some $\un{t}\in\ZZ^f$ and $\lambda=1$, then we have $a\in\FF\un{Y}^{-t}$. Otherwise, we have $a=0$.
\end{lemma}

\begin{proof}
    Let $m>0$ be large enough such that $q^m$ is a multiple of $\abs{\FF}$ and $\lambda^m=1$. In particular, $\varphi_q^m$ acts as $x\mapsto x^{q^m}$ on $A$. By iteration, we have $a^{q^m}=\un{Y}^{-((q^m-1)/(q-1))\un{s}}a$. Suppose that $a\neq0$. Since $A$ is an integral domain, we have $a^{q^m-1}=\un{Y}^{-((q^m-1)/(q-1))\un{s}}$. In particular, we have $a\in A\x$, hence we can write $a=c\un{Y}^{-\un{t}}a_1$ with $c\in\FF\x$, $\un{t}\in\ZZ^f$ and $a_1\in1+F_{-1}A$. Then we deduce that $a_1=1$ and $(q^m-1)\un{t}=((q^m-1)/(q-1))\un{s}$, which implies $\un{s}=(q-1)\un{t}$, and we necessarily have $\lambda=1$.
\end{proof}

For $a\in A\x$ and $k=\sum_{i=0}^mk_i\varphi^i\in\ZZ[\varphi]$ with $m\in\NNN$ and $k_i\in\ZZ$ for all $0\leq i\leq m$, we define $a^k\eqdef\prod\nolimits_{i=0}^m\varphi^i(a^{k_i})\in A\x$. This makes $A\x$ a $\ZZ[\varphi]$-module. By completeness, $1+F_{-1}A$ is a $\ZZ_{(p)}[\varphi]$-module, where $\ZZ_{(p)}$ is the localization of $\ZZ$ with respect to the prime ideal $(p)$.

\begin{lemma}\label{General Lem theta map}
    Let $J,J'\subseteq\cJ$, $\lambda_j\in\FF\x$ and $1\leq h_j\leq p-2$ for all $j\in\cJ$. Consider the map
    \begin{equation*}
    \begin{aligned}
        \theta:\bigbra{A^{[\Fq\x]}}^f&\to\bigbra{A^{[\Fq\x]}}^f\\
        (a_i)_{i\in\cJ}&\mapsto\bbra{a_i-\lambda_i\bbbra{\scalebox{1}{$\prod\limits_{j-i\in J\setminus J'}$}Y_j^{h_j(1-\varphi)}\scalebox{1}{$\prod\limits_{j-i\in J'\setminus J}$}Y_j^{-h_j(1-\varphi)}}\varphi(a_{i+1})}_{i\in\cJ}.
    \end{aligned}    
    \end{equation*}
    \begin{enumerate}
    \item 
    If $J'\neq J$, then $\theta(\un{a})=\un{0}$ implies $\un{a}=\un{0}$.
    \item
    If $J'=J$ and $\lambda_i=1$ for all $i$, then $\theta(\un{a})=\un{0}$ implies $\un{a}=\un{\mu}$ for some $\mu\in\FF$.
    \item 
    If $J'=J\setminus\set{j_0}$ for some $j_0\in J$ and $\un{0}\neq\un{b}\in\bigbra{(F_{0}A\setminus F_{-1}A)\cap A^{[\Fq\x]}}^f$, then the equation $\theta(\un{a})=\un{b}$ has no solution.
    \item 
    If $J'\subsetneqq J$ and $b_i\in F_{|J\setminus J'|(1-p)}A\cap A^{[\Fq\x]}$ for all $i$, then the equation $\theta(\un{a})=\un{b}$ has at most one solution. If moreover $1\leq h_j\leq p-1-f$ for all $j$, then there is a unique solution which moreover satisfies $a_i\equiv b_i~\mod~F_{(f+1)(1-p)}A$ for all $i$.
    \end{enumerate}
\end{lemma}

\begin{proof}
    (i). We write $\lambda\eqdef\prod\nolimits_{i=0}^{f-1}\lambda_i$ and $h^{(j)}\eqdef\sum\limits_{i=0}^{f-1}h_{j+i}p^i$ for all $j\in\cJ$. If $\theta(\un{a})=0$, then by iteration we have $a_0=\lambda\un{Y}^{\un{s}}\varphi_q(a_0)$, where 
    \begin{equation}\label{General Eq theta map s}
        \un{s}\eqdef\sum\limits_{j\in J\setminus J'}h^{(j)}(e_j-pe_{j-1})-\sum\limits_{j\in J'\setminus J}h^{(j)}(e_j-pe_{j-1})\neq\un{0}.
    \end{equation}

    \hspace{\fill}

    \noindent\textbf{Claim.} Suppose that $\un{s}=(q-1)\un{t}$ for some $\un{t}\in\ZZ^f$. Then we have $|t_j|\leq p-1$ for all $j\in\cJ$ and $\un{0}\neq\un{t}\neq\pm(\un{p}-\un{1})$. 
    
    \proof Since $0\leq h^{(j)}\leq(p-2)(1+p+\cdots+p^{f-1})$ for all $j\in\cJ$, we deduce from (\ref{General Eq theta map s}) that
    \begin{equation*}
        |t_j|\leq\frac{(p-2)(1+p+\cdots+p^{f-1})(p+1)}{q-1}=\frac{(p-2)(p+1)}{p-1}<p,
    \end{equation*}
    hence $|t_j|\leq p-1$ for all $j\in\cJ$. We also deduce from (\ref{General Eq theta map s}) that 
    \begin{equation*}
        \bigabs{\norm{\un{t}}}\leq\frac{(p-2)(1+p+\cdots+p^{f-1})(p-1)f}{q-1}=(p-2)f,
    \end{equation*}
    which implies $\un{t}\neq\pm(\un{p}-\un{1})$.\qed

    \hspace{\fill}
    
    By Lemma \ref{General Lem a varphiq a}, the only possible nonzero solution for $a_0$ in $A$ is a scalar multiple of $\un{Y}^{-\un{t}}$, which is not fixed by $[\Fq\x]$ since $-(\un{p}-\un{1})<\un{t}<\un{p}-\un{1}$ and $\un{t}\neq\un{0}$. Hence $a_0=0$, and we conclude that $a_i=0$ for all $i$.

    \hspace{\fill}

    (ii). If $\theta(\un{a})=\un{0}$, then by iteration we have $a_0=\varphi_q(a_0)$. By Lemma \ref{General Lem a varphiq a}, we deduce that $a_0=\mu\in\FF$, hence $a_i=\mu$ for all $i$.

    \hspace{\fill}

    (iii). In this case, the equation $\theta(\un{a})=\un{b}$ becomes 
    \begin{equation}\label{General Eq ai ai+1}
        a_i=\lambda_iY_{j_0+i}^{h_{j_0+i}(1-\varphi)}\varphi(a_{i+1})+b_i~\forall\,i\in\cJ.
    \end{equation}
    For $0\neq a\in A$, we say that $a$ has degree $m$ if $a\in F_{-m}A\setminus F_{-(m+1)}A$. We also define $\deg(0)\eqdef\infty$. In particular, a nonzero scalar has degree zero, and $\varphi$ multiplies the degree by $p$ (see (\ref{General Eq action on Yj})). We choose $i_0\in\cJ$ such that $b_{i_0}\neq0$ (hence $\deg(b_{i_0})=0$) and let $i=i_0$ in (\ref{General Eq ai ai+1}). Since the degree of $Y_{j_0+i_0}^{h_{j_0+i_0}(1-\varphi)}$ is not a multiple of $p$, the two terms of the RHS of (\ref{General Eq ai ai+1}) have different degrees. Comparing the degrees of both sides of (\ref{General Eq ai ai+1}), we deduce that $\deg\bbra{a_{i_0}}\leq0$. 

    Then we let $i=i_0-1$ in (\ref{General Eq ai ai+1}). Since $\deg\bbra{a_{i_0}}\leq0$, we have
    \begin{equation*}
        \deg\bbra{Y_{j_0+i_0-1}^{h_{j_0+i_0-1}(1-\varphi)}\varphi(a_{i_0})}=p\deg(a_{i_0})-(p-1)h_{j_0+i_0-1}<\min\set{\deg(a_{i_0}),0}.
    \end{equation*}
    Comparing the degrees of both sides of (\ref{General Eq ai ai+1}), we deduce that $\deg\bbra{a_{i_0-1}}<\deg\bbra{a_{i_0}}$.
    
    Then we let $i=i_0-2$ in (\ref{General Eq ai ai+1}) and continue this process. We finally deduce that 
    \begin{equation*}
        \deg(a_{i_0})=\deg(a_{i_0-f})<\deg(a_{i_0-f+1})<\cdots<\deg(a_{i_0-1})<\deg(a_{i_0}),
    \end{equation*}
    which is a contradiction.

    \hspace{\fill}
    
    (iv). By (i), the equation $\theta(\un{a})=\un{b}$ has at most one solution. If moreover $1\leq h_j\leq p-1-f$ for all $j\in\cJ$, since $b_i\in F_{|J\setminus J'|(1-p)}A$ for all $i\in\cJ$, we have for all $i\in\cJ$
    \begin{align*}
        \deg\Big(\!\bigbra{(\id-\theta)(\un{b})}_i\!\Big)&=\deg\bbra{\lambda_i\bbbra{\scalebox{1}{$\prod$}_{j-i\in J\setminus J'}Y_j^{h_j(1-\varphi)}}\varphi(b_{i+1})}\\
        &\geq p\deg(b_{i+1})-|J\setminus J'|(p-1)(p-1-f)\\
        &\geq\max\sset{(f+1)(p-1),\deg(b_{i+1})+1}.
    \end{align*}
    Hence the series $\un{a}\eqdef\un{b}+\sum\nolimits_{k=1}^{\infty}(\id-\theta)^k(\un{b})$ converges in $\bigbra{A^{[\Fq\x]}}^f$, gives a solution of the equation and satisfies the required congruence relation.
\end{proof}

Recall from \S\ref{General Sec basis} that $\Hom_A(D_A(\pi),A)(1)$ is an \'etale $(\varphi,\OK\x)$-module of rank $2^f$. We denote by $\Mat(\varphi)$ and $\Mat(a)$ ($a\in\OK\x$) the matrices of the actions of $\varphi$ and $\OK\x$ on $\Hom_A(D_A(\pi),A)(1)$ with respect to the basis $\sset{x_J:J\subseteq\cJ}$ of Theorem \ref{General Thm rank 2^f}, whose rows and columns are indexed by the subsets of $\cJ$.

Let $Q\in\GL_{2^f}(A)$ be the diagonal matrix with $Q_{J,J}=\un{Y}^{\un{r}^{J^c}}$ for $J\subseteq\cJ$ (see \eqref{General Eq rJ} for $\un{r}^{J^c}$). Then the matrices $\Mat(\varphi)'$, $\Mat(a)'$ ($a\in\OK\x$) with respect to the new basis $\set{x''_J\eqdef\un{Y}^{-\un{r}^{J^c}}(x_J):J\subseteq\cJ}$ of $\Hom_A(D_A(\pi),A)(1)$ are given by $\Mat(\varphi)'=Q\Mat(\varphi)\varphi(Q)\inv$ and $\Mat(a)'=Q\Mat(a)a(Q)\inv$. 

\begin{proposition}\label{General Prop phi-OK* action new}
    \begin{enumerate}
    \item We have (see \eqref{General Eq gamma} for $\gamma_{J+1,J'}$)
    \begin{equation*}
        \Mat(\varphi)'_{J',J+1}=
    \begin{cases}
        \gamma_{J+1,J'}\prod\limits_{j\notin J}Y_j^{(r_j+1)(1-\varphi)}&\text{if}~J^{\ss}\subseteq J'\subseteq J\\
        0&\text{otherwise}.
    \end{cases}
    \end{equation*}
    \item 
    For $a\in[\Fq\x]$, we have $\Mat(a)'=I$. 
    \item
    Assume that $J_{\rhobar}\neq\cJ$. Up to twist by a continuous character $\OK\x\to\FF\x$, there exists a unique $\OK\x$-action on $\Hom_A(D_A(\pi),A)(1)$ which satisfies (ii) and commutes with $\varphi$ as in (i). Moreover, the matrix $\Mat(a)'$ ($a\in\OK\x$) satisfies for $J,J'\subseteq\cJ$
    \begin{enumerate}
    \item
    $\Mat(a)'_{J',J}=0$ if $J'\nsubseteq J$;
    \item 
    $\Mat(a)'_{J',J}\in F_{|J\setminus J'|(1-p)}A$ if $J'\subseteq J$.
    \end{enumerate}
    \item 
    Assume that $J_{\rhobar}=\cJ$. Up to diagonal matrices $B\in\GL_{2^f}(\FF)$ such that $B_{J,J}=B_{J+1,J+1}$ for all $J\subseteq\cJ$, there exists a unique $\OK\x$-action on $\Hom_A(D_A(\pi),A)(1)$ which satisfies (ii) and commutes with $\varphi$ as in (i). Moreover, the matrix $\Mat(a)'$ is diagonal for all $a\in\OK\x$.
    \end{enumerate} 
\end{proposition}

\begin{proof}
    (i). We have $\Mat(\varphi)'_{J',J+1}=Q_{J',J'}\Mat(\varphi)_{J',J+1}\varphi(Q_{J+1,J+1})\inv$. Hence we deduce from Proposition \ref{General Prop phi-OK* action}(i) that $\Mat(\varphi)'_{J',J+1}\neq0$ if and only if $J^{\ss}\subseteq J'\subseteq J$, in which case we have (see \eqref{General Eq cJ} for $\un{c}^J$)
    \begin{equation*}
        \Mat(\varphi)'_{J',J+1}=\gamma_{J+1,J'}\un{Y}^{\un{r}^{(J')^c}}\un{Y}^{-(\un{c}^{J}+\un{r}^{J\setminus J'})}\varphi\bigbra{\un{Y}^{\un{r}^{(J+1)^c}}}\inv=\gamma_{J+1,J'}\un{Y}^{\un{r}^{J^c}-\un{c}^J-p\delta(\un{r}^{(J+1)^c})},
    \end{equation*}
    where the last equality follows from Lemma \ref{General Lem r and c}(ii) and \eqref{General Eq action on Yj}. By \eqref{General Eq Lem rJ} and \eqref{General Eq Lem cJ} we have
    \begin{equation*}
    \begin{aligned}
        r^{J^c}_j-c^J_j-pr^{(J+1)^c}_{j+1}
        &=\bigbra{\delta_{j+1\notin J}(r_j+1)-\delta_{j\notin J}}-\bigbra{\delta_{j\notin J}(p-2-r_j)+\delta_{j+1\notin J}(r_j+1)}\\
        &\hspace{1.5cm}-p\bigbra{\delta_{j+1\notin J}(r_{j+1}+1)-\delta_{j\notin J}}\\
        &=\delta_{j\notin J}(r_j+1)-\delta_{j+1\notin J}\bigbra{p(r_{j+1}+1)},
    \end{aligned}
    \end{equation*}
    which proves the required formula for $\Mat(\varphi)'_{J',J+1}$ using \eqref{General Eq action on Yj}.

    \hspace{\fill}

    (ii). Let $a\in[\Fq\x]$. We deduce from Proposition \ref{General Prop phi-OK* action}(ii) and (\ref{General Eq action on Yj}) that $\Mat(a)'$ is a diagonal matrix with
    \begin{equation*}
        \Mat(a)'_{J,J}=Q_{J,J}\Mat(a)_{J,J}a(Q_{J,J})\inv=\un{Y}^{\un{r}^{J^c}}\ovl{a}^{\un{r}^{J^c}}a\bigbra{\un{Y}^{\un{r}^{J^c}}}\inv=1.
    \end{equation*}

    \hspace{\fill}
    
    (iii) and (iv). For simplicity, we denote by $P_{\varphi}$ the matrix $\Mat(\varphi)'$ and we let $P_a\in\GL_{2^f}(A)$ ($a\in\OK\x$) be the matrices for the $\OK\x$-action. Since $[\Fq\x]$ fixes the basis $\set{x''_J:J\subseteq\cJ}$ by (ii), it also fixes the matrices $P_a$. By the commutativity of the actions of $\varphi$ and $\OK\x$, we have
    \begin{equation}\label{General Eq phi ok commute}
        P_a\,a(P_{\varphi})=P_{\varphi}\,\varphi(P_a).
    \end{equation}
    Since $(P_{\varphi})_{J',J+1}\neq0$ if and only if $J^{\ss}\subseteq J'\subseteq J$ by (i), comparing the $(J',J+1)$-entries of (\ref{General Eq phi ok commute}) we get
    \begin{equation}\label{General Eq (J',J+1)-entry P}
        \sum\limits_{J^{\ss}\subseteq J''\subseteq J}(P_a)_{J',J''}\,a(P_{\varphi})_{J'',J+1}=\sum\limits_{J'':(J'')^{\ss}\subseteq J'\subseteq J''}(P_{\varphi})_{J',J''+1}\varphi(P_a)_{J''+1,J+1}.
    \end{equation}

    \hspace{\fill}

    \noindent\textbf{Claim 1.} For $j\in\cJ$ we let $P_{a,j}\eqdef f_{a,j}^{h^{(j)}(1-\varphi)/(1-q)}\in1+F_{1-p}A$, where $f_{a,j}\eqdef\ovl{a}^{p^j}Y_j/a(Y_j)\in1+F_{1-p}A$ and $h^{(j)}=\sum\nolimits_{i=0}^{f-1}h_{j+i}p^i$ as in the proof of Lemma \ref{General Lem theta map} with $h_j\eqdef r_j+1$. For $J\subseteq\cJ$ we let $P_{a,J}\eqdef\prod\nolimits_{j\notin J}P_{a,j}\in1+F_{1-p}A$. In particular, $P_{a,J}$ is fixed by $[\Fq\x]$. Then for all $J\subseteq\cJ$, we have
    \begin{equation*}
        P_{a,J}\,a(P_{\varphi})_{J,J+1}=(P_{\varphi})_{J,J+1}\varphi(P_{a,J+1}).
    \end{equation*}
    
    \proof By (i) and by definition, it suffices to show that for all $j\in\cJ$ we have
    \begin{equation*}
        P_{a,j}~a\bbra{Y_j^{(r_j+1)(1-\varphi)}}=Y_j^{(r_j+1)(1-\varphi)}\varphi(P_{a,j+1}).
    \end{equation*}
    Since $\varphi(Y_{j+1})=Y_j^p$ by (\ref{General Eq action on Yj}), it suffices to show that for all $j\in\cJ$ we have
    \begin{equation*}
        f_{a,j}^{h^{(j)}/(1-q)}\,a\bigbra{Y_j^{r_j+1}}=Y_j^{r_j+1}f_{a,j}^{ph^{(j+1)}/(1-q)},
    \end{equation*}
    which follows from the equality $ph^{(j+1)}-h^{(j)}=(q-1)(r_j+1)$.\qed

    \hspace{\fill}
    
    We define $Q_a\in\GL_{2^f}(A)$ by $(Q_a)_{J',J}=(P_a)_{J',J}P_{a,J}\inv$, which is fixed by $[\Fq\x]$. Then it suffices to prove the uniqueness for $Q_a$. Dividing the LHS of (\ref{General Eq (J',J+1)-entry P}) by  $P_{a,J}\,a(P_{\varphi})_{J,J+1}\in A\x$ and the RHS of (\ref{General Eq (J',J+1)-entry P}) by $(P_{\varphi})_{J,J+1}\varphi(P_{a,J+1})\in A\x$ using Claim 1 and (i), we get
    \begin{multline}\label{General Eq (J',J+1)-entry Q}
        \sum\limits_{J^{\ss}\subseteq J''\subseteq J}\bbbra{\frac{\gamma_{*,J''}}{\gamma_{*,J}}(Q_a)_{J',J''}\scalebox{1}{$\prod$}_{j\in J\setminus J''}P_{a,j}}\\
        =\sum\limits_{J'':(J'')^{\ss}\subseteq J'\subseteq J''}\bbbra{\frac{\gamma_{J''+1,J'}}{\gamma_{J+1,J}}\bbra{\scalebox{1}{$\prod$}_{j\in J\setminus J''}Y_j^{h_j(1-\varphi)}\scalebox{1}{$\prod$}_{j\in J''\setminus J}Y_j^{-h_j(1-\varphi)}}\varphi(Q_a)_{J''+1,J+1}}.
    \end{multline}

    \hspace{\fill}

    (a). We assume that $J'\nsubseteq J$. We use increasing induction on $|J|-|J'|$ (which ranges from $-f$ to $f$) to show that $(Q_a)_{J',J}=0$. By the induction hypothesis, we have $(Q_a)_{J',J''}=0$ if $J''\subsetneqq J$, and $(Q_a)_{J''+1,J+1}=0$ if $J''\supsetneqq J'$. Hence it follows from (\ref{General Eq (J',J+1)-entry Q}) that 
    \begin{equation}\label{General Eq Qa varphiq Qa}
        (Q_a)_{J',J}=\frac{\gamma_{J'+1,J'}}{\gamma_{J+1,J}}\bbbra{\scalebox{1}{$\prod$}_{j\in J\setminus J'}Y_j^{h_j(1-\varphi)}\scalebox{1}{$\prod$}_{j\in J'\setminus J}Y_j^{-h_j(1-\varphi)}}\varphi(Q_a)_{J'+1,J+1}.
    \end{equation}
    A similar equality holds replacing $(J',J)$ with $(J'+i,J+i)$ (for all $i\in\cJ$), hence it follows from Lemma \ref{General Lem theta map}(i) (with $\lambda_i=\gamma_{J'+i+1,J'+i}/\gamma_{J+i+1,J+i}$) that $(Q_a)_{J',J}=0$.

    In the case $J_{\rhobar}=\cJ$, which implies $J^{\ss}=J$ for all $J\subseteq\cJ$, the equation (\ref{General Eq (J',J+1)-entry Q}) is the same as (\ref{General Eq Qa varphiq Qa}). Then as in the previous paragraph, we deduce from Lemma \ref{General Lem theta map}(i) that $(Q_a)_{J',J}=0$ for all $J'\neq J$.
    
    \hspace{\fill}
    
    (b). We assume that $J'=J$. Then by a similar argument, the equation (\ref{General Eq Qa varphiq Qa}) still holds and becomes $(Q_a)_{J,J}=\varphi(Q_a)_{J+1,J+1}$. By Lemma \ref{General Lem theta map}(ii), we deduce that $(Q_a)_{J,J}=\xi_{a,J}$ for some $\xi_{a,J}\in\FF\x$ (nonzero since $Q_a$ is invertible), and we have $\xi_{a,J}=\xi_{a,J+1}$. In particular, this completes the proof of (iv).

    \hspace{\fill}

    \noindent\textbf{Claim 2.} If $J_{\rhobar}\neq\cJ$, then $\xi_{a,J}$ does not depend on $J$. 
    
    \proof It suffices to show that $\xi_{a,J}=\xi_{a,J'}$ for all $J,J'$ such that $J'=J\setminus\set{j_0}$ for some $j_0\in J$. Since $(Q_a)_{J',J}=0$ for $J'\nsubseteq J$, we deduce from (\ref{General Eq (J',J+1)-entry Q}) that
    \begin{equation*}
        (Q_a)_{J',J}+\delta_{j_0\notin J_{\rhobar}}\frac{\gamma_{*,J'}}{\gamma_{*,J}}\xi_{a,J'}P_{a,j_0}=\frac{\gamma_{J'+1,J'}}{\gamma_{J+1,J}}Y_{j_0}^{h_{j_0}(1-\varphi)}\varphi(Q_a)_{J'+1,J+1}+\delta_{j_0\notin J_{\rhobar}}\frac{\gamma_{*,J'}}{\gamma_{*,J}}\xi_{a,J}.
    \end{equation*}
    A similar equality holds replacing $(J',J)$ with $(J'+i,J+i)$ (hence $j_0$ is replaced with $j_0+i$). For each $i\in\cJ$, we let 
    \begin{equation*}
        b_i\eqdef\delta_{j_0+i\notin J_{\rhobar}}\frac{\gamma_{*,J'+i}}{\gamma_{*,J+i}}\bbra{\xi_{a,J+i}-\xi_{a,J'+i}P_{a,j_0+i}}=\delta_{j_0+i\notin J_{\rhobar}}\frac{\gamma_{*,J'+i}}{\gamma_{*,J+i}}\bbra{\xi_{a,J}-\xi_{a,J'}P_{a,j_0+i}}.
    \end{equation*}
    Suppose on the contrary that $\xi_{a,J}\neq\xi_{a,J'}$. Since $P_{a,j}\in1+F_{1-p}A$ for all $j$, we deduce that $b_i\in(F_0A\setminus F_{-1}A)\cap A^{[\Fq\x]}$ for all $i$, and not all equal to $0$ since $J_{\rhobar}\neq\cJ$. Then by Lemma \ref{General Lem theta map}(iii) (with $\lambda_i=\gamma_{J'+i+1,J'+i}/\gamma_{J+i+1,J+i}$) we deduce a contradiction.\qed

    \hspace{\fill}

    (c). In the rest of the proof we assume that $J_{\rhobar}\neq\cJ$. Since $\xi_{a,J}$ does not depend on $J$ by Claim 2, we denote it by $\xi_a$. Since $(Q_a)_{J',J}=0$ for all $J'\nsubseteq J$ by (a), the assignment $a\mapsto\xi_a$ defines a continuous character of $\OK\x$ with values in $\FF\x$. By considering $\xi_a\inv P_a$, we may assume that $\xi_a=1$ for all $a\in\OK\x$. To finish the proof of (iii), we use increasing induction on $|J\setminus J'|$ to show that for $J'\subseteq J$ there is a unique choice of $(Q_a)_{J',J}$, which moreover satisfies
    \begin{equation*}
        (Q_a)_{J',J}\equiv
    \begin{cases}
        \frac{\gamma_{*,J'}}{\gamma_{*,J}}\prod\limits_{j\in J\setminus J'}\bbra{1-P_{a,j}}&\mod~F_{(f+1)(1-p)}A\quad\text{if}~J'\supseteq J^{\ss}\\
        0&\mod~F_{(f+1)(1-p)}A\quad\text{if}~J'\nsupseteq J^{\ss}.
    \end{cases}
    \end{equation*}
    Since $(Q_a)_{J,J}=\xi_a=1$ by (b) and assumption, the case $J'=J$ is true. Then we assume that $J'\subsetneqq J$. Since $(Q_a)_{J',J}=0$ for $J'\nsubseteq J$ by (a), (\ref{General Eq (J',J+1)-entry Q}) gives
    \begin{multline}\label{General Eq (J',J+1)-entry Q 2}
        \sum\limits_{J'\cup J^{\ss}\subseteq J''\subseteq J}\bbbra{\frac{\gamma_{*,J''}}{\gamma_{*,J}}(Q_a)_{J',J''}\scalebox{1}{$\prod$}_{j\in J\setminus J''}P_{a,j}}\\
        =\sum\limits_{J'':(J'')^{\ss}\subseteq J'\subseteq J''\subseteq J}\bbbra{\frac{\gamma_{J''+1,J'}}{\gamma_{J+1,J}}\bbra{\scalebox{1}{$\prod$}_{j\in J\setminus J''}Y_j^{h_j(1-\varphi)}}\varphi(Q_a)_{J''+1,J+1}},
    \end{multline}
    which implies that 
    \begin{equation}\label{General Eq (J',J+1)-entry Q 3}
        (Q_a)_{J',J}-\frac{\gamma_{J'+1,J'}}{\gamma_{J+1,J}}\bbbra{\scalebox{1}{$\prod$}_{j\in J\setminus J'}Y_j^{h_j(1-\varphi)}}\varphi(Q_a)_{J'+1,J+1}=b_0,
    \end{equation}
    where $b_0\eqdef b_{0,1}-b_{0,2}$ with 
    \begin{equation*}
    \begin{aligned}
        b_{0,1}&\eqdef\sum\limits_{J'':(J'')^{\ss}\subseteq J'\subsetneqq J''\subseteq J}\bbbra{\frac{\gamma_{J''+1,J'}}{\gamma_{J+1,J}}\bbra{\scalebox{1}{$\prod$}_{j\in J\setminus J''}Y_j^{h_j(1-\varphi)}}\varphi(Q_a)_{J''+1,J+1}};\\
        b_{0,2}&\eqdef\sum\limits_{J'\cup J^{\ss}\subseteq J''\subsetneqq J}\bbbra{\frac{\gamma_{*,J''}}{\gamma_{*,J}}(Q_a)_{J',J''}\scalebox{1}{$\prod$}_{j\in J\setminus J''}P_{a,j}}.
    \end{aligned}
    \end{equation*}
    
    By the induction hypothesis together with $1-P_{a,j}\in F_{1-p}A$ and $h_j\leq p-1-f$ (by (\ref{General Eq genericity})), each term in the summation of $b_{0,1}$ lies in $F_{(f+1)(1-p)}A$ unless the term for $J''=J$, which appears if and only if $J'\supseteq J^{\ss}$. If $J'\nsupseteq J^{\ss}$, then we have $b_{0,1}\in F_{(f+1)(1-p)}A$. Moreover, for each $J''$ such that $J'\cup J^{\ss}\subseteq J''\subsetneqq J$, we have $J'\nsupseteq J^{\ss}=(J'')^{\ss}$. Hence by the induction hypothesis, we deduce that $b_{0,2}\in F_{(f+1)(1-p)}A$, hence $b_0\in F_{(f+1)(1-p)}A$. If $J'\supseteq J^{\ss}$, then by the induction hypothesis we have
    \begin{equation*}
    \begin{aligned}
        b_0=b_{0,1}-b_{0,2}&\equiv\frac{\gamma_{*,J'}}{\gamma_{*,J}}-\sum\limits_{J'\subseteq J''\subsetneqq J}\bbbra{\frac{\gamma_{*,J''}}{\gamma_{*,J}}\frac{\gamma_{*,J'}}{\gamma_{*,J''}}\scalebox{1}{$\prod$}_{j\in J''\setminus J'}(1-P_{a,j})\scalebox{1}{$\prod$}_{j\in J\setminus J''}P_{a,j}}\\
        &\hspace{-0.5cm}=\frac{\gamma_{*,J'}}{\gamma_{*,J}}\bbbra{\scalebox{1}{$\prod$}_{j\in J\setminus J'}\bigbra{(1-P_{a,j})+P_{a,j}}-\scalebox{1}{$\sum\limits_{J'\subseteq J''\subsetneqq J}$}\bbra{\scalebox{1}{$\prod$}_{j\in J''\setminus J'}(1-P_{a,j})\scalebox{1}{$\prod$}_{j\in J\setminus J''}P_{a,j}}}\\
        &\hspace{-0.5cm}=\frac{\gamma_{*,J'}}{\gamma_{*,J}}\scalebox{1}{$\prod$}_{j\in J\setminus J'}(1-P_{a,j})\quad\bigbra{\mod~F_{(f+1)(1-p)}A}.
    \end{aligned}
    \end{equation*}
    In particular, we have $b_0\in F_{|J\setminus J'|(1-p)}A$ since $1-P_{a,j}\in F_{1-p}A$ for all $j$. For $i\in\cJ$, we define $b_i$ in a similar way as $b_0$ replacing $(J',J)$ with $(J'+i,J+i)$, and a similar equality as (\ref{General Eq (J',J+1)-entry Q 3}) holds replacing $(J',J)$ with $(J'+i,J+i)$ and $b_0$ with $b_i$. Then we deduce from Lemma \ref{General Lem theta map}(iv) (with $\lambda_i=\gamma_{J'+i+1,J'+i}/\gamma_{J+i+1,J+i}$) that there is a unique solution of $(Q_a)_{J',J}$, which satisfies 
    \begin{equation*}
        (Q_a)_{J',J}\equiv b_0~\mod F_{(f+1)(1-p)}A.
    \end{equation*}
    This completes the proof.
\end{proof}

Finally, we can determine the $\OK\x$-action on $\Hom_A(D_A(\pi),A)(1)$. In the semisimple case, this is computed explicitly in \cite[Prop.~3.8.3]{BHHMS3}.

\begin{corollary}\label{General Cor a-action}
    If $J_{\rhobar}\neq\cJ$, then the $\OK\x$-action on $\Hom_A(D_A(\pi),A)(1)$ is the unique one in Proposition \ref{General Prop phi-OK* action new}(iii) which satisfies $\Mat(a)'_{J,J}\in1+F_{1-p}A$ for all $a\in\OK\x$ and $J\subseteq\cJ$. 
\end{corollary}

\begin{proof}
    By the proof of Proposition \ref{General Prop phi-OK* action new}(iii), there exists a continuous character $\xi:\OK\x\to\FF\x$ such that for all $a\in\OK\x$ and $J\subseteq\cJ$ we have $\Mat(a)'_{J,J}=\xi(a)P_{a,J}$ with $P_{a,J}\in1+F_{1-p}A$. To prove that $\xi$ is trivial, it suffices to show that $\Mat(a)'_{\emptyset,\emptyset}\in1+F_{1-p}A$. Using the change of basis matrix $Q$ which is diagonal, it suffices to show that $\Mat(a)_{\emptyset,\emptyset}\in\un{a}^{\un{r}}(1+F_{1-p}A)$. Hence it is enough to prove that $a(x_{\emptyset})\in\un{a}^{\un{r}}(1+F_{1-p}A)x_{\emptyset}$. 

    We claim that for all $\un{i}\in\ZZ^f$, we have
    \begin{equation}\label{General Eq explicit sequence}
        x_{\emptyset,\un{i}}=\mu_{\emptyset,\emptyset}^{-n}\un{Y}^{\un{p^n}-\un{1}-\un{i}}\smat{p&0\\0&1}^nv_{\emptyset}
    \end{equation}
    for any $n\geq0$ such that $\un{p^n}-\un{1}-\un{i}\geq\un{0}$. Indeed, by Proposition \ref{General Prop vector simple} with $J=\emptyset$, we have $\un{Y}^{\un{p}-\un{1}}\smat{p&0\\0&1}v_{\emptyset}=\mu_{\emptyset,\emptyset}v_{\emptyset}$, hence using Lemma \ref{General Lem Yj}(i) the RHS of (\ref{General Eq explicit sequence}) does not depend on $n$. By (\ref{General Eq seq small}) and (\ref{General Eq zw}) with $J=\emptyset$, we deduce that (\ref{General Eq explicit sequence}) is true for $\un{i}=\un{f}$. Moreover, using Lemma \ref{General Lem Yj}(i) one easily checks that the RHS of (\ref{General Eq explicit sequence}) satisfies Theorem \ref{General Thm seq}(ii),(iii) for $J=\emptyset$. Hence by the uniqueness of $x_{\emptyset,\un{i}}$ (see Theorem \ref{General Thm seq} and its proof) we deduce that (\ref{General Eq explicit sequence}) is true for all $\un{i}\in\ZZ^f$.
    
    In particular, $x_{\emptyset,\un{i}}$ has the same expression as in the semisimple case, see \cite[(103)]{BHHMS3}. Then we conclude by the explicit computation for the semisimple case, see \cite[Prop.~3.8.3]{BHHMS3}.
\end{proof}

\begin{remark}
    If $J_{\rhobar}=\cJ$, then similar to the proof of Corollary \ref{General Cor a-action} and using the explicit computation in \cite[Prop.~3.8.3]{BHHMS3} for all $J$, one can show that the $\OK\x$-action on $\Hom_A(D_A(\pi),A)(1)$ is the unique one in Proposition \ref{General Prop phi-OK* action new}(iv) which satisfies $\Mat(a)'_{J,J}\in1+F_{1-p}A$ for all $a\in\OK\x$ and $J\subseteq\cJ$.
\end{remark}

\section{Some relations between constants}\label{General Sec app lemmas}

In this appendix, we collect some equalities among the various constants defined throughout this article, whose proofs are elementary. They are used throughout this article.

\begin{lemma}\label{General Lem p-2-s and s}
    Let $J,J'\subseteq\cJ$ satisfying $(J-1)^{\ss}=(J')^{\ss}$. Then for each $j\in(J\Delta J')-1$, we have (see (\ref{General Eq sJ}) for $\un{s}^J$)
    \begin{equation*}
        2\delta_{j\in(J\cap J')^{\nss}}+(p-2-s^J_j)+\delta_{j\in J\Delta J'}=s^{J'}_j.
    \end{equation*}
\end{lemma}

\begin{proof}
    We assume that $j+1\in J$ and $j+1\notin J'$. Otherwise we have $j+1\notin J$ and $j+1\in J'$, and the proof is similar. We separate the following cases.
    
    If $j\in J$ and $j\in J'$, then the LHS equals $2\delta_{j\notin J_{\rhobar}}+(p-2-(p-3-r_j+2\delta_{j\notin J_{\rhobar}}))+0=r_j+1$, which equals the RHS.

    If $j\in J$ and $j\notin J'$, then the LHS equals $0+(p-2-(p-3-r_j+2\delta_{j\notin J_{\rhobar}}))+1=r_j+2-2\delta_{j\notin J_{\rhobar}}$. Hence it suffices to show that $j\notin J_{\rhobar}$. Indeed, if $j\in J_{\rhobar}$, then $j\in (J-1)^{\ss}=(J')^{\ss}\subseteq J'$, which is a contradiction.

    If $j\notin J$ and $j\in J'$, then the LHS equals $0+(p-2-(p-2-r_j))+1=r_j+1$, which equals the RHS.

    If $j\notin J$ and $j\notin J'$, then the LHS equals $0+(p-2-(p-2-r_j))+0=r_j$, which equals the RHS.
\end{proof}

\begin{lemma}\label{General Lem m}
    Let $J,J'\subseteq\cJ$ satisfying $(J-1)^{\ss}=(J')^{\ss}$, and let $\un{m}=\un{m}\bigbra{\un{e}^{(J\cap J')^{\nss}},J,(J\Delta J')-1}$ (see (\ref{General Eq mj})). Then we have 
    \begin{equation*}
        m_j=\delta_{j\in J'}(-1)^{\delta_{j+1\notin J}}~\forall\,j\in\cJ.
    \end{equation*}
\end{lemma}

\begin{proof}
    For $j\in\cJ$, by definition we have
    \begin{equation*}
        m_j=(-1)^{\delta_{j+1\notin J}}\bbra{2\delta_{j\in(J\cap J')^{\nss}}+\delta_{j\in(J-1)^{\ss}}-\delta_{j\in J\Delta(J-1)^{\ss}}+\delta_{j\in J\Delta J'}}.
    \end{equation*}
    If $j\notin J_{\rhobar}$, then we have
    \begin{equation*}
    \begin{aligned}
        m_j&=(-1)^{\delta_{j+1\notin J}}\bbra{2\delta_{j\in J\cap J'}+0-\delta_{j\in J}+\delta_{j\in J\Delta J'}}\\
        &=(-1)^{\delta_{j+1\notin J}}\bbra{2\delta_{j\in J}\delta_{j\in J'}-\delta_{j\in J}+(\delta_{j\in J}+\delta_{j\in J'}-2\delta_{j\in J}\delta_{j\in J'})}=\delta_{j\in J'}(-1)^{\delta_{j+1\notin J}}.
    \end{aligned}
    \end{equation*}
    If $j\in J_{\rhobar}$, then the assumption $(J-1)^{\ss}=(J')^{\ss}$ implies that $j\in J-1$ if and only if $j\in J'$, hence we have 
    \begin{equation*}
    \begin{aligned}
        m_j&=(-1)^{\delta_{j+1\notin J}}\bigbra{0+\delta_{j\in J-1}-\delta_{j\in J\Delta(J-1)}+\delta_{j\in J\Delta J'}}\\
        &=(-1)^{\delta_{j+1\notin J}}\bigbra{\delta_{j\in J'}-\delta_{j\in J\Delta J'}+\delta_{j\in J\Delta J'}}=\delta_{j\in J'}(-1)^{\delta_{j+1\notin J}}.
    \end{aligned}
    \end{equation*}
    This completes the proof.
\end{proof}

\begin{lemma}\label{General Lemma relation 2 appendix}
    Keep the assumptions of Proposition \ref{General Prop relation 2}.
\begin{enumerate}
    \item 
    Let $\un{m}\eqdef\un{m}(\un{i},J,J')$ and $\un{m}'\eqdef\un{m}(\un{i}',J\setminus\set{j_0+2},J'')$ (see (\ref{General Eq mj})). Then we have $\un{m}=\un{m}'$ and $m_{j_0+1}=m'_{j_0+1}=0$.
    \item 
    We have (see (\ref{General Eq tJJ'}) for $t^J(J')$)
    \begin{equation}\label{General Eq relation 2 tJ appendix}
    \begin{aligned}
        2i_j+t^J(J')_j&=2i'_j+t^{J\setminus\set{j_0+2}}(J'')_j~\text{if}~j\neq j_0+1;\\
        2i_{j_0+1}+t^J(J')_{j_0+1}&=r_{j_0+1}+1;\\
        2i'_{j_0+1}+t^{J\setminus\set{j_0+2}}(J'')_{j_0+1}&=p-1-r_{j_0+1}.\\
    \end{aligned}        
    \end{equation}
    \item 
    We let $\un{c},\un{c}'\in\ZZ^f$ such that
    \begin{align*}
    &\begin{aligned}
        c_j&= pi_{j+1}+\delta_{j+1\in J\Delta(J-1)^{\ss}}s^{(J-1)^{\ss}}_j\\
        &\hspace{0.8cm}+\delta_{j+1\notin J\Delta(J-1)^{\ss}}(p-1)-\delta_{j\notin J'}\bbra{2i_j+t^J(J')_j}-\delta_{j=j_0+1}\delta_{j_0+1\notin J};
    \end{aligned}\\
    &\begin{aligned}
        c'_j&= pi'_{j+1}+\delta_{j+1\in(J\setminus\set{j_0+2})\Delta(J-1)^{\ss}}s^{(J-1)^{\ss}}_j\\
        &\hspace{0.8cm}+\delta_{j+1\notin(J\setminus\set{j_0+2})\Delta(J-1)^{\ss}}(p-1)-\delta_{j\notin J''}\bbra{2i'_j+t^{J\setminus\set{j_0+2}}(J'')_j}-\delta_{j=j_0+1}\delta_{j_0+1\notin J}.
    \end{aligned}
    \end{align*}
    Then we have $\un{c}=\un{c}'$.
    \item 
    If moreover $2i_j-\delta_{j\in J\Delta(J-1)^{\ss}}+\delta_{j-1\in J'}\geq0$ for all $j\in\cJ$, then we have $\un{c}=\un{c}'\geq\un{0}$ and $\un{Y}^{\un{c}}B_1=\un{Y}^{\un{c}}B_2\neq0$. In particular, $B_1$ and $B_2$ have the same $H$-eigencharacter. 
\end{enumerate}
\end{lemma}

\begin{proof}
    (i). If $j\neq j_0+2$ or $f=1$, then by definition we have $m_j=m'_j$ and $m_{j_0+1}=m'_{j_0+1}=0$. If $j=j_0+2$ and $f\geq2$, using $j_0+1\in(J-1)^{\nss}$ (which implies $j_0+2\in J$) and $j_0+3\neq j_0+2$, we have
    \begin{align*}
        (-1)^{\delta_{j_0+3\notin J}}m'_{j_0+2}&=(-1)^{\delta_{j_0+3\notin J\setminus\set{j_0+2}}}m'_{j_0+2}\\        &\hspace{-0.5cm}=2i'_{j_0+2}+\delta_{j_0+2\in((J\setminus\set{j_0+2})-1)^{\ss}}-\delta_{j_0+2\in(J\setminus\set{j_0+2})\Delta((J\setminus\set{j_0+2})-1)^{\ss}}+\delta_{j_0+1\in J''}\\
        &\hspace{-0.5cm}=2\bbra{i_{j_0+2}-\delta_{j_0+1\notin J'}+\delta_{j_0+2\in(J-1)^{\ss}}}+\delta_{j_0+2\in(J-1)^{\ss}}-\delta_{j_0+2\in(J-1)^{\ss}}+\delta_{j_0+1\notin J'}\\
        &\hspace{-0.5cm}=2i_{j_0+2}-\delta_{j_0+1\notin J'}+2\delta_{j_0+2\in(J-1)^{\ss}}\\
        &\hspace{-0.5cm}=2i_{j_0+2}-1+\delta_{j_0+1\in J'}+\delta_{j_0+2\in(J-1)^{\ss}}+1-\delta_{j_0+2\notin(J-1)^{\ss}}\\
        &\hspace{-0.5cm}=2i_{j_0+2}+\delta_{j_0+2\in(J-1)^{\ss}}-\delta_{j_0+2\in J\Delta(J-1)^{\ss}}+\delta_{j_0+1\in J'}\\
        &\hspace{-0.5cm}=(-1)^{\delta_{j_0+3\notin J}}m_{j_0+2},
    \end{align*}
    hence $m_{j_0+2}=m'_{j_0+2}$.
    
    (ii). We prove the case $j=j_0+2$ and $f\geq2$, the other cases being similar and simpler. We also assume that $j_0+3\in J$, the case $j_0+3\notin J$ being similar. Then using (\ref{General Eq sJ}), we have
    \begin{align*}
        2i'_{j_0+2}+t^{J\setminus\set{j_0+2}}(J'')_{j_0+2}&=2i'_{j_0+2}+p-1-s^{J\setminus\set{j_0+2}}_{j_0+2}+\delta_{j_0+1\in J''}\\
        &\hspace{-0.5cm}=2\bbra{i_{j_0+2}-\delta_{j_0+1\notin J'}+\delta_{j_0+2\in J_{\rhobar}}}+p-1-(p-2-r_{j_0+2})+\delta_{j_0+1\notin J'}\\
        &\hspace{-0.5cm}=2i_{j_0+2}-\delta_{j_0+1\notin J'}+p-1-\bbra{p-1-r_{j_0+2}-2\delta_{j_0+2\in J_{\rhobar}}}+1\\
        &\hspace{-0.5cm}=2i_{j_0+2}+p-1-s^J_{j_0+2}+\delta_{j_0+1\in J'}\\
        &\hspace{-0.5cm}=2i_{j_0+2}+t^J(J')_{j_0+2}.
    \end{align*}

    (iii). By (\ref{General Eq relation 2 tJ appendix}) we have $c_j=c'_j$ for $j\neq j_0+1$, so it remains to prove that $c_{j_0+1}=c'_{j_0+1}$. We assume that $j_0+2\in(J-1)^{\ss}$, the case $j_0+2\notin(J-1)^{\ss}$ being similar. Then using (\ref{General Eq sJ}) and (\ref{General Eq relation 2 tJ appendix}) we have
    \begin{equation*}
    \begin{aligned}
        c'_{j_0+1}&=p\bbra{i_{j_0+2}-\delta_{j_0+1\notin J'}+1}+(p-2-r_{j_0+1})+0-\delta_{j_0+1\in J'}(p-1-r_{j_0+1})-\delta_{j_0+1\notin J}\\
        &=p\bbra{i_{j_0+2}-\delta_{j_0+1\notin J'}+1}-1+\delta_{j_0+1\notin J'}(p-1-r_{j_0+1})-\delta_{j_0+1\notin J}\\
        &=pi_{j_0+2}+0+(p-1)-\delta_{j_0+1\notin J'}(r_{j_0+1}+1)-\delta_{j_0+1\notin J}=c_{j_0+1}.
    \end{aligned}
    \end{equation*}

    (iv). If $j\in J'$, then by the definition of $c_j$ and using $i_{j+1}\geq0$, we have 
    \begin{equation*}
        c_j\geq\min\sset{s^{(J-1)^{\ss}}_j,p-1}-1\geq0,
    \end{equation*}
    where the last inequality follows from (\ref{General Eq bound s}).
    
    If $j\notin J'$, then the assumption $2i_{j+1}-\delta_{j+1\in J\Delta(J-1)^{\ss}}+\delta_{j\in J'}\geq0$ implies that either $i_{j+1}\geq1$ or $j+1\notin J\Delta(J-1)^{\ss}$. By the definition of $c_j$ and using $i_{j+1}\geq0$, we have if $j\neq j_0+1$
    \begin{equation*}
        c_j\geq\min\bigset{p,p-1}-\bigbra{2i_j+t^J(J')_j}\geq0,
    \end{equation*}
    where the last inequality follows from (\ref{General Eq bound tJ}) and $\un{i}\leq\un{f}-\un{e}^{J^{\sh}}$. By the definition of $c_{j_0+1}$ and using $i_{j_0+1}=0$ (hence $j_0+1\notin J\Delta(J-1)^{\ss}$) and (\ref{General Eq relation 2 tJ appendix}), we have $c_{j_0+1}\geq(p-1)-(r_{j_0+1}+1)-1\geq0$, where the last inequality follows from (\ref{General Eq genericity}).

    By the definition of $\un{c}$ and since $\un{c}\geq\un{0}$, we have
    \begin{align*}
        \un{Y}^{\un{c}}B_1&=\un{Y}^{p\delta(\un{i})}\bbbra{\scalebox{1}{$\prod$}_{j+1\in J\Delta(J-1)^{\ss}}Y_j^{s^{(J-1)^{\ss}}_j}\scalebox{1}{$\prod$}_{j+1\notin J\Delta(J-1)^{\ss}}Y_j^{p-1}}\smat{p&0\\0&1}\bbra{\un{Y}^{-\un{i}}v_J}\\
        &=\bbbra{\scalebox{1}{$\prod$}_{j+1\in J\Delta(J-1)^{\ss}}Y_j^{s^{(J-1)^{\ss}}_j}\scalebox{1}{$\prod$}_{j+1\notin J\Delta(J-1)^{\ss}}Y_j^{p-1}}\smat{p&0\\0&1}v_J=\mu_{J,(J-1)^{\ss}}v_{(J-1)^{\ss}},
    \end{align*}
    where the second equality follows from Lemma \ref{General Lem Yj}(i) and the last equality follows from Proposition \ref{General Prop vector simple} applied to $J$. Similarly, we have (recall that $\bra{(J\setminus\set{j_0+2})-1}^{\ss}=(J-1)^{\ss}$)
    \begin{align*}
        \un{Y}^{\un{c}}B_2&=\un{Y}^{\un{c}'}B_2=\frac{\mu_{J,(J-1)^{\ss}}}{\mu_{J\setminus\set{j_0+2},(J-1)^{\ss}}}\bbbra{\scalebox{1}{$\prod$}_{j+1\in J'''}Y_j^{s^{(J-1)^{\ss}}_j}\scalebox{1}{$\prod$}_{j+1\notin J'''}Y_j^{p-1}}\smat{p&0\\0&1}v_{J\setminus\set{j_0+2}}\\
        &=\mu_{J,(J-1)^{\ss}}v_{(J-1)^{\ss}},
    \end{align*}
    where $J'''\eqdef(J\setminus\set{j_0+2})\Delta(J-1)^{\ss}$ and the last equality follows from Proposition \ref{General Prop vector simple} applied to $J\setminus\set{j_0+2}$. In particular, we deduce from Lemma \ref{General Lem Yj}(ii) that $B_1$ and $B_2$ have the same $H$-eigencharacter.
\end{proof}

\begin{lemma}\label{General Lem r and c}
    \begin{enumerate}
    \item
    For $J\subseteq\cJ$, we have $\chi'_J\alpha^{\un{r}^J}=\chi_{(\un{r},\un{0})}$, where $\chi'_J\eqdef\chi_J\alpha^{\un{e}^{J^{\sh}}}$ (see (\ref{General Eq rJ}) for $\un{r}^J$ and \S\ref{General Sec sw} for $\chi_J$, $\chi_{(\un{r},\un{0})}$, $\alpha^{\un{i}}$ and $\un{e}^{J^{\sh}}$).
    \item 
    For $J_1,J_2\subseteq\cJ$ such that $J_1\cap J_2=\emptyset$, we have $\un{r}^{J_1\cup J_2}=\un{r}^{J_1}+\un{r}^{J_2}$ (see (\ref{General Eq rJ}) for $\un{r}^J$).
    \item 
    For $J\subseteq\cJ$, we have $\alpha^{\un{c}^J}=\alpha^{\un{r}^{J+1}-\un{r}^J}$ (see (\ref{General Eq cJ}) for $\un{c}^J$).
    \item 
    Let $J'\subseteq J\subseteq\cJ$ and $J''\eqdef J'\Delta(J-1)$. Write $\un{c}\eqdef p\un{e}^{J'\cap(J-1)}+\un{c}^{J'}-\un{f}-\un{r}^{J\setminus J'}$ and let $\un{\delta}\in\set{0,1}^f$. Then for all $j\in\cJ$ we have (see (\ref{General Eq tJJ'}) for $\un{t}^J(J')$)
    \begin{equation*}
        c_j-\delta_j\geq\delta_{j\notin J''}\bbbra{2\bigbra{\delta_{j\in(J'+1)\cap J}-\delta_{j\in(J'+1)^{\sh}}}+t^{J'+1}(J'')_j}+\delta_{J''=\emptyset}.
    \end{equation*}
    \item 
    Let $J'\subseteq J\subseteq\cJ$. Then for all $j\in\cJ$ we have
    \begin{equation*}
        c^{J'}_j-r^{J\setminus J'}_j=c_j^J+\delta_{j\in J\setminus J'}(p-1-r_j).
    \end{equation*}
    \end{enumerate}
\end{lemma}

\begin{proof}
    (i). By definition, we have $\un{t}^J=\un{r}^J+\un{e}^{J^{\sh}}$ for $J\subseteq\cJ$ (see (\ref{General Eq tJ}) for $\un{t}^J$). Hence it suffices to show that $\chi_J\alpha^{\un{t}^J}=\chi_{(\un{r},\un{0})}$. By definition, we have $\chi_J\alpha^{\un{t}^J}=\chi_{\lambda}$ with (see \S\ref{General Sec sw} for $\chi_{\lambda}$ and $\un{s}^J$)
    \begin{equation*}
        \lambda=\lambda_J+\alpha^{\un{t}^J}=(\un{s}^J+\un{t}^J,\un{t}^J)+(\un{t}^J,-\un{t}^J)=(\un{s}^J+2\un{t}^J,\un{0}).
    \end{equation*}
    Since $\sigma_J=(\un{s}^J)\otimes\det^{\un{t}^J}$ (see \S\ref{General Sec sw}) has the same central character as $\sigma_{\emptyset}=(\un{r})$, we deduce that $a^{\un{s}^J+2\un{t}^J}=a^{\un{r}}$ for all $a\in\Fq$, which completes the proof.

    (ii). 
    This follows immediately from \eqref{General Eq Lem rJ}.

    (iii). By \eqref{General Eq Lem rJ} and \eqref{General Eq Lem cJ} we have
    \begin{align*}
        c^J_j+r^J_j-r^{J+1}_j&=
    \begin{aligned}[t]
        &\bigbra{\delta_{j\notin J}(p-1-r_j)+\delta_{j+1\notin J}(r_j+1)-\delta_{j\notin J}}\\
        &\hspace{3cm}+\bigbra{\delta_{j+1\in J}(r_j+1)-\delta_{j\in J}}-\bigbra{\delta_{j\in J}(r_j+1)-\delta_{j-1\in J}}
    \end{aligned}\\
    &=
    \begin{aligned}[t]
        &p\delta_{j\notin J}-\bigbra{\delta_{j\notin J}+\delta_{j\in J}}(r_j+1)\\
        &\hspace{3cm}+\bigbra{\delta_{j+1\notin J}+\delta_{j+1\in J}}(r_j+1)-\bigbra{\delta_{j\notin J}+\delta_{j\in J}}+\delta_{j-1\in J}
    \end{aligned}\\
    &=p\delta_{j\notin J}-1+\delta_{j-1\in J}=p\delta_{j\notin J}-\delta_{j-1\notin J}.
    \end{align*}
    Hence we have 
    \begin{equation*}
        \alpha^{\un{c}^J+\un{r}^J-\un{r}^{J+1}}=\prod\limits_{j\notin J}\alpha_{j}^p\prod\limits_{j-1\notin J}\alpha_j\inv=\prod\limits_{j\notin J}\alpha_{j+1}\prod\limits_{j-1\notin J}\alpha_j\inv=1,
    \end{equation*}
    which proves (iii).

    (iv). We assume that $j\notin J''$, the case $j\in J''$ being similar and simpler. By definition we have
    \begin{align*}
        &2\bigbra{\delta_{j\in(J'+1)\cap J}-\delta_{j\in(J'+1)^{\sh}}}+\bigbra{t^{J'+1}(J'')_j+\delta_{J''=\emptyset}}+\delta_j\\
        &\hspace{1.5cm}\leq 2\bigbra{\delta_{j\in(J'+1)\cap J}-\delta_{j\in(J'+1)^{\sh}}}+\bigbra{p-1-s^{J'+1}_j+1}+1\\
        &\hspace{1.5cm}\leq2\bigbra{\delta_{j\in(J'+1)\cap J}-\delta_{j\in(J'+1)^{\sh}}}+p-1-\bigbra{2(f-\delta_{j\in(J'+1)^{\sh}})+1+\delta_{f=1}}+1+1\\
        &\hspace{1.5cm}=p-2f+2\delta_{j\in(J'+1)\cap J}-\delta_{f=1}\leq p-2f+2-\delta_{f=1}\leq p-f,
    \end{align*}
    where the second inequality follows from (\ref{General Eq bound s}). Since $j\notin J''$, we have either $j\in J'\cap(J-1)$, or $j\notin J'$ and $j\notin J-1$. We give the proof when $j\in J'\cap(J-1)$, the other case being similar. By the definition of $\un{c}$, (\ref{General Eq rJ}) and (\ref{General Eq cJ}) we have
    \begin{equation*}
    \begin{aligned}
        c_j&=
    \begin{cases}
        p+0-f-0&\text{if}~j+1\in J'\\
        p+(r_j+1)-f-(r_j+1)&\text{if}~j+1\notin J'
    \end{cases}\\
        &=p-f,
    \end{aligned}
    \end{equation*}
    which proves (iv).

    (v). By \eqref{General Eq Lem rJ} and \eqref{General Eq Lem cJ} we have
    \begin{align*}
        c^{J'}_j-r^{J\setminus J'}_j-c_j^J
    &\begin{aligned}[t]
        &=\bigbra{\delta_{j\notin J'}(p-1-r_j)+\delta_{j+1\notin J'}(r_j+1)-\delta_{j\notin J'}}-\bigbra{\delta_{j+1\in J\setminus J'}(r_j+1)-\delta_{j\in J\setminus J'}}\\
        &\hspace{1.5cm}-\bigbra{\delta_{j\notin J}(p-1-r_j)+\delta_{j+1\notin J}(r_j+1)-\delta_{j\notin J}}
    \end{aligned}\\
    &\begin{aligned}[t]
        &=\bigbra{\delta_{j\notin J'}-\delta_{j\notin J}}(p-1-r_j)+\bigbra{\delta_{j+1\notin J'}-\delta_{j+1\in J\setminus J'}-\delta_{j+1\notin J}}(r_j+1)\\
        &\hspace{1.5cm}-\bigbra{\delta_{j\notin J'}-\delta_{j\in J\setminus J'}-\delta_{j\notin J}}
    \end{aligned}\\
        &=\delta_{j\in J\setminus J'}(p-1-r_j).
    \end{align*}
    This proves (v).
\end{proof}


\bibliography{1}
\bibliographystyle{alpha}

\end{document}